\documentclass{article}
\usepackage{arxiv}
\usepackage{setspace,amsmath,amssymb,amsthm,color}

\usepackage[backend=biber, 
    language=english,%
    style=authoryear,
	  citestyle=authoryear,
    uniquename=init,
    giveninits,
    dashed=false, 
    sorting=nyt, 
    maxbibnames=5, 
    maxcitenames=1, 
    bibencoding=utf8,
    doi=false,
    isbn=false,
    url=false,
    refsegment=section,
    defernumbers=true
    ]{biblatex}
\AtEveryBibitem{%
  \clearfield{pages}%
  \clearfield{month}%
  \clearlist{language}%
}

\newtheorem{Proposition}{Proposition}[section]

\newtheorem{Lemma}[Proposition]{Lemma}
\newtheorem{Corollary}[Proposition]{Corollary}

\theoremstyle{definition}
\newtheorem{Definition}[Proposition]{Definition}
\newtheorem{Remark}[Proposition]{Remark}

\usepackage{cleveref}
\crefname{lstlisting}{listing}{listings}
\Crefname{lstlisting}{Listing}{Listings}

\crefname{equation}{equation}{equations}
\crefname{figure}{figure}{figures}
\crefname{Definition}{definition}{definitions}
\crefname{Proposition}{proposition}{propositions}
\crefformat{footnote}{#2\footnotemark[#1]#3}

\usepackage{amsfonts}
\usepackage{relsize}
\usepackage{paralist}
\usepackage{graphicx}
\usepackage{array, tabularx, framed}
\usepackage{longtable}
\usepackage[margin=40pt, labelfont=bf, format=plain, hypcap=false, labelsep= space]{caption}
\usepackage{tikz}
\usepackage{pgfplots}
\usepgfplotslibrary{external}
\usetikzlibrary{patterns}
\usetikzlibrary{external}
\usetikzlibrary{shapes}
\usetikzlibrary{plotmarks}
\usetikzlibrary{trees}
\tikzset{%
  external/system call={pdflatex \tikzexternalcheckshellescape --halt-on-error --interaction=batchmode --output-directory=./build/ --jobname "\image" "\texsource"},
  /pgf/images/include external/.code={%
    \includegraphics{build/#1}%
  },
}

\definecolor{lightgrey}{rgb}{0.93,0.93,0.93}
\definecolor{verylightgrey}{rgb}{0.965,0.965,0.965}
\definecolor{lightblue}{rgb}{0,0.4,0.93}
\definecolor{darkpurple}{rgb}{0.53,0,0.83}
      \definecolor{COL1}{HTML}{525564}
      \definecolor{COL2}{HTML}{74828F}
      \definecolor{COL3}{HTML}{96C0CE}
      \definecolor{COL4}{HTML}{BEB9B5}
      \definecolor{COL5}{HTML}{C25B56}
      \definecolor{COL6}{HTML}{FEF6EB}
      \definecolor{COL7}{HTML}{D3D3D3}
      \definecolor{COL8}{HTML}{808080}

      \makeatletter 
        \pgfdeclarepatternformonly[\LineSpace]{my north east lines}{\pgfqpoint{-5pt}{-5pt}}{\pgfqpoint{\LineSpace}{\LineSpace}}{\pgfqpoint{\LineSpace}{\LineSpace}}%
        {
          \pgfsetcolor{\tikz@pattern@color} 
          \pgfsetlinewidth{0.4pt}
          \pgfpathmoveto{\pgfqpoint{0pt}{0pt}}
          \pgfpathlineto{\pgfqpoint{\LineSpace + 0.1pt}{-\LineSpace - 0.1pt}}
          \pgfusepath{stroke}
        }

        \pgfdeclarepatternformonly[\LineSpace]{my north west lines}{\pgfqpoint{-1pt}{-1pt}}{\pgfqpoint{\LineSpace}{\LineSpace}}{\pgfqpoint{\LineSpace}{\LineSpace}}%
        {
          \pgfsetcolor{\tikz@pattern@color} 
          \pgfsetlinewidth{0.4pt}
          \pgfpathmoveto{\pgfqpoint{0pt}{0pt}}
          \pgfpathlineto{\pgfqpoint{\LineSpace + 0.1pt}{\LineSpace + 0.1pt}}
          \pgfusepath{stroke}
        }
        \makeatother 
        \newdimen\LineSpace
        \tikzset{
           line space/.code={\LineSpace=#1}, %
           line space=5pt%
        }

\usepackage{listings}

\makeatletter%
 \def\lst@PlaceNumber{\makebox[\dimexpr 1em+\lst@numbersep][r]{\normalfont
 \lst@numberstyle{\thelstnumber}}\hspace{1em}}%
\makeatother%

\makeatletter%
\def\lst@PlaceNumber{\makebox[\dimexpr 1em+\lst@numbersep][r]{\normalfont
\lst@numberstyle{\thelstnumber}}\hspace{1em}}%
\makeatother%

\usepackage{algorithm}
\usepackage[noend]{algpseudocode}
\algdef{SE}[DOWHILE]{Do}{doWhile}{\algorithmicdo}[1]{\algorithmicwhile\ #1}%

\makeatletter
\def\BState{\State\hskip-\ALG@thistlm}
\makeatother

  \DeclareMathOperator*{\divergence}{div}

  \DeclareMathOperator*{\image}{im}
  \newcommand{\constant}{\ensuremath{c}}
  \newcommand{\robin}{\ensuremath{\beta_r}}
  \newcommand{\consistencyerror}{\ensuremath{\varepsilon}}

  \newcommand{\boundary}{\ensuremath{\Gamma}}
  \newcommand{\boundarydisplacements}{\ensuremath{\mbf{{\bar u}}}}

  \newcommand{\continuousfunctions}{\ensuremath{C}}

  \newcommand{\damageRHS}{\ensuremath{r}}
  \newcommand{\damage}{\ensuremath{d}}
  \newcommand{\damageconstant}{\ensuremath{\omega}}
  
  \newcommand{\damagefunctions}{\ensuremath{\mathcal{D}}}
  \newcommand{\damageprocess}{\ensuremath{g}}
  \newcommand{\damageprocesses}{\ensuremath{\mathcal{G}}}

  \newcommand{\dimension}{\ensuremath{N}}
  
  \newcommand{\displacements}{\ensuremath{\mbf{u}}}
  \newcommand{\displacementdomain}[1]{%
  \ifnum\pdfstrcmp{#1}{}=0 %
    \ensuremath{\mathcal{V}} %
  \else %
    \ensuremath{\mathcal{V}_{#1}} %
  \fi %
  }
  \newcommand{\integrablefunctions}{\ensuremath{L}}
  
  \newcommand{\forces}{\ensuremath{\mbf{f}}}

    \newcommand{\measure}[1]{%
    \ifnum\pdfstrcmp{#1}{H}=0 %
      \ensuremath{\mu}%
    \else %
      \ensuremath{\lambda} %
    \fi %
  }
  \newcommand{\mollifiedgradient}{\ensuremath{\nabla^\mu}}

  \newcommand{\normals}{\ensuremath{\mbf{\nu}}}

  \newcommand{\sobolevfunctions}{\ensuremath{W}}
  \newcommand{\nemytskiioperator}{\ensuremath{G}}

  \newcommand{\spatialhilbertfunctions}{\ensuremath{H}}
  \newcommand{\spatialfunctions}{\ensuremath{V}}
  \newcommand{\spatialdomain}{\ensuremath{\Omega}}

  \newcommand{\spatialvariable}{\ensuremath{x}}
  \newcommand{\spatialvariables}{\ensuremath{\mbf{x}}}

  \newcommand{\strains}{\ensuremath{\mbf{\varepsilon}}}
  
  \newcommand{\stresses}{\ensuremath{\mbf{\sigma}}}
  \newcommand{\stressforce}{\ensuremath{\tau}}
  \newcommand{\stressforces}{\ensuremath{\mbf{\tau}}}

  \newcommand{\timevariable}{\ensuremath{t}}
  \newcommand{\timedomain}{\ensuremath{S}}
  \newcommand{\timedomainmax}{\ensuremath{T}}

  \newcommand{\discretizationInTime}{\ensuremath{s}}
  \newcommand{\discretizationInSpace}{\ensuremath{h}}
  \newcommand{\trace}{\ensuremath{\tn{tr}\,}}

  \newcommand{\tractiondrivendisplacements}{\ensuremath{\mbf{v}}}
  \newcommand{\elasticitytensor}{\ensuremath{\mbf{\mathbb{E}}\ \!}}

  \newcommand{\transpose}{\ensuremath{\intercal}}

  \newcommand{\pd}[3]{%
    \ifnum\pdfstrcmp{#3}{1}=0 %
      \ensuremath{\frac{\partial #1}{\partial #2}}%
    \else %
      \ensuremath{\frac{\partial^#3 #1}{\partial #2^#3}} %
    \fi %
  }
  \newcommand{\mbb}[1]{\ensuremath{\mathbb{#1}}}
  \newcommand{\td}[3]{%
    \ifnum\pdfstrcmp{#3}{1}=0 %
      \ensuremath{\frac{d #1}{d #2}}%
    \else %
      \ensuremath{\frac{d^#3 #1}{d #2^#3}} %
    \fi %
  }
  \newcommand{\tn}[1]{\textnormal{#1}}
  
  \newcommand{\mbf}[1]{\ensuremath{\boldsymbol{#1}}}
  
  \newcommand{\fa}{\forall\ \!}

\renewcommand{\phi}{\varphi}
\renewcommand{\epsilon}{\varepsilon}

\begin{document}

\newcolumntype{Y}{>{\centering\arraybackslash}X}
\newcommand{\Cpp}{\texttt{c++}}
\newcommand{\Doxygen}{Doxygen\textsuperscript{\textcopyright}}
\newcommand{\FEniCS}{FEniCS\textsuperscript{\textcopyright}}
\newcommand{\ParaView}{ParaView\textsuperscript{\textcopyright}}
\newcommand{\Gmsh}{Gmsh\textsuperscript{\textcopyright}}
\label{page:t}
\thispagestyle{plain}
\title{{I}DENTIFYING {P}ROCESSES {G}OVERNING {D}AMAGE {E}VOLUTION IN {Q}UASI-{S}TATIC {E}LASTICITY\\{P}ART 2 - {N}UMERICAL {S}IMULATIONS}


\renewcommand{\shorttitle}{Damage Process Identification in Quasi-Static Elasticity}
\renewcommand{\headeright}{}
\author{ Simon Gr\"utzner \\
	Center for Industrial Mathematics\\
	University of Bremen\\
	Germany \\
	\texttt{simon.gruetzner@uni-bremen.de} \\
	\And
	Adrian Muntean\\
	Department of Mathematics and Computer Science\\
	University of Karlstad\\
	Sweden \\
	\texttt{adrian.muntean@kau.se}\\
}

\maketitle
\noindent
{\bf Abstract.}
We investigate numerically a quasi-static elasticity system of Kachanov-type.
To do so we propose an Euler time discretization combined with a suitable finite elements scheme (FEM) to handle the discretization is space.
We use ODE-type arguments to prove the consistency of the scheme as well as its convergence rate.
We rely on the computational platform \FEniCS\ to perform the FEM discretizations in space needed to compute the model output.
The simulation results show a good agreement with both the physics of the problem and with our previous qualitative mathematical analysis results obtained for precisely the same problem setting.
Furthermore, our implementation recovers nicely the theoretically expected convergence rate.
This is a preliminary study preparing the framework for the rigorous numerical identification of the damage process in Kachanov-type models.

\section{Introduction}\label{sec:1}
Due to recent technological developments, there is a rapidly increasing  interest in getting grip on computable quantitative indicators of mechanical damage in real materials.
The difficulty of the topic is considerable if one has in mind that damage cannot be directly measured and, on top of this,  materials are increasingly complex due to their composite structure and built-in functionalities.
Having in mind the  durability (i.e. the large time behavior) of materials and involved maintenance costs, significant efforts are paid continuously to identify damage properties.
Seen from this perspective, our study wants to contribute using the classical framework of a Kachanov-type model (see e.g. \cite{Ka86}) which is in line with the series of sophisticated damage models developed by the engineers in the last approximately 20 years.\\[2ex]
Besides any practical consideration, our interest in the model lies in its interesting mathematical structure (a nonlinear coupling between elliptic and ordinary differential equations).
We refer the reader to a couple of relevant computational approaches like \cite{DA12,NP98,LP04,HL16,OKM19,LW21} as well as to \cite{LiLi05} for a recent review.
In \cite{GM18}, we presented a gentle introduction to the topic and we listed our philosophy concerning damage identification via Kachanov-type models.
The mathematical analysis of the forward problem, as formulated in Section \ref{sec:2}, and the main ingredients for handling the inverse problem have been considered in \cite{GM21}.
We are considering here only the elliptic case, but it is worth noting that a discussion of a hyperbolic scenario, very much inspired by our problem setting, is done in \cite{GeGr20}.
We stress the fact that handling inverse questions for our problem is particularly challenging because of the nonlinearity arising in the structure of the differential operator (a typical feature of the Kachanov setting).\\[2ex]
The main objective of this paper is to illustrate numerically the behavior of our model in physically relevant parameter ranges.
We use an Euler-type discretization in time combined with a FEM discretization in space with controlled consistency and convergence properties.
Our numerical results confirm the theoretically expected behavior of solutions. \\[2ex]
The structure of the paper is organized in the following fashion: Section \ref{sec:2} contains the list of our model equations.
As our problem is mathematically well-posed (cf. Theorem 3.1 in \cite{GM21}), we perform a basic numerical analysis of the employed time discretization in Section \ref{sec:numerical_analysis}.
In Section \ref{Sec:Corners} we discuss the possibility that singularities arise due to an eventual mismatch of boundary conditions and indicate practical options to avoid the occurrence of such unwanted effects.
After discussing implementation matters in \FEniCS, we illustrate in Section \ref{sec:3} the typical numerical output our model can deliver.
We refer to \cite{LoMa2012a} for an extensive introduction to this open-source \Cpp\ library.
Note that all meshes were generated by \Gmsh, all images showing spatial distributions of the solution were generated by \ParaView.
It is worth noting that the convergence rates of our numerical scheme matches the theoretical {\em a priori} convergence rates.
Using numerical simulations, we point out  concrete scenarios when corner-type singularities may arise and indicate possible remedies.
The paper closes with a brief conclusion section.

\section{Forward Problem. Setting of model equations.}\label{sec:2}
Our aim is to simulate numerically the following Kachanov-type system, which we will refer to as the \emph{quasi-static problem of linear elasticity for damaged, isotropic continua}.
The reader can find in \cite{GM21} mathematical analysis investigations of the weak solvability of this problem as well as  some of the features needed for the investigation of inverse-type questions.
Within this framework, the focus is on a more pragmatic side.
We wish to provide suitable discretizations for our model which are easily implementable in \FEniCS.
Our system reads as follows.
\begin{framed}
\begin{Definition}[Forward Problem]
  \label{Definition:Forward Problem}
  \begin{subequations}\label{Equations:Forward Problem}
    \begin{alignat}{2}\label{equation:effective stresses}
      \stresses(\timevariable,\spatialvariables) %
      &= \big(1-\damage(\timevariable,\spatialvariables)\big)\Big(\lambda(\spatialvariables)\,\trace\!\big(\strains\big(\displacements(\timevariable,\spatialvariables)\big)\big)I+2\mu(\spatialvariables)\strains\big(\displacements(\timevariable,\spatialvariables)\big)\big)\Big), &\quad \text{in }&S\times\spatialdomain, \\ \label{equation:momentum equation}
      -\divergence{\big(\stresses(\timevariable,\spatialvariables)\big)} %
      &= \forces(\timevariable,\spatialvariables) ,  &\quad  \text{in }&\timedomain\times\spatialdomain, \\ \label{damageevolution}
      \damage'(\timevariable,\spatialvariables)&= \big(1-\damage(\timevariable,\spatialvariables)\big)^{-\alpha}\damageprocess(\timevariable,\spatialvariable,\mollifiedgradient\displacements(\timevariable,\spatialvariables)), &\quad\text{in } & \timedomain\times\spatialdomain \\
      \displacements(\timevariable,\spatialvariables) 	&= \boldsymbol{0},  &\quad \text{on }&\timedomain\times\boundary_0,	 \\
      \stresses(\timevariable,\spatialvariables)\normals(\timevariable,\spatialvariables)	&= \mbf{0},  &\quad\text{on }&\timedomain\times\boundary_1, \\
      \stresses(\timevariable,\spatialvariables)\normals(\timevariable,\spatialvariables) 	&= \stressforces(\timevariable,\spatialvariables), &\quad\text{on } &\timedomain\times\boundary_2, \\
      \displacements(0,\spatialvariables) &= \displacements_0(\spatialvariables), &\quad\text{in } &\spatialdomain,\\ \label{initial_damage}
      \damage(0,\spatialvariables) &= \damage_0(\spatialvariables), &\quad\text{in } &\spatialdomain.
    \end{alignat}
  \end{subequations}
\end{Definition}
\end{framed}
Besides the time interval $\timedomain:=(0,\timedomainmax)\subset\mathbb{R}$ with $0<\timedomainmax<\infty$ and the spatial domain $\spatialdomain\subset\mathbb{R}^2$, in \Cref{Definition:Forward Problem}, $\stresses$ denotes the Cauchy stresses in the reference configuration expressed via the \emph{effective stresses} derived in the previous section, $\damage$ the damage variable, $\lambda,\mu$ the Lam\'e coefficients, $\strains$ the linearized Cauchy-Green strain tensor, $\forces$ the volumetric forces, $\damageprocess$ an expression for the effective stresses which is in particular dependent on the mollified gradient of the displacments $\mollifiedgradient\displacements$.
The objects $\normals$ and $\stressforces$ are representing the outer normal on $\partial\spatialdomain$ and $\stressforces$ the stresses acting on $\boundary_2$, a part of the boundary of $\spatialdomain$, respectively.
\subsection{Analytical setting}
The functional analytical setting has been introduced and well investigated in part one of this work (see \cite{GM21}).
To keep presentation concise, we will only repeat the basic functional spaces from that part.\\[2ex]
$\sobolevfunctions^{k,p}(\spatialdomain)$ denotes the specified Sobolev space of $k$-times weakly differentiable p-integrable functions.
We denote by $\spatialhilbertfunctions^k(\spatialdomain):=\sobolevfunctions^{k,2}(\spatialdomain)$ and $\integrablefunctions^p(\spatialdomain):=\sobolevfunctions^{0,p}(\spatialdomain)$ square integrable Sobolev and Lebesgue spaces, respectively.
We introduce basic function spaces for damage evolution and start with
\begin{equation}\label{eq:damage_functions}
    \damagefunctions:=\left\{%
    \damage\in\sobolevfunctions^{1,\infty}\!\left(\timedomain;\, \integrablefunctions^\infty(\spatialdomain)\right);\ %
      0\le\damage(\timevariable,\spatialvariables)\le\damageconstant_1 %
    \text{ a.e. in } \timedomain\times\spatialdomain \right\}
\end{equation}
where $\damageconstant_1\in\mbb{R}$ denotes a fixed non-negative constant such that $0\le\damageconstant_1 <1$ holds. We name its elements \emph{damage functions}.
We refer to elements of %
\begin{equation}\label{eq:initial_damage}
  \damagefunctions_0:= %
  \{\damage_0\in L^\infty(\spatialdomain);\, %
  0\le\damage_0(\spatialvariables)\le\damageconstant_0 %
  \text{ a.e. in }\spatialdomain\}
\end{equation}
as \emph{initial damage} and expect the constant $\damageconstant_0$ to suffice $0\le\damageconstant_0\le\damageconstant_1$.
This is to allow for non-zero initial damage as well.
Let $\overline{Y}:=\overline{\mathbb{B}}(0;\overline{y})\subset\mbb{R}^{\dimension^2}$ be the closed ball of radius $\bar y>0$ and center placed in $0$.
Furthermore, we introduce the set of admissible \emph{damage processes}
\begin{equation}\label{eq:damage_sources}
  \begin{aligned}
    \damageprocesses:=\Big\{\damageprocess %
    &\in  \integrablefunctions^\infty\Big(\timedomain; \integrablefunctions^\infty\big(\spatialdomain;\continuousfunctions^{1,1}(\overline{Y})\big)\Big); \\ %
    &\fa\mbf{y}\in\overline{Y}:\ %
    0\le\damageprocess(\cdot,\cdot,\mbf{y})%
    \le T^{-1}(\damageconstant_1 %
    - \damageconstant_0)(1-\damageconstant_1)^\alpha %
    \text{ a.e. in }\timedomain\times\spatialdomain\
    \Big\},
  \end{aligned}
\end{equation}
where $\alpha\ge 1$ is some fixed constant and $\continuousfunctions^{m,\lambda}(\overline{Y})$ indicates the specified H\"older space of $m$-times continuously differentiable functions $\overline{Y}\subset\mathbb{R}^{\dimension^2}\to\mathbb{R}$ satisfying the H\"older condition with exponent $\lambda$.

\subsection{Discretization}
Starting off from equations \eqref{Equations:Forward Problem}, we can immediately state our choice of discretization.
Note that the weak formulation for the continuous problem is stated in \cite{GM21}, where also a detailed well-posedness study is offered.
We rely on this specific  contribution to study the quality of our discretization scheme shown in this section.
For the discretization $\timedomain_n$, corresponding  to the time interval  $\timedomain:=(0,\timedomainmax)$ with $0<\timedomainmax<\infty$,  we set
\begin{equation}
  \timedomain_n:=\{\timevariable_n\in\bar{\timedomain};\ \timevariable_{n+1} = \timevariable_{n} + \discretizationInTime,\ n\le N_{\timedomainmax)}\},\quad \discretizationInTime_N:=\timedomainmax/N_{\timedomainmax},\quad N_{\timedomainmax}\in\mathbb{N},
\end{equation}
where $\discretizationInTime\in\mathbb{R}$ denotes the step size in time.
We choose standard continuous Lagrange elements for the spatial discretization in two dimensions, i.e.,
\begin{equation}\label{equation:FE space for displacements}
  \spatialfunctions_h:=\left\{\varphi\in\continuousfunctions^0(\bar{\spatialdomain}_h)^2;\,\forall\,T_h\in\mathcal{T}:\varphi|_{\bar{T_h}}\in\mathcal{P}_1(\bar{T_h}),\boldsymbol{\varphi}=0\text{ on }\boundary_{0,h}\right\},
\end{equation}
where $T_h\in\mathcal{T}$ is assumed to be an element of the regular triangulation $\mathcal{T}$ of $\spatialdomain$ and $\mathcal{P}_1(\bar{T}_h)$ denotes the polynomials over $T_h$ of first order.
Analogously, we introduce the space
\begin{equation}\label{equation:FE space for damage}
  \damagefunctions_{h}:=\left\{\varphi\in\continuousfunctions^0(\bar{\spatialdomain}_h);\, \forall T_h\in\mathcal{T}:\varphi|_{\bar{T_h}}\in\mathcal{P}_1(\bar{T}_h) \right\}
\end{equation}
to find the appropriate FE representation of the damage variable $\damage$ with respect in space.
\begin{framed}
  \begin{Definition}[Discretized Forward Problem]\label{Definition:Discretization}
    For all $n\in\mathbb{N}$ satisfying $0\le n \le N_{\timedomainmax}$, we look for ($\displacements_h(\timevariable_n)$, $\boldsymbol{w}_h(\timevariable_n)$, $\damage_h(\timevariable_n)$) $\in$ $\spatialfunctions_h\times\spatialfunctions_h\times\damagefunctions_h$ such that for all $(\boldsymbol{\varphi}_{\spatialfunctions_h}, \boldsymbol{\varphi}_{\damagefunctions_h})\in \spatialfunctions_h\times \damagefunctions_h$, the expressions
    \begin{subequations}
      \begin{multline} \label{equation:discretized EqOfMo}
        \int_{\spatialdomain_h} \Big(1-\damage_h(\timevariable_n)\Big)\Big(\lambda_h\,\trace\!\big(\strains\big(\displacements_h(\timevariable_n\big)\big)I+2\mu_h\strains\big(\displacements_h(\timevariable_n)\big)\Big):\strains\big(\boldsymbol{\varphi}_{\spatialfunctions_h}\big)\,d\mathcal{L}^2 \\
          = \int_{\spatialdomain_h}\forces(\timevariable_n)\cdot\boldsymbol{\varphi}_{\spatialfunctions_h}\,d\mathcal{L}^2 + \int_{\boundary_{2,h}}\stressforces_h(\timevariable_n)\cdot\boldsymbol{\varphi}_{\spatialfunctions_h}\,d\mathcal{H}^1,
      \end{multline}
      \begin{equation}\label{equation:Projection of Displacementgradient}
        \int_{\spatialdomain_h}\boldsymbol{w}_h\cdot\nabla\boldsymbol{\varphi}_{\spatialfunctions_h}\,d\mathcal{L}^2
          = \int_{\spatialdomain_h}\nabla\displacements_h\cdot\boldsymbol{\varphi}_{\spatialfunctions_h}\,d\mathcal{L}^2,
      \end{equation}
      \begin{gather}
        \timevariable_n =\timevariable_{n-1}+\Delta\timevariable, \\ \label{equation:Explicit_Euler}
        \damage_{\discretizationInSpace}^i(t_n) =\damage_{\discretizationInSpace}^i(\timevariable_{n-1}) + \Big(1-\damage_h^i(\timevariable_{n-1})\Big)^{-\alpha}\damageprocess_h^i\Big(\timevariable_{n-1},\boldsymbol{w}_h^i(\timevariable_{n-1})\Big),\\ \label{equation:damage in FE space}
        \damage_h(\timevariable_n) =\sum_{i=1}^{N_h}\damage_h^i(\timevariable_n)\varphi_{\damagefunctions_h}^i,
      \end{gather}
      hold for all $i=1,\dots,N_h<\infty$ with the discretized initial values
      \begin{align}
        \displacements_h(\timevariable_0) &= \displacements_{h,0}, \\
        \damage_h(\timevariable_0) &= \damage_{h,0}.
      \end{align}
    \end{subequations}
  \end{Definition}
\end{framed}
In \cref{equation:discretized EqOfMo} we start with the FE discretization of the equation of motion.
Since we want to solve the damage equation on every node of the \emph{continuous Lagrange triangulation}, we have to use the $L^2$-Projection in order to compute the values of the displacement gradient on the nodes (see \cref{equation:Projection of Displacementgradient}).
In \cref{equation:Explicit_Euler} we use an explicit Euler scheme to discretize the ODE and employ the computed values to find the FE representation of $\damage_h$ in \cref{equation:damage in FE space}.

\section{Numerical analysis}\label{sec:numerical_analysis}
We aim to show convergence of the numerical scheme to the solution of the continuous model.
Having this scope in mind, we limit our investigation to the study of on an explicit Euler scheme.
Transferring these results to schemes of higher order, like the Runge-Kutta ones, is straightforward.
Note that we chose to keep in our presentation the mollification of the displacements' gradient that we extensively in the mathematical analysis of this problem.
However, in the context of this paper, a natural choice for the mollification could as well be an $\integrablefunctions^2$-projection, as it is already used in the FE scheme.\\[2ex]
\Cref{figure:consistency and convergence} illustrates the general situation.
We take an arbitrary point $\spatialvariables_{}$ of the mesh used to discretize the model in space and look at the damage variable and its continuous evolution $\Phi$, respectively, at the initial and two consecutive time steps of the discretized time interval; see the definition of an evolution in \cite{DeBo2002}.
Compared to a discrete evolution $\Psi$ via the explicit Euler scheme, one ends up making different types of errors.
For numerical schemes aiming to solve these kind of problems, a common approach is to talk about consistency.
We will see that having to deal with a coupled problem introduces additional aspects that have to be taken into account.
In a classical decoupled ODE setting, so-called consistency errors, say $\varepsilon_i$ for $i=1,2$, describe the errors occurring in single time steps, if the correct damage at the previous step is known; compare to Definition 4.3 in \cite{DeBo2002}.
Note that in our case, the correct displacements have to be known, too.
This translates to
\begin{multline}\label{eq:consistency error}
  \varepsilon_i=\Phi(\timevariable_{i-1}, \timevariable_{i},\damageRHS_{\discretizationInSpace},\spatialvariables_{\discretizationInSpace})\damage_{\discretizationInSpace} %
    - \Psi(\timevariable_{i-1},\timevariable_{i},\damageRHS_{\discretizationInSpace},\spatialvariables_{\discretizationInSpace})\damage_{\discretizationInSpace} 
      = \damage_{\discretizationInSpace}(\timevariable_{i},\spatialvariables_{\discretizationInSpace}) - \damage_{\discretizationInSpace}(\timevariable_{i-1},\spatialvariables_{\discretizationInSpace})-\tau_i\damageRHS_{\discretizationInSpace}(\timevariable_{i-1},\damage_{\discretizationInSpace}(\timevariable_{i-1},\spatialvariables_{\discretizationInSpace}))
\end{multline}
for $i=1,2$.
In a subsequent time step, this error is then transported via the discretization scheme, since the right-hand side $f$ is evaluated at an approximation $\damage_{\discretizationInTime}^{i}(\spatialvariables_{\discretizationInSpace})$ rather than the correct damage value $\damage(\timevariable_i,\spatialvariables_{\discretizationInSpace})$.
This leads to an error when calculating the tangent's slope for the next time step and thus to another type of error $\varepsilon_{\discretizationInTime}$, i.e., a \emph{discretization error in time}.
\begin{center}
  \begin{minipage}{.49\linewidth}
    \centering
    \includegraphics{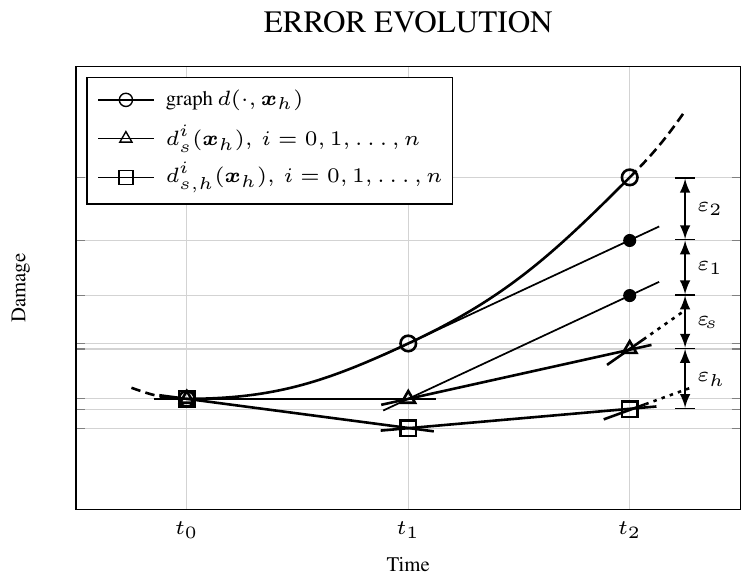}
  \end{minipage}\hfill
  \begin{minipage}{.49\linewidth}
    \centering
    \includegraphics{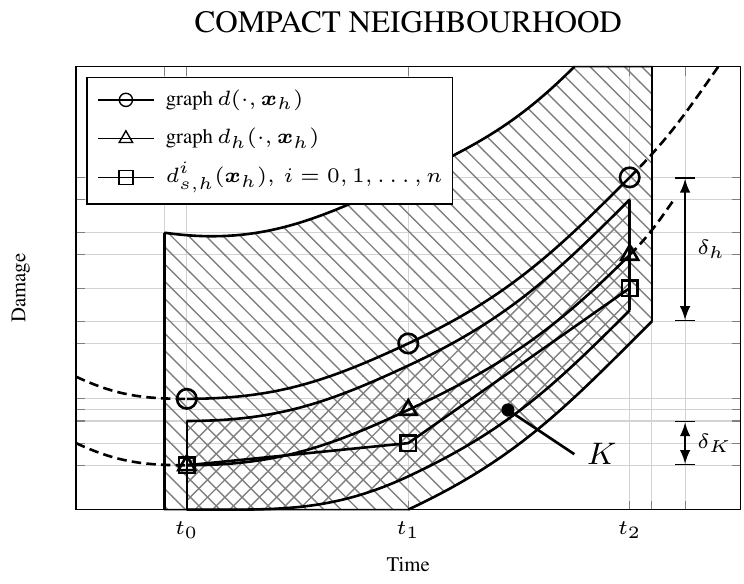}
  \end{minipage}
  \captionof{figure}{Error evolution and overall convergence of numerical scheme.}
  \label{figure:consistency and convergence}
\end{center}
In our case, we face an additional difficulty -- due to the intimate coupling of damage and displacements, we do not have access to the right-hand side $\damageRHS$ but only to its approximation $\damageRHS_{\discretizationInSpace}$ in space.
The fact that $\damageprocess$ depend on the displacements $\displacements$, too, whose solution is approximated via an FEM scheme is the main reason for this.
This causes an additional error $\varepsilon_{\discretizationInSpace}$, which already effects the computation of the first time step.
The situation is shown in the left illustration of \Cref{figure:consistency and convergence}.
This effects the consistency of the numerical scheme in the classical sence.
Figuratively speaking, the approximated right-hand side $r_{\discretizationInSpace}$ causes a deviation from the correct tangent's slope at a specific damage and time.
Therefore, the approximation error from the displacements has to be factored into our investigation. \\[2ex]
The general idea to prove convergence of the numerical scheme is presented in the right subfigure of \Cref{figure:consistency and convergence}.
Here $\damage(\cdot, \spatialvariables_{\discretizationInSpace})$ denotes the exact solution evaluated at some of the time steps and at a specific point in space.
We approximate the right-hand side of the damage evolution equation by utilizing the Finite Element Method we implemented to solve the equation of motion and denote the solution $\damage_{\discretizationInSpace}$
Since we showed in the analysis part of this work that the solution $\damage$ depends Lipschitz-continuously on the right-hand side, we can then directly estimate the approximation error $\|\damage-\damage_{\discretizationInSpace}\|$.
By employing a consistent numerical scheme we solve the discrete damage evolution equuation numerically in time in a second step.
The discrete solution is denoted by $\damage_{\discretizationInTime, \discretizationInSpace}$.
Thus, we can estimate the error between exact and discrete solution by splitting the error into $\|\damage - \damage_{\discretizationInSpace}\|$ and $\|\damage_{\discretizationInSpace} - \damage_{\discretizationInTime, \discretizationInSpace}\|$.
Looking at the right part of \Cref{figure:consistency and convergence} this basically translates into making such small steps in time so that the numerical solution stays within a compact neighborhood $K$ of $\damage_{\discretizationInSpace}$ that completely lies within a neighborhood of the exact solution $\damage$.
%
\subsection{Approximated right-hand side}
We start by formaly introducing the right-hand sides mentioned in the previous section.
Our main statements about them will concern their Lipschitz-in-time property.
\begin{Definition}\label{def:right-hand_sides}
  Let be $\damageprocess\in\damageprocesses\cap \continuousfunctions^{0,1}(\timedomain; \integrablefunctions^\infty (\spatialdomain;\continuousfunctions^{1,1}(\overline{Y})))$.
  We set
  \begin{align}\nonumber
    \damageRHS(\timevariable, \spatialvariables, \damage(\timevariable, \spatialvariables)) &:=\damageprocess(\timevariable, \spatialvariables, \mollifiedgradient O_{E}^{\forces+\stressforces}(\damage(\timevariable,\spatialvariables)))(1-\damage(\timevariable, \spatialvariables))^{-\alpha}, \\
    \damageRHS_{\discretizationInSpace}(\timevariable, \spatialvariables, \damage(\timevariable, \spatialvariables)) &:=\damageprocess(\timevariable, \spatialvariables, \mollifiedgradient O_{E, \discretizationInSpace, \timevariable}^{\forces+\stressforces}(\damage(\timevariable, \spatialvariables)))(1-\damage(\timevariable,\spatialvariables))^{-\alpha}
  \end{align}
  and refer to $\damageRHS$, $\damageRHS_{\discretizationInSpace}$ as the right-hand sides of the damage evolution equation and its approximation, respectively.
  Here the operator $O_{E}^{\forces+\stressforces}\colon\damagefunctions\to\integrablefunctions^{\infty}(\timedomain;\spatialfunctions)$ denotes the solution operator taken from Corollary 3.7 of part 1 (cf. \cite{GM21}) mapping a given damage to the respective displacements.
  We define $O_{E, \discretizationInSpace, \timevariable}^{\forces+\stressforces}\colon\integrablefunctions^{\infty}(\spatialdomain)\to\spatialfunctions_{\discretizationInSpace},\ \damage_{\timevariable}\mapsto \displacements_{\discretizationInSpace, \timevariable}$ mapping a given spatial damage distribution to the respective solution of the discretized equation of motion in space.
\end{Definition}
\begin{Remark}\label{remark:lipschitz_data_EqOfMo}
  Note that time is merely a parameter for the equation of motion in this quasi-static setting.
  Thus choosing data for the equation of motion to be Lipschitz in time suffices to ensure that the solution inherits the property as well.
  For details, see Proposition 3.4 in \cite{GM21} and its proof.
\end{Remark}
\begin{Lemma}\label{Lemma: Lipschitz continuity of right-hand side}
  By choosing data for the equation of motion to be Lipschitz in time, the ODE's right-hand sides from \Cref{def:right-hand_sides} becomes Lipschitz with respect to time.
\end{Lemma}
\begin{proof}
  The claim for $\damageRHS$ follows directly from Theorem 3.12, Proposition 3.28, Lemma 4.2, and Proposition 4.6 of \cite{SiGr15}.
  Regarding $\damageRHS_{\discretizationInSpace}$, we refer to Remark \ref{remark:lipschitz_data_EqOfMo}.
\end{proof}
\begin{Lemma}\label{Lemma:Lipschitz with respect to state variable}
  $\damageRHS, \damageRHS_{\discretizationInSpace}$ are Lipschitz with respect to the state variable $\damage$ for almost all $\spatialvariables\in\spatialdomain$.
\end{Lemma}
\begin{proof}
  We recall that $\damage\in\damagefunctions$ is essentially bounded.
  Also note that $x\mapsto x^{-\alpha}\in\continuousfunctions^{\infty}_0([1-\damageconstant_1,1])$ holds.
  Since $\damageprocess\in\damageprocesses\cap\continuousfunctions^{0,1}(\timedomain; \integrablefunctions^\infty (\spatialdomain;\continuousfunctions^{1,1}(\overline{Y})))$ holds too, the claim holds true.
\end{proof}
We state now a standard auxiliary result concerning the estimation of the spatial errors in the context of the Finite Element Method.
For more on this matter, we refer the reader, for instance, to the standard texts \cite{Br07, PhCi02,DauLi00D,Dz2010}.
Note that we chose continuous Lagrange elements on a regular triangulated domain (see \eqref{equation:FE space for displacements}), i.e., our finite element space is regular enough so that the following result is valid.
\begin{Lemma}\label{lemma:fem_error_estimate}
  Assuming the weak solution $\displacements$ to the equation of motion satisfies $\displacements\in\integrablefunctions^{\infty}(\timedomain;\spatialhilbertfunctions^{k+1}(\spatialdomain))$ for some integer $k\ge 1$ and the data is smooth enough.
  Then it exists a constant $c>0$, which is independent of $\discretizationInSpace$, such that
  \begin{equation}
    \|\displacements(\timevariable) - \displacements_{\discretizationInSpace}(\timevariable)\|_{\spatialhilbertfunctions^1(\spatialdomain)} %
      \le C\discretizationInSpace^k \text{ a.e. in }\timedomain.
  \end{equation}
\end{Lemma}
\begin{Remark}
  See also Proposition 3.5, 3.6 in \cite{GM21} regarding higher regularity properties of the solution to the equation of motion.
\end{Remark}
\begin{Lemma}
  Let $\damage\in\damagefunctions$, $\damageprocess\in\damageprocesses\cap \continuousfunctions^{0,1}(\timedomain; \integrablefunctions^\infty (\spatialdomain;\continuousfunctions^{1,1}(\overline{Y})))$ and respective data for the equation of motion be given.
  Together with the right-hand sides $\damageRHS,\damageRHS_{\discretizationInSpace}$ from Definition \ref{def:right-hand_sides}, the following {\em a priori} estimate
  \begin{equation}
    \|\damageRHS(\damage)-\damageRHS_{\discretizationInSpace}(\damage)\|_{\integrablefunctions^{\infty}(\timedomain;\integrablefunctions^{\infty}(\spatialdomain))} %
    \le C\discretizationInSpace
  \end{equation}
  holds for some constant $C>0$ independent of $\discretizationInSpace$.
\end{Lemma}
\begin{proof}
  The equations from Definition \ref{def:right-hand_sides} together with the uniform bound for damage functions immediately allow us to bound the term $(1-\damage)^{-\alpha}$.
  By definition, $\damageprocess\in\damageprocesses\cap \continuousfunctions^{0,1}(\timedomain; \integrablefunctions^\infty (\spatialdomain;\continuousfunctions^{1,1}(\overline{Y})))$ is Lipschitz.
  Taking  into account a suitable mollifier for the gradient immediately leads us to obtain
  \begin{multline}
    \|\damageRHS_{\spatialvariables}(,\damage(,\spatialvariables))-r_{\spatialvariables,{\discretizationInSpace}}(\timevariable,\damage(\timevariable,\spatialvariables))\|_{\integrablefunctions^{\infty}(\timedomain;\integrablefunctions^{\infty}(\spatialdomain))} \\
      \le c \|\nemytskiioperator(\mollifiedgradient\displacements_{\damage})-\nemytskiioperator(\mollifiedgradient\displacements_{\discretizationInSpace, \damage})\|_{\integrablefunctions^{\infty}(\timedomain;\integrablefunctions^{\infty}(\spatialdomain))}
      \le c \|\displacements_{\damage} - \displacements_{\discretizationInSpace, \damage}\|_{\integrablefunctions^{\infty}(\timedomain;\spatialhilbertfunctions^1(\spatialdomain))}
      \le C \discretizationInSpace.
  \end{multline}
  For this to happen, we used Lemma \ref{lemma:fem_error_estimate}.
  This completes the proof.
\end{proof}
Relying on these preliminaries, we can directly state the error estimate for approximating the right-hand side $\damageRHS$ of the damage evolution equation by $\damageRHS_{\discretizationInSpace}$.
\begin{Corollary}
  Let $\damage, \damage_{\discretizationInSpace}\in\damagefunctions$ be  the solutions to the continuous, and respectively, discretized damage evolution equations.
  $\damageRHS, \damageRHS_{\discretizationInSpace}$ refer to the exact and approximated right-hand sides, respectively.
  Then the following estimate
  \begin{equation}
    \|\damage - \damage_{\discretizationInSpace}\|_{\sobolevfunctions^{1,\infty}(\timedomain;\integrablefunctions^{\infty}(\spatialdomain))} \le C\discretizationInSpace
  \end{equation}
  holds true.
\end{Corollary}
\begin{proof}
  We can start right away with estimating the error in damage.
  For details we refer to part one of this work, especially to the chapter on well-posedness of the damage evolution equation, i.e., Chapter 1.3.1.
  \begin{multline}
    \|\damage - \damage_{\discretizationInSpace}\|_{\sobolevfunctions^{1,\infty}(\timedomain;\integrablefunctions^{\infty}(\spatialdomain))} 
      \le \|r(\damage)-r_{\discretizationInSpace}(\damage)\|_{\integrablefunctions^{\infty}(\timedomain;\integrablefunctions^{\infty}(\spatialdomain))} %
        + \int_0^{\timedomainmax}\|r(\damage)-r_{\discretizationInSpace}(\damage)\|_{\integrablefunctions^{\infty}(\timedomain;\integrablefunctions^{\infty}(\spatialdomain))}\, d\tau %
      \le C\timedomainmax\discretizationInSpace.
  \end{multline}
  This proves our claim.
\end{proof}
After having dealt with the error stemming from the approximation of the right-hand side of the damage equation, we focus now on the consistency error.
As previously, we take the approximated right-hand side $\damageRHS_{\discretizationInSpace}$ and discuss the error made by using an explicit Euler scheme.
Note that transferring such results to the case of higher-order schemes like Runge-Kutta is straightforward; see for instance \cite{DeBo2002}.
\subsection{Consistency}
We are interested in deriving an upper bound of the \emph{consistency error} from \eqref{eq:consistency error} and some kind of uniform estimate on the \emph{order of convergence} $p$ of this particular error.
\begin{Lemma}\label{Lemma:Consistency}
  Take the damage evolution equation from \eqref{damageevolution}, \eqref{initial_damage} with initial damage $\damage_0\in\damagefunctions_0$ and exchange the right-hand side with $\damageRHS_{\discretizationInSpace}$ from \eqref{def:right-hand_sides}.
  Then the consistency error in \eqref{eq:consistency error} is bounded by the size of the time step of the discretization, i.e. it exists $C>0$ such that
  \begin{equation}
    \|\consistencyerror_i\|_{\infty,\timedomain, \spatialdomain} = \|\damage-\damage_{\discretizationInSpace} - \tau\damageRHS_{\discretizationInSpace}\|\le C \tau^2.
  \end{equation}
\end{Lemma}
\begin{proof}
  Let $\spatialvariables_{\discretizationInSpace}$ be fixed.
  Note that $\spatialvariables\in\spatialdomain$ is merely a parameter in \eqref{damageevolution} and \eqref{initial_damage}.
  Employing the Mean-Value Theorem for absolutely continuous functions leads us to
  \begin{multline}
    \damage_{\discretizationInSpace}(\timevariable + \tau,\spatialvariables_{\discretizationInSpace}) - \damage_{\discretizationInSpace}(\timevariable,\spatialvariables_{\discretizationInSpace}) 
      =\tau\int_0^1\partial_y\damage_{\discretizationInSpace}(\timevariable+y\tau, \spatialvariables_{\discretizationInSpace})\,dy = \tau\int_0^1 \damageRHS_{\discretizationInSpace}(\timevariable +y\tau, \spatialvariables_{\discretizationInSpace}, \damage_{\discretizationInSpace}(\timevariable+y\tau, \spatialvariables_{\discretizationInSpace}))\,dy\text{ a.e. in }\timedomain
  \end{multline}
  and insterting this into \eqref{eq:consistency error} reveals
  \begin{multline}
    \consistencyerror(\timevariable,\spatialvariables_{\discretizationInSpace}, \tau)
    = \damage_{\discretizationInSpace}(\timevariable + \tau,\spatialvariables_{\discretizationInSpace}) - \damage_{\discretizationInSpace}(\timevariable,\spatialvariables_{\discretizationInSpace}) - \tau \damageRHS_{\discretizationInSpace}(\timevariable, \spatialvariables_{\discretizationInSpace},\damage_{\discretizationInSpace}(\timevariable,\spatialvariables_{\discretizationInSpace})) \\
    = \tau\int_0^1\Big(\damageRHS_{\discretizationInSpace}(\timevariable + y\tau, \spatialvariables_{\discretizationInSpace},\damage_{\discretizationInSpace}(\timevariable+y\tau,\spatialvariables_{\discretizationInSpace}))-\damageRHS_{\discretizationInSpace}(\timevariable,\spatialvariables_{\discretizationInSpace},\damage_{\discretizationInSpace}(\timevariable,\spatialvariables_{\discretizationInSpace}))\Big)\,dy.
  \end{multline}
  Together with the previous lemmas on the expected regularity for $\damageRHS_{\discretizationInSpace}$ we arrive at
  \begin{multline}
    |\varepsilon(\timevariable, \spatialvariables_{\discretizationInSpace}, \tau)|
      \le \constant\tau\int_0^1\Big|\damageRHS_{\discretizationInSpace}(\timevariable + y\tau, \spatialvariables_{\discretizationInSpace},\damage_{\discretizationInSpace}(\timevariable+y\tau,\spatialvariables_{\discretizationInSpace}))-\damageRHS_{\discretizationInSpace}(\timevariable,\spatialvariables_{\discretizationInSpace},\damage_{\discretizationInSpace}(\timevariable,\spatialvariables_{\discretizationInSpace}))\\
      \pm \damageRHS_{\discretizationInSpace}(\timevariable, \spatialvariables_{\discretizationInSpace},\damage_{\discretizationInSpace}(\timevariable+y\tau,\spatialvariables_{\discretizationInSpace}))\Big|\,dy %
      \le \constant\tau^2
  \end{multline}
  for some $\constant>0$.
  Note that there is only a finite number of discretization points $\spatialvariables_{\discretizationInSpace}\in\spatialdomain_{\discretizationInSpace}$ in space.
  In addition, the point in time was arbitrarily chosen.
  Thus, we conclude this proof by
  \begin{equation}
    \|\consistencyerror\|_{\infty,\timedomain,\spatialdomain} \le \constant\tau^2
  \end{equation}
\end{proof}
\begin{Corollary}
  The numerical scheme is consistent with \emph{order of convergence} equal to one.
\end{Corollary}
\subsection{Convergence}
In this subsection, we concentrate on showing the convergence of our numerical scheme.
Firstly, we focus the attention on the discretized damage evolution with approximated right-hand side.
Afterwards, we tie our results together and show the wanted convergence for the coupled system.
\subsubsection{The approximated damage evolution}
\begin{Proposition}
  The following estimate for the discretization error of the damage evolution equation with approximated right-hand side holds for some constant $\constant>0$
  \begin{equation}
    \|\varepsilon_{\discretizationInSpace,\discretizationInTime}\|_{\integrablefunctions^{\infty}(\timedomain;\integrablefunctions^{\infty}(\spatialdomain))} \le \constant \tau.
  \end{equation}
\end{Proposition}
\begin{proof}
  We start with some preliminaries. Let $K\subset S\times [0,\damageconstant_1]$ be a compact subset.
  According to Lemma \ref{Lemma:Lipschitz with respect to state variable}, the right-hand sides $\damageRHS, \damageRHS_{\discretizationInSpace}$ are Lipschitz in $\damage$.
  Therefore, there exist some constants $\tau_K, \Lambda_K>0$ such that
  \begin{equation}
    \forall (\timevariable,\damage_{\damage_{\discretizationInSpace}}), (\timevariable,\bar\damage_{\discretizationInSpace})\in K,\, 0<\tau\le\tau_K:\,|\tau\damageRHS_{\discretizationInSpace}(\timevariable,\damage_{\discretizationInSpace}) - \tau\damageRHS_{\discretizationInSpace}(\timevariable,\bar\damage_{\discretizationInSpace})| \le \Lambda_K|\damage_{\discretizationInSpace}-\bar\damage_{\discretizationInSpace}|.
  \end{equation}
  Without loss of generality we further assume $\hat\tau\le\tau_K$, so that for all $\timevariable_i\in\timedomain_{N}$ and $j=0,1,\dots,n_{\discretizationInTime}-1$ there exists a discrete evolution (cf.~\Cref{figure:consistency and convergence}). Also, there exists at least one $\delta_K>0$ such that for almost all $\spatialvariables\in\spatialdomain$
  \begin{equation}
    \forall\timevariable\in\bar\timedomain:\forall y\in\mathbb{R}:|y-\damage_{\discretizationInSpace}(\timevariable,\spatialvariables)|\le\delta_K \Rightarrow (\timevariable,y)\in K
  \end{equation}
  holds. \\[2ex]
  We have to show that our scheme is introducing a function on the discretized time domain $\timedomain_N$.
  For the moment, we will just assume just this for the discrete evolution and that the discretization error at every discrete point in time $\timevariable\in\timedomain_N$ is bounded by $\delta_K$, i.e.,
  \begin{equation}
    \forall\timevariable\in\timedomain_{N}:\,|\varepsilon_{\discretizationInSpace, \discretizationInTime}(\timevariable)| = |\damage_{\discretizationInSpace}(\timevariable)-\damage_{\discretizationInSpace, \discretizationInTime}(\timevariable)|\le\delta_K.
  \end{equation}
  The last property ensures for our discrete evolution to stay within a tubular neighborhood around our proper solution $\damage_{\discretizationInSpace}$. \\[2ex]
  The next step is then to split up the discretization error into different parts by adding a zero, i.e.,
  \begin{multline}
    \varepsilon_{\discretizationInSpace, \discretizationInTime}(\timevariable_{j+1}) %
      =\damage_{\discretizationInSpace}(\timevariable_{j+1})-\damage_{\discretizationInSpace, \discretizationInTime}(\timevariable_{j+1}) \\
      = \damage_{\discretizationInSpace}(\timevariable_{j+1})-\big(\damage_{\discretizationInSpace, \discretizationInTime}(\timevariable_{j}) + \tau\damageRHS_{\discretizationInSpace}(\timevariable_j, \damage_{\discretizationInSpace, \discretizationInTime}(\timevariable_j)\big) %
      \pm \damage_{\discretizationInSpace}(\timevariable_j) + \tau\damageRHS_{\discretizationInSpace}(\timevariable_j, \damage_{\discretizationInSpace}(\timevariable_j)) \\
      = \underbrace{\damage_{\discretizationInSpace}(\timevariable_{j+1})-\big(\damage_{\discretizationInSpace}(\timevariable_j) + \tau\damageRHS_{\discretizationInSpace}(\timevariable_j, \damage_{\discretizationInSpace}(\timevariable_j))\big)}_{=\varepsilon_{j+1}} \hspace{2.5cm}\\
      + \underbrace{\big(\damage_{\discretizationInSpace}(\timevariable_j) + \tau\damageRHS_{\discretizationInSpace}(\timevariable_j, \damage_{\discretizationInSpace}(\timevariable_j))\big) -\big(\damage_{\discretizationInSpace, \discretizationInTime}(\timevariable_{j}) + \tau\damageRHS_{\discretizationInSpace}(\timevariable_j, \damage_{\discretizationInSpace, \discretizationInTime}(\timevariable_j)\big)}_{=:\varepsilon_{\discretizationInSpace, \discretizationInTime}(\timevariable_{j+1})}.
  \end{multline}
  Realizing that here $\varepsilon_{j+1}$ is the consistency error from Lemma \ref{Lemma:Consistency}, we focus on the remainder and build on the fact that the right-hand sides are Lipschitz with respect to the damage variable (see Lemma \ref{Lemma:Lipschitz with respect to state variable}).
  It yields
  \begin{multline}
    |\varepsilon_{\discretizationInSpace, \discretizationInTime}(\timevariable_{j+1})| %
      = \Big|\damage_{\discretizationInSpace}(\timevariable_j) - \damage_{\discretizationInSpace, \discretizationInTime}(\timevariable_j) + \tau \Big( \damageRHS_{\discretizationInSpace}(\timevariable_j,\damage_{\discretizationInSpace, \discretizationInTime}(\timevariable_j)) - \damageRHS_{\discretizationInSpace}(\timevariable_j,\damage_{\discretizationInSpace}(\timevariable_j)) \Big) \Big| \\
      \le (1+\tau\Lambda_K)|\damage_{\discretizationInSpace}(\timevariable_j)-\damage_{\discretizationInSpace, \discretizationInTime}(\timevariable_j)| %
        = (1+\tau\Lambda_K)|\varepsilon_{\discretizationInSpace, \discretizationInTime}(\timevariable_j)|.
  \end{multline}
  This allows for an recursive formulation of the discretization error:
  \begin{enumerate}[(i)]
    \item $|\varepsilon_{\discretizationInSpace, \discretizationInTime}(\timevariable_0)|=0$,
    \item $|\varepsilon_{\discretizationInSpace, \discretizationInTime}(\timevariable_{j+1})|\le \constant\tau^2 + (1+\tau\Lambda_K)|\varepsilon_{\discretizationInSpace, \discretizationInTime}(\timevariable_j)|, \quad j=0,1,\dots,n_{\discretizationInTime}-1$.
  \end{enumerate}
  Using the recursion repeatedly leads us to
  \begin{equation}
    |\varepsilon_{\discretizationInSpace, \discretizationInTime}(\timevariable_{j+1})|\le \constant(\tau^2\sum_{j=0}^{n_{\discretizationInTime}-1}(1+\tau\Lambda_K)^j + \underbrace{(1+\tau\Lambda_K)^j|\varepsilon_{\discretizationInSpace, \discretizationInTime}(\timevariable_0)}_{=0}|
  \end{equation}
  for all $j=0,1,\dots,n_{\discretizationInTime}-1$.
  Using the representation $(a-1)s_n=a^{n+1}-1$ for partial sums $s_n$ of a geometric series then gives
  \begin{multline}
    |\varepsilon_{\discretizationInSpace, \discretizationInTime}(\timevariable_{j+1})|\le \constant\tau^2\sum_{j=0}^{n_{\discretizationInTime}-1}(1+\tau\Lambda_K)^j %
    = \constant\tau^2\frac{(1+\tau\Lambda_K)^{n_{\discretizationInTime}}-1}{1+\tau\Lambda_k-1} \\
    = \tau\frac{\constant}{\Lambda_K}\Big((1+\tau)\Lambda_K)^{n_{\discretizationInTime}}-1\Big)
    \le \tau\frac{\constant}{\Lambda_K}\Big((\exp(\tau\Lambda_K)^{n_{\discretizationInTime}}-1\Big) \\
    \le \tau\frac{\constant}{\Lambda_K}\Big((\exp(\Lambda_K(\timedomainmax-\timevariable_0))-1\Big)
  \end{multline}
  for $j=0,1,\dots,n_{\discretizationInTime}-1$.
  Choosing $\tau>0$ suitably small then shows
  \begin{equation}\label{Equation:Discretization error}
    |\varepsilon_{\discretizationInSpace, \discretizationInTime}(\timevariable_{j})|\le\delta_K, \quad j=0,1,\dots,n_{\discretizationInTime}
  \end{equation}
  as well as that $\damage_{\discretizationInSpace, \discretizationInTime}(\timevariable_j)$ indeed is a mesh function on $\timedomain_N$.
  Since the right-hand side of \eqref{Equation:Discretization error} does not depend on $\spatialvariables\in\spatialdomain_{\discretizationInSpace}$ we have proven the claim.
\end{proof}
\subsubsection{Fully coupled numerical scheme}
We close this section by showing the convergence of the coupled numerical scheme.
Since we are approximating damage in our FE scheme via a specific time discretization, we have to take this error into account as well.
\begin{Proposition}
  The implemented numerical scheme converges towards the exact solution of at least first order.
  Thus the following estimate holds
  \begin{equation}
    \|\displacements - \displacements_{\discretizationInSpace, \discretizationInTime}\|_{\integrablefunctions^{\infty}(\timedomain;\spatialhilbertfunctions^1(\spatialdomain))} %
      + \|\damage - \damage_{\discretizationInSpace, \discretizationInTime}\|_{\sobolevfunctions^{1, \infty}(\timedomain;\integrablefunctions^{\infty}(\spatialdomain))} %
        \le\constant(\tau + \discretizationInSpace)
  \end{equation}
  for some constant $\constant>0$.
\end{Proposition}
\begin{proof}
  We start with the equation of motion.
  Note that we neglect the interpolation errors of data in the favour of a clear presentation.
  It is common knowledge that the interpolation errors can be bound by the step size in space. \\[2ex]
  Let $\timevariable\in\timedomain$ be fixed.
  We start with the weak formulation of both the original and discretized equations of motion
  \begin{align}
    -\divergence\left((1-\damage(\timevariable))\elasticitytensor\strains(\displacements)(\timevariable)\right) &=\forces(\timevariable), \\
    -\divergence\left((1-\damage_{\discretizationInSpace, \discretizationInTime}(\timevariable))\elasticitytensor\strains(\displacements_{\discretizationInSpace, \discretizationInTime})(\timevariable)\right) &=\forces(\timevariable),
  \end{align}
  which holds true almost everywhere in $\spatialvariables\in\spatialdomain$.
  Subtracting both equations from one another, adding a zero, and reorganizing the result leads to
  \begin{equation}
    -\divergence\Big(\big(1-\damage_{\discretizationInSpace, \discretizationInTime}(\timevariable)\big)\big(\elasticitytensor\strains(\displacements(\timevariable) - \displacements_{\discretizationInSpace, \discretizationInTime}(\timevariable)\big)\Big)
      = -\divergence\Big(\big(\damage(\timevariable) - \damage_{\discretizationInSpace, \discretizationInTime}(\timevariable)\big)\elasticitytensor\strains(\displacements(\timevariable)\Big).
  \end{equation}
  Testing this identity with $\displacements - \displacements_{\discretizationInSpace, \discretizationInTime}$ and combining with integration by parts allow for a suitable weak form.
  Keep in mind that the damage variable is uniformly bounded by 1.
  Thus, employing standard arguments commonly to establish energy estimates for the equation of motion,  we obtain
  \begin{equation}
    \|\displacements(\timevariable) - \displacements_{\discretizationInSpace, \discretizationInTime}(\timevariable)\|_{\spatialhilbertfunctions^1(\spatialdomain)} %
      \le\constant\| \damage - \damage_{\discretizationInSpace, \discretizationInTime}\|_{\sobolevfunctions^{1, \infty}(\timedomain;\integrablefunctions^{\infty}(\spatialdomain))}.
  \end{equation}
  Since this holds for an arbitrary point in time, we can replace the left-hand side with $\|\displacements - \displacements_{\discretizationInSpace, \discretizationInTime}\|_{\integrablefunctions^{\infty}(\spatialhilbertfunctions^1(\spatialdomain))}$.
  In other words, we can bound the error made during computing the displacements by the one from our time-stepping scheme which we investigated in the previous sections.
  Employing these results we are able to state
  \begin{multline}
    \|\displacements - \displacements_{\discretizationInSpace, \discretizationInTime}\|_{\integrablefunctions^{\infty}(\spatialhilbertfunctions^1(\spatialdomain))} %
      + \|\damage - \damage_{\discretizationInSpace, \discretizationInTime}\|_{\sobolevfunctions^{1, \infty}(\timedomain;\integrablefunctions^{\infty}(\spatialdomain))} 
      \le (1 + \constant)\| \damage - \damage_{\discretizationInSpace, \discretizationInTime}\|_{\sobolevfunctions^{1, \infty}(\timedomain;\integrablefunctions^{\infty}(\spatialdomain))}
      \le\constant(\discretizationInSpace + \tau),
  \end{multline}
  which proves the wanted statement.
\end{proof}

\section{Singularities in corners}\label{Sec:Corners}
Before discussing the implementation of our numerical scheme we address some of the statements we made in part one of this work regarding boundary conditions and regularity (see Remark 2.1(a) in \cite{GM21}).\\[2ex]
We have indicated some references as to why connecting boundaries of different types causes additional issues when trying to show higher regularity of a solution.
Without an additional compatibility condition or asking for a boundary that is made up of unions of connected components, singularities my arise at the points where these different types of boundary conditions intersect.
Our intention here is to give physical motivation for the observed phenomena and also give a few options to overcome these drawbacks.
In Section \ref{sec:3}, we discuss the general numerical response of our model, while in one of the subsequent subsection we compare these different approaches.
\subsection{Physically motivated example}
When looking at the spatial distribution of damage pointed out in Subsection \ref{sec:tc03_1}, see \Cref{figure:results TC03S01} e.g., we observe that the maximum damage is being attained only at the domain's corners.
This is not what we would expect physically.
In a standard tensile test, e.g., the geometry of a cylinder-like specimen is chosen in such a way that the highest stresses appear in the smallest cross-section.
Damage evolution is essentially governed by $\damageprocess$ representing some sort of norm of the equivalent stresses and thus connecting damage to stress.
So we would expect the highest stresses to actually appear in this cross-sectional area.
To get a better understanding we take a closer look at the stresses.
\Cref{figure:corners} illustrates the situation.
\begin{center}
  \begin{minipage}{.7\linewidth}
    \centering
    \includegraphics{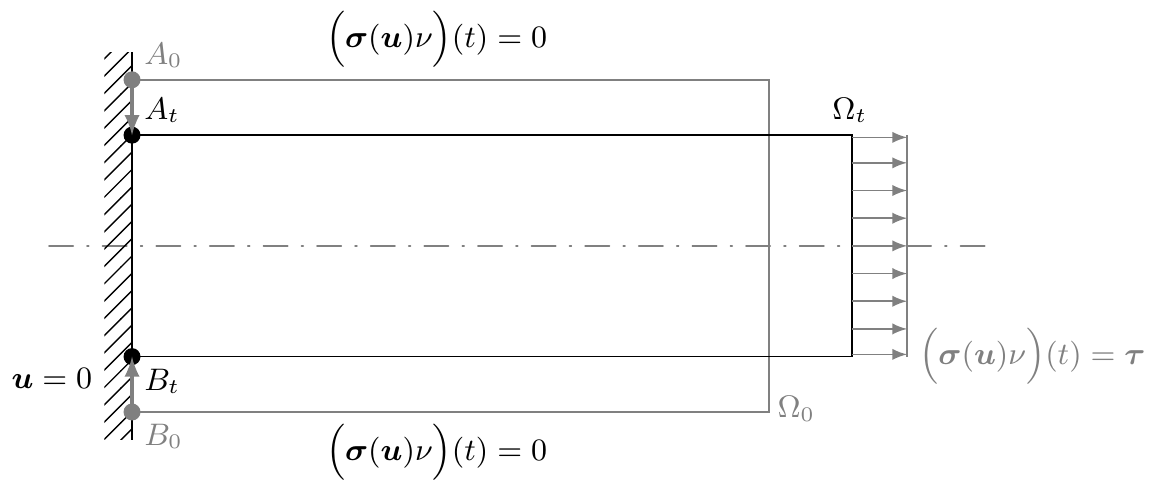}
  \end{minipage}
  \captionof{figure}{Explaining corner singularities}
  \label{figure:corners}
\end{center}
Figure \ref{figure:corners} shows an example.
We have a linear-elastic material, that is fixed at the left end, pulled at the right end, and tension free on the remaining parts of the boundary.
The reference and damaged configuration is denoted by $\spatialdomain_0$, $\spatialdomain_{\timevariable}$, respectively.
Due to the traction on the right end, the material is also deformed vertically, which is caused by \emph{Poisson's ratio} describing the signed ratio of transverse to axial strain.
Since the upper and lower parts of the boundary are traction free, these boundary parts are free to move towards the neutral axis depicted by the dashed line.
This transverse contraction is often referred to as \emph{necking} in engineering literature.
If we take a closer look at the corners of the undeformed configuration, i.e., namely $A_0, B_0$, these points, however, are fixed due to the Dirichlet type boundary condition and therefore obstructed to move to positions $A_{\timevariable},B_{\timevariable}$ despite the transverse contraction.
The only way to remedy this situation physically are such high stresses to keep $A_{\timevariable}=A_0$ and $B_{\timevariable}=B_0$ satisfied.
Note this is a phenomena not connected to damage or the coupling for damage and elastic behavior, but inherent in the elastic model.
Mathematically, we have commented on this in part one of our work and referred to general works on regularity in connection with different types of boundary conditions.
\begin{Remark}
  Modeling this situation with Dirichlet-type boundary conditions entails a couple limitations, physically.
  Firstly, no material point on that part of the boundary can not have any sort of displacement.
  In most cases when looking at geometries of test specimen this means only a small part of the boundary that connects to the test machine can be modelled with Dirichlet conditions.
  Secondly, this also entails a completely rigid test machine, that is no displacements at all.
  These conditions of no displacement at all is a very strong imposition and idealization for a linear elastic model.
\end{Remark}
\subsection{How to remedy the occurrence of singularities?}
One way to solve this issue is to only apply a Neumann-type boundary condition on the whole of $\partial\spatialdomain$.
This would entail to change the basic functional space for displacements to
\begin{equation}
  \spatialfunctions:=\left\{\tractiondrivendisplacements\in\sobolevfunctions^{1,2}(\spatialdomain);\, \int_{\partial\spatialdomain}\tractiondrivendisplacements\,d\boldsymbol{s}=0\right\}.
\end{equation}
Depending on the type of application, this makes sense.
For instance, in many tensile test setups specimen are clamped into the machine used for testing and therefore held by force or friction, respectively.
One could also change the roughness of a specimens' surface in the area where the specimen is fixed in the test machine.
This would allow for a smoother transfer of forces into the material and could be modeled with boundary forces of higher regularity.
Note that there are measuring procedures were actual measurements are taken close to the surface.
Please also note that choosing Neumann-type boundary condition as proposed here, limits our identification setting.
The reason being that $\damageprocess$ have satisfy certain symmetry condition, i.e., $\damageprocess\in\continuousfunctions^{1,1}(\mathbb{R}^4)$ instead of $\damageprocess\in\continuousfunctions^{1,1}(\timedomain\times\spatialdomain\times\mathbb{R}^4)$, for example.\\[2ex]
If the specimen is fixed in the test machine by form closure, then one can make use of the fact that for a classical tensile test that the state of stress is homogeneous, which would also allow a modelling with only Neumann-type boundary conditions, like we did in the first test cases of Section \ref{sec:3}.
Also, see Table \ref{Table:TestCases}.\\[2ex]
This brings us to applying Robin-type boundary conditions with suitable $\robin>0$ such that
\begin{equation}
  \stresses(\displacements)\normals + \robin(\displacements - \boundarydisplacements) = \stressforces, \quad \text{ on }  \partial\spatialdomain.
\end{equation}
Here, $\boundarydisplacements$ denotes displacements prescribed on the boundary and $\stressforce$ forces on part of the boundary.
In the case that $\boundarydisplacements,\stressforces=\boldsymbol{0}$ holds, the parameter $\robin$ can be viewed as a type of modulus of resilience and as such taking the stiffness or elasticity of the test machine into account.
This relaxes the effect of the observed regularity and can model a smoother transition from one boundary condition to the next.
The parameter can also be modeled as time and space dependent.
Note that these type of boundary conditions does not work in our setting (cf. \cite{GM21}).\\[2ex]
Another way is to change the geometry in such a way that different type of boundary conditions do not intersect.
This situation is not possible in a three-dimensional setting.
If modelling plane phenomena or reducing dimensions to fewer than three dimensions are topics of interest, then this is a good alternative.

\section{Implementation}
We briefly comment on some of the implementation details.
This entails specifying the different sets of parameters that were used for the simulation, including different types of damage processes $\damageprocess$, types of boundary conditions, or different types of domains investigated, e.g.
We close the section with a high-level look at the algorithm that we employed for solving the coupled problem.
\subsection{Model parameters}
In view of \Cref{Definition:Discretization}, we start with parametrizations of the discretized model according to \Cref{Table:Parametrization of the numerical model}.
For simplicity, we decided to use a quadratically shaped domain.
Keep in mind that we want to model a standard tensile.
\begin{center}
    \captionof{table}{Parametrization of numerical models}
    {%
    \renewcommand{\arraystretch}{1.25}
    \begin{longtable}{@{\extracolsep{\fill}}lll@{}} \noalign{\hrule height 1.5pt}
      {\bfseries Parameter}                & {\bfseries Description}                  &  {\bfseries Choice/Value}	                                                                      \\ \noalign{\hrule height 1.5pt} \\[-2ex]
      $\timedomain\times\spatialdomain_0$  &  computational domain                    &  $[0,10]\times([0,1]\times[0,1])$                                                               \\
      $\boundary_0$                        &  fixed bottom                            &  $[0,10]\times[0,1]\times\{0\}$                                                                 \\
      $\boundary_1$                        &  tension free sides                      &  $[0,10]\times\{0\}\times[0,1]$,                                                                \\
                                           &                                          &  $[0,10]\times\{1\}\times[0,1]$                                                                 \\
      $\boundary_2$                        &  loaded top side                         &  $[0,10]\times[0,1]\times\{1\}$                                                                 \\
      $\lambda$                            &  Lam\'e coefficient                      &  $121.15$                                                                                       \\
      $\mu$                                &  Lam\'e coefficient                      &  $80.77$                                                                                        \\
      $\forces$                            &  volumetric forces                       &  $\boldsymbol{0}$                                                                               \\
      $\alpha$                             &  model parameter                         &  $1.0$                                                                                          \\
      $\robin$                             &  parameter Robin                         &  $100.0$                                                                                        \\
      $\damageprocess_0$                   &  damageprocess Model 0                   &  $0$                                                                                            \\
      $\damageprocess_1$                   &  damageprocess Model 1                   &  $0.008(\frac{2}{3}(1+\nu)\sigma_{eq}^2 + 3(1-2\nu)\sigma_{H}^2)$                               \\
      $\damageprocess_2$                   &  damageprocess Model 2                   &  $0.00625\damageprocess_1$                                                                      \\
      $\sigma_{H}$                         &  hydrostatic stresses                    &  $\frac{1}{3}\trace \stresses$                                                                  \\
      $\sigma_{eq}$                        &  von Mises stresses                      &  $\sqrt{\frac{3}{2}(\stresses-\sigma_{H}\boldsymbol{I}):(\stresses-\sigma_{H}\boldsymbol{I})}$  \\
      $\stressforces_0$                    &  boundary stresses                       &  $\boldsymbol{0}$                                                                               \\
      $\stressforces_1$                    &  boundary stresses                       &  $(\begin{matrix} 0 & 0.5\,\timevariable\end{matrix})^{\transpose}$                             \\
      $\stressforces_2$                    &  boundary stresses                       &  $(\begin{matrix} 0 & 5.0\sin(\frac{\pi}{5.0}\timevariable)\end{matrix})^{\transpose}$          \\
      $\boundarydisplacements_0$           &  displacements on parts of $\boundary$   &  $\boldsymbol{0}$                                                                               \\
      $\boundarydisplacements_1$           &  displacements on parts of $\boundary$   &  $(\begin{matrix} 0 & 0.5\,\timevariable\end{matrix})^{\transpose}$                             \\
      $\boundarydisplacements_2$           &  displacements on parts of $\boundary$   &  $(\begin{matrix} 0 & 0.5\sin(\frac{\pi}{5.0}\timevariable)\end{matrix})^{\transpose}$          \\
      $\displacements_0$                   &  initial displacements                   &  $\boldsymbol{0}$                                                                               \\
      $\displacements_0$                   &  initial displacements                   &  $\boldsymbol{0}$                                                                               \\
      $\damageconstant_0$                  &  damage constant                         &  $0.999$
    \end{longtable}
    \label{Table:Parametrization of the numerical model}
    }
\end{center}
  The damage process $\damageprocess_0$ is used for referencing reasons, relating the damage models to damage-free ones.
  Taken from \cite{JeLe84}, the choice of $\damageprocess_1$ represents a simple version of an early Kachanov-type model.
  The difference in coefficients for $\damageprocess_1$, $\damageprocess_2$ is due to the effect different types of boundary conditions have on the solution.
  This will be discussed in more detail in sections \ref{Sec:Corners}, \ref{sec:3}.
\subsection{Computational domain}
For our numerical simulations, we investigated three different shapes of domains for spatial discretization.
These will be discussed in test case \texttt{TC03} Section \ref{sec:3}.
\begin{center}
  \begin{minipage}[t]{0.33\linewidth}
   \centering
   \includegraphics[trim=300 100 250 75, clip,width=\linewidth]{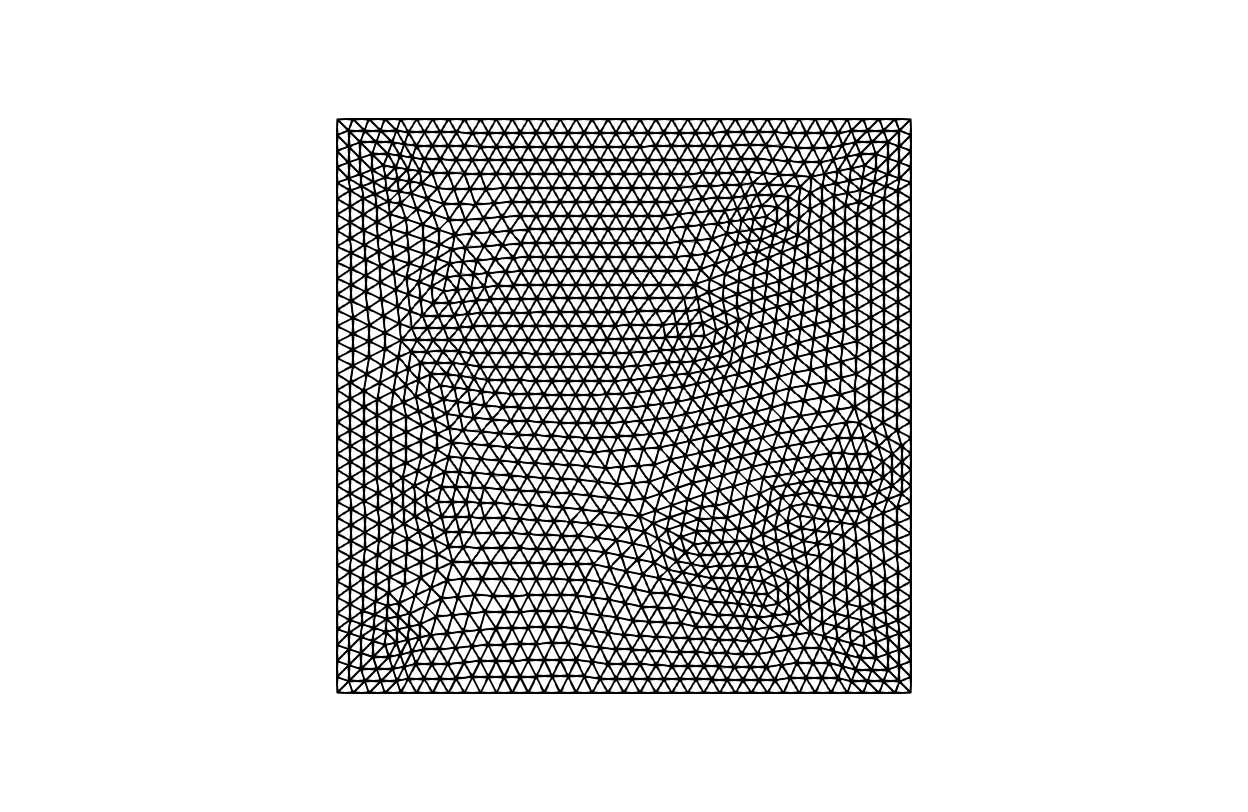} 
  \end{minipage}\hfill
  \begin{minipage}[t]{0.33\linewidth}
   \centering
   \includegraphics[trim=300 100 250 75, clip,width=\linewidth]{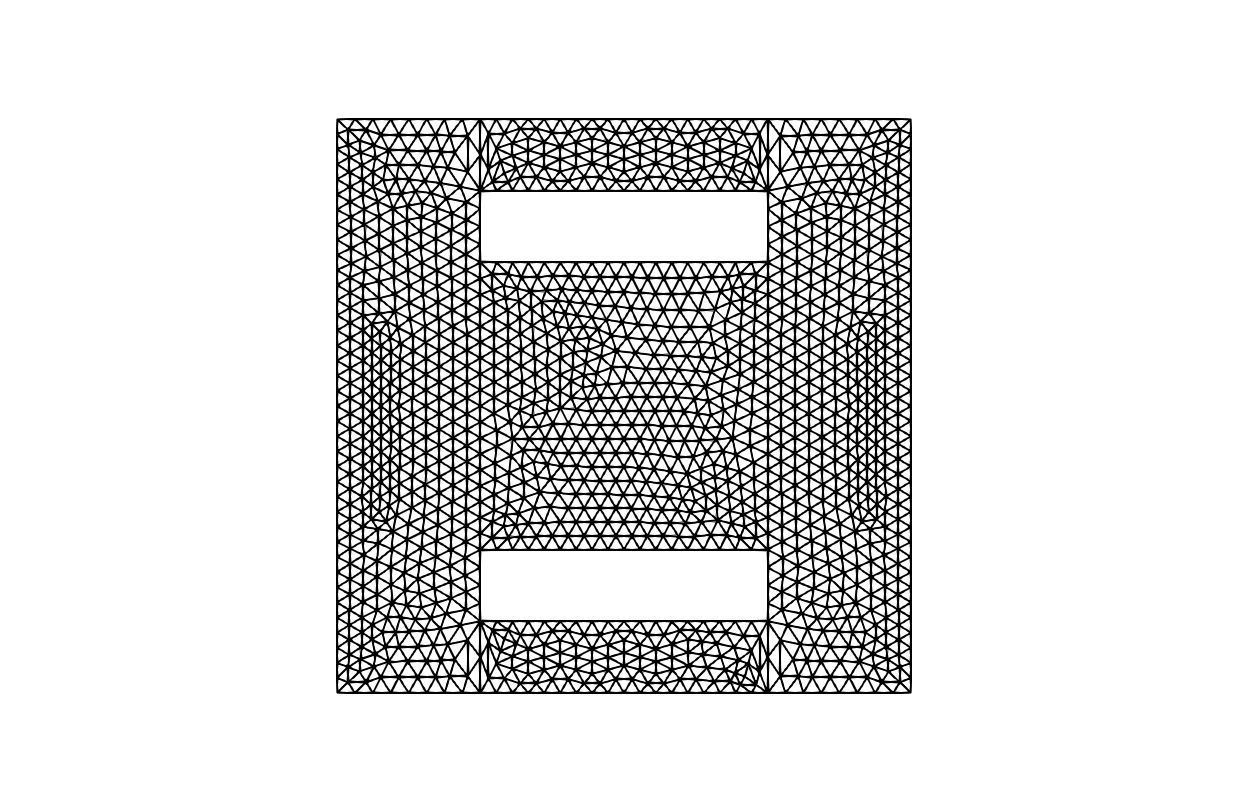} 
  \end{minipage}\hfill
  \begin{minipage}[t]{0.33\linewidth}
   \centering
   \includegraphics[trim=300 100 250 75, clip,width=\linewidth]{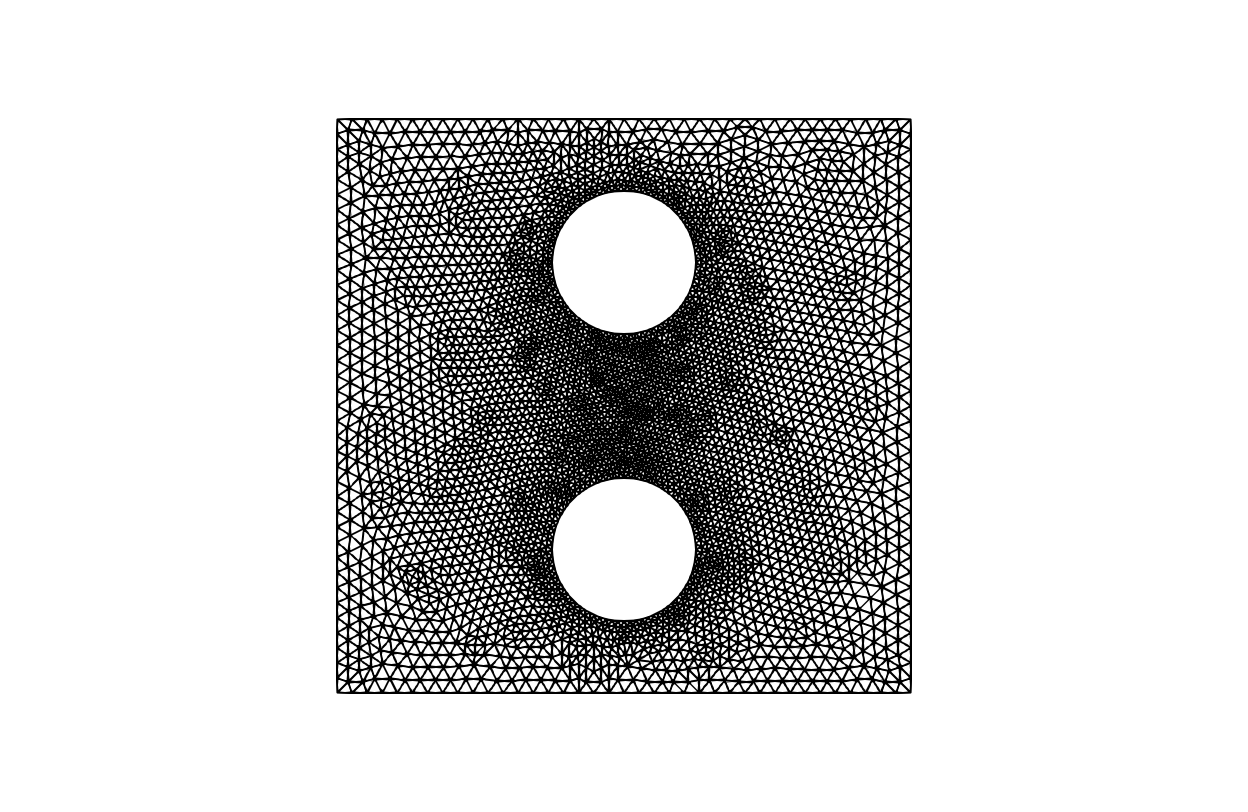} 
  \end{minipage}
  \captionof{figure}{Computational domains $\spatialdomain_0$, $\spatialdomain_1$, $\spatialdomain_2$ from left to right}
  \label{figure:Computational domains}
 \end{center}
Domain $\spatialdomain_0$ is a triangulated unit square where the boundaries and boundary conditions are prescribed as listed in Table \ref{Table:Parametrization of the numerical model}.
For the remaining cases the outer boundaries will take the role of $\boundary_1$ in $\spatialdomain_0$ and we will always apply homogeneous Neumann-type conditions on them.
The inner boundaries of $\spatialdomain_1$, $\spatialdomain_2$ will be denoted by $\boundary_0$, $\boundary_2$, respectively.
The lower position boundary will be indexed with $0$, and the upper with $2$.
\subsection{The algorithm}
  \Cref{main.cpp} shows in principle, how the strategy to approximate the solution of the direct problem was implemented in \FEniCS.
\begin{algorithm}[H]
\caption{}\label{main.cpp}
\begin{algorithmic}[1]
\State{\textbf{set} $\timevariable_0$}\Comment{Set initial time}
\State{\textbf{set} $\displacements_{h,0}$}\Comment{Set initial displacements}
\State{\textbf{set} $\damage_{h,0}$}\Comment{Set initial damage}
\While{$\timevariable_n<\timedomainmax$}
  \State{$\damage_{\discretizationInSpace, \discretizationInTime,n}^i\gets$ \textbf{solve} \textnormal{equation \eqref{equation:Explicit_Euler}}}\Comment{Solve damage equation with \emph{explicit Euler scheme}}
  \State{$\damage_{\discretizationInSpace, \discretizationInTime,n}\gets \sum_{i=1}^{N_h}\damage_{\discretizationInSpace, \discretizationInTime,n}^i\varphi_{\damagefunctions_h}^i,$}\Comment{Compare to equation \eqref{equation:damage in FE space}}
  \State{$\displacements_{\discretizationInSpace, \discretizationInTime,n} \gets$ \textbf{solve} \textnormal{equation \eqref{equation:discretized EqOfMo}}}
  \State{$\boldsymbol{w}_{\discretizationInSpace, \discretizationInTime,n}\gets$ \textbf{solve} \textnormal{equation \eqref{equation:Projection of Displacementgradient}}}\Comment{$\integrablefunctions^2$-Projection of $\nabla\displacements_{\discretizationInSpace, \discretizationInTime}$}
  \State{$\timevariable_{n+1}\gets\timevariable_{n}+\Delta\timevariable$}
\EndWhile
\end{algorithmic}
\end{algorithm}
We start by setting the initial time step $\timevariable_0$, $\displacements_{\discretizationInSpace, \discretizationInTime, 0}$, and $\damage_{\discretizationInSpace, \discretizationInTime, 0}$.
In every time step, we then solve the damage equation using the displacements from the previous time step.
We then project the solution onto our FE space and solve the equation of motion.
At the end, we also compute the $\integrablefunctions^2$-projection of the displacements' gradient so that it can be used to determine damage evolution in the next time step.

\section{Numerical simulations}\label{sec:3}
We composed several test cases to assess our model performance from the perspective of numerical simulations.
These test cases are listed in Table \ref{Table:TestCases}.
We start by taking a look at convergence rates. This is the role of  Subsection \ref{sec:tc00}. We compare the obtained numerical results to the theoretical ones proven in Section \ref{sec:numerical_analysis}.
Test case \texttt{TC01} compares the Kachanov-type model output with a damage free setting by applying a linear in time increasing load.
The results are shown and discussed in Subsection \ref{sec:tc01}.
The next test case \texttt{TC02} then addresses the model response to a dynamic (cyclic) loading.
We compare our model output to a damage free one in Subsection \ref{sec:tc02}.
Afterwards, we pick up again the discussion from Section \ref{Sec:Corners} and discuss boundary regularity issues in Subsection \ref{sec:tc03}.
We close this section by looking  in Subsection \ref{sec:tc04} at a test case involving a substantial damage.
\begin{center}
  \captionof{table}{Overview test cases}
  {%
  \renewcommand{\arraystretch}{1.25}
  \begin{longtable}{@{\extracolsep{\fill}}ll@{\hspace{.15\textwidth}}lllll@{}} \noalign{\hrule height 1.5pt}
    {\bfseries Test case}  & {\bfseries Description}  & $\boldsymbol{\spatialdomain_i}$	 & $\boldsymbol{\damageprocess_i}$ & $\boldsymbol{y|_{\boundary_0}}$                   & $\boldsymbol{y|_{\boundary_1}}$ & $\boldsymbol{y|_{\boundary_2}}$                   \\ \noalign{\hrule height 1.5pt}
    \texttt{TC00S00}       &  Convergence in space    & $\spatialdomain_0$               & $\damageprocess_1$              & $-\stressforces_2$                                & $\stressforces_0$               & $\stressforces_2$                                 \\
    \texttt{TC00S01}       &  Convergence in time     & $\spatialdomain_0$               & $\damageprocess_1$              & $-\stressforces_2$                                & $\stressforces_0$               & $\stressforces_2$                                 \\ \noalign{\hrule}
    \texttt{TC01S00}       &  Linear load 1           & $\spatialdomain_0$               & $\damageprocess_0$              & $-\stressforces_2$                                & $\stressforces_0$               & $-\stressforces_2$                                \\
    \texttt{TC01S01}       &  Linear load 2           & $\spatialdomain_0$               & $\damageprocess_1$              & $-\stressforces_2$                                & $\stressforces_0$               & $-\stressforces_2$                                \\ \noalign{\hrule}
    \texttt{TC02S00}       &  Dynamic load 1          & $\spatialdomain_0$               & $\damageprocess_1$              & $-\stressforces_2$                                & $\stressforces_0$               & $-\stressforces_2$                                \\
    \texttt{TC02S01}       &  Dynamic load 2          & $\spatialdomain_0$               & $\damageprocess_1$              & $-\stressforces_2$                                & $\stressforces_0$               & $-\stressforces_2$                                \\ \noalign{\hrule}
    \texttt{TC03S00}       &  Dirichlet 1             & $\spatialdomain_0$               & $\damageprocess_2$              & $\boundarydisplacements_0$                        & $\stressforces_0$               & $\boundarydisplacements_2$                        \\
    \texttt{TC03S01}       &  Dirichlet 2             & $\spatialdomain_0$               & $\damageprocess_2$              & $\boundarydisplacements_0$                        & $\stressforces_0$               & $\boundarydisplacements_2$                        \\
    \texttt{TC03S02}       &  Robin                   & $\spatialdomain_0$               & $\damageprocess_2$              & $\robin,\boundarydisplacements_0,\stressforces_0$ & $\stressforces_0$               & $\robin,\boundarydisplacements_2,\stressforces_0$ \\
    \texttt{TC03S03}       &  Domain 1                & $\spatialdomain_1$               & $\damageprocess_2$              & $\boundarydisplacements_0$                        & $\stressforces_0$               & $\boundarydisplacements_2$                        \\
    \texttt{TC03S04}       &  Domain 2                & $\spatialdomain_2$               & $\damageprocess_2$              & $\boundarydisplacements_0$                        & $\stressforces_0$               & $\boundarydisplacements_2$                        \\
  \end{longtable}
  \label{Table:TestCases}
  }
\end{center}

\subsection{Convergence}\label{sec:tc00}
We test numerically the analytically established convergence rates of our  model from Section \ref{sec:numerical_analysis}.
The model is parameterized according to \texttt{TC00}.
\begin{center}
  \begin{minipage}[t]{0.49\linewidth}
    \centering
    \includegraphics{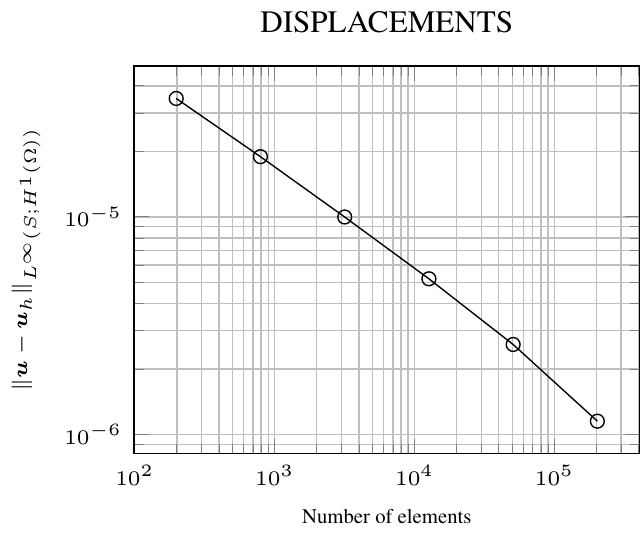}
  \end{minipage}\hfill
  \begin{minipage}[t]{0.49\linewidth}
   \centering
   \includegraphics{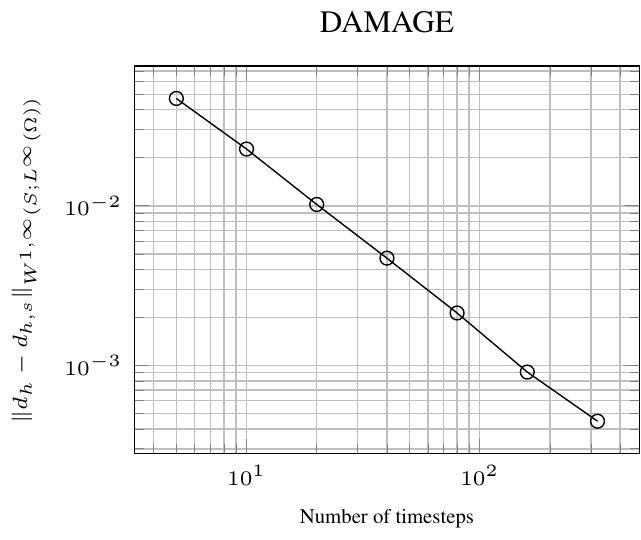}
  \end{minipage}
  \captionof{figure}{\texttt{TC00S00}. Convergence in time and space.}
  \label{figure:convergence}
\end{center}
Analytically, we established at least linear convergence of the discretization error in time and space.
This is confirmed by the results shown in Figure \ref{figure:convergence}.
The discretization error in time meets our expectations since we used an explicit Euler scheme which is of first order.
\subsection{Damage versus no damage evolution}\label{sec:tc01}
We start by looking more closely on the effect damage has on the equation of motion.
Therefore we look into linear and cyclic loading to get a feel for Kachanov-type model responses.
\subsubsection{Linear loading in time}
We tend to test cases \texttt{TC01} and \texttt{TC02}, where we impose symmetric linear and dynamic forces on boundaries $\boundary_0$ and $\boundary_2$, respectively.
Our goal is to observe a different of response when comparing these to outputs of damage free models.
The results for a damage free model are displayed in Figure \ref{figure:results TC01S00}.
  \begin{center}
   \begin{minipage}[t]{0.33\linewidth}
    \centering
    \includegraphics[trim=300 100 200 75, clip,width=\linewidth]{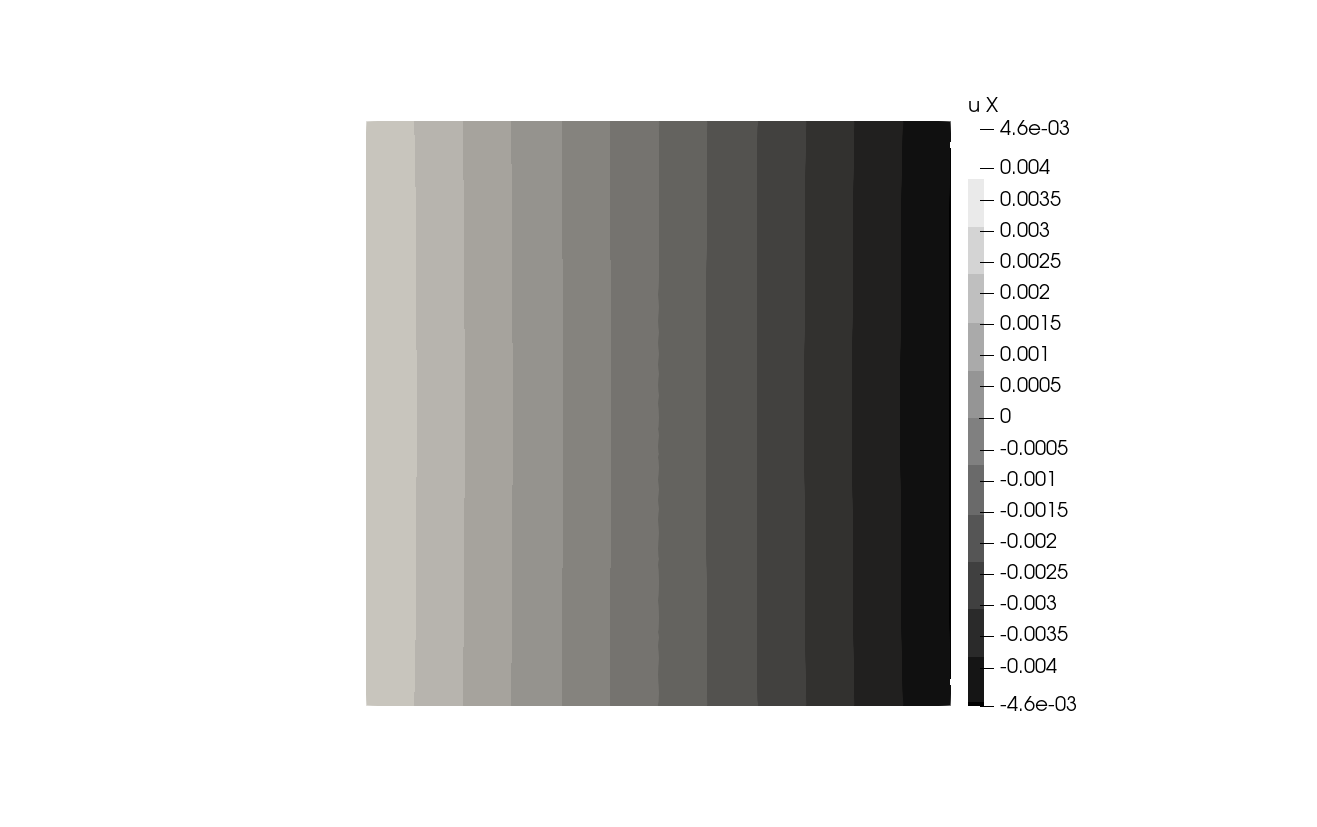}  
   \end{minipage}\hfill
   \begin{minipage}[t]{0.33\linewidth}
    \centering
    \includegraphics[trim=300 100 200 75, clip,width=\linewidth]{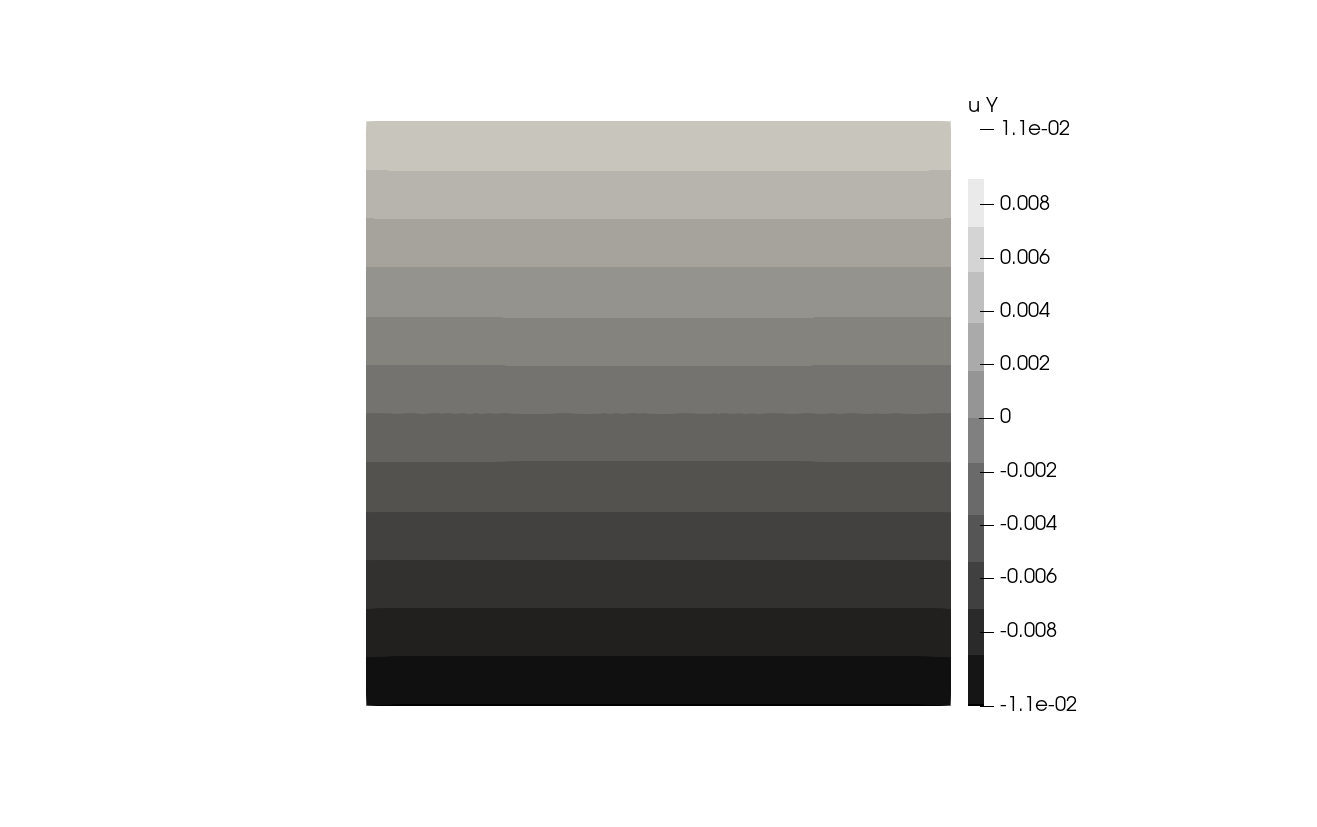}  
   \end{minipage}\hfill
   \begin{minipage}[t]{0.33\linewidth}
    \centering
    \includegraphics[trim=300 100 200 75, clip,width=\linewidth]{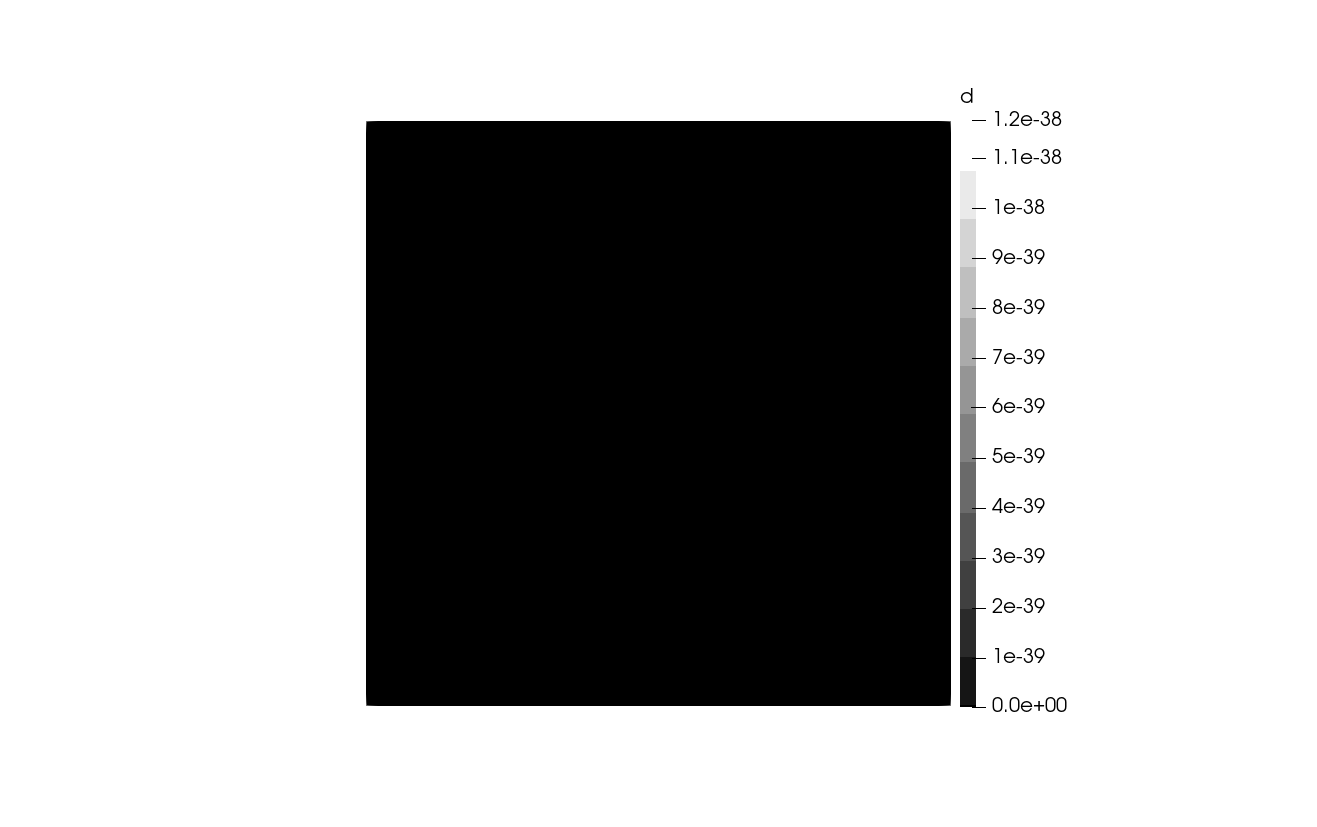}  
   \end{minipage}
   \captionof{figure}{\texttt{TC01S00}. Displacements in $X$- and $Y$-direction, damage after $10$ seconds.}
   \label{figure:results TC01S00}
  \end{center}
  A fact that is immediately apparent is that damage appears to be constant with value zero over space since it was set to zero for every spatial position and all steps in time.
  Thus the model should respond in every time step like a static equation of motion with respective data.
  This seems to be the case if one is looking at the displacements in $X$- and $Y$-direction. We observe symmetric displacements in both directions.
\begin{center}
  \begin{minipage}[t]{0.33\linewidth}
    \centering
    \includegraphics[trim=300 100 200 75, clip,width=\linewidth]{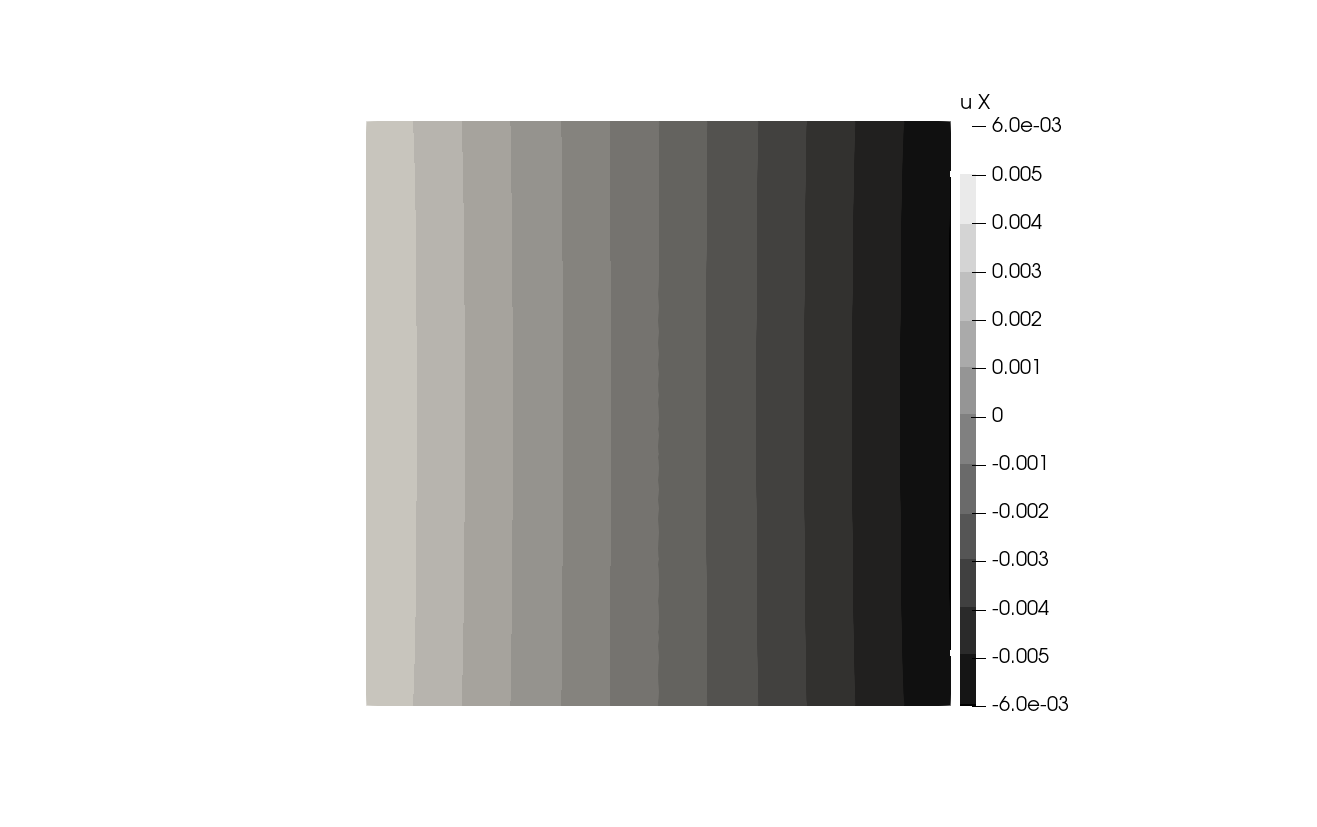}  
  \end{minipage}\hfill
  \begin{minipage}[t]{0.33\linewidth}
    \centering
    \includegraphics[trim=300 100 200 75, clip,width=\linewidth]{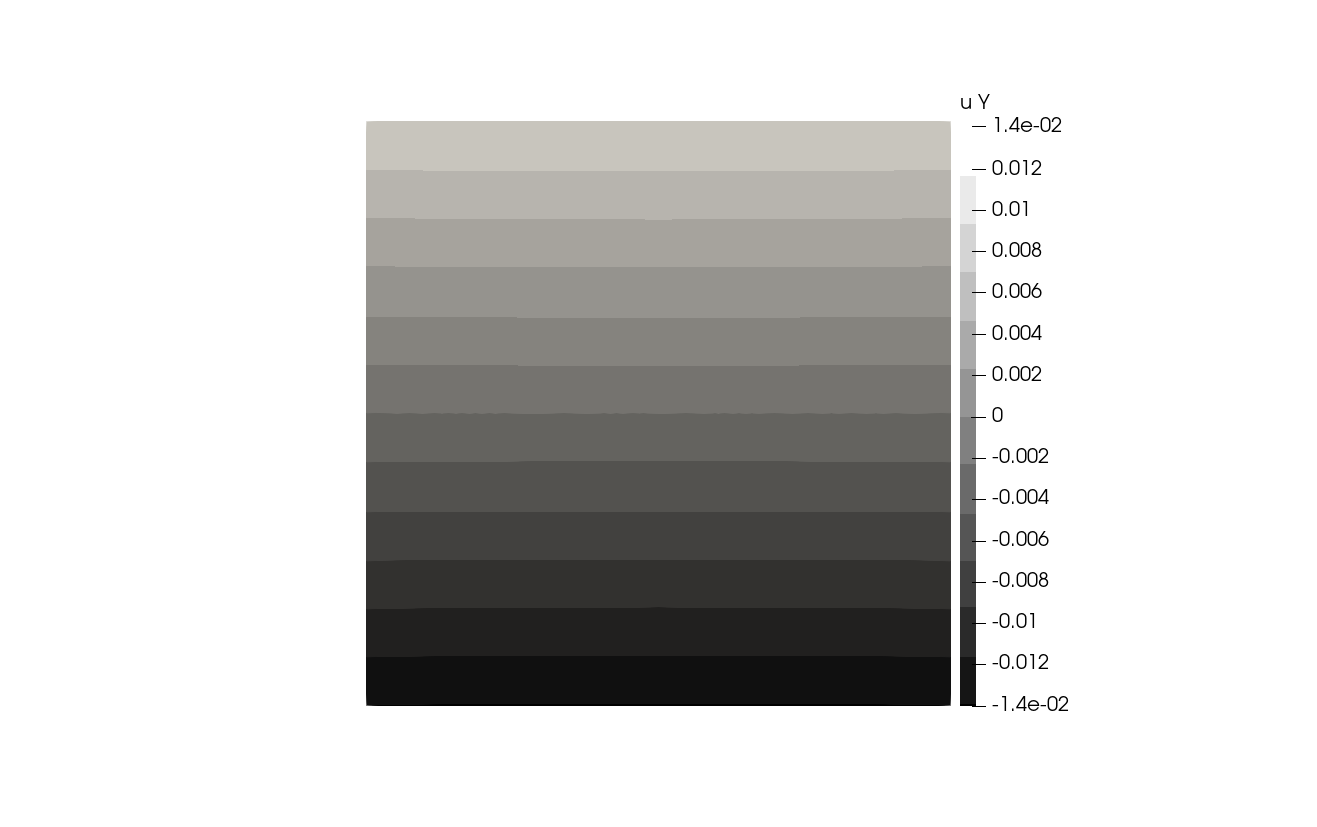}  
  \end{minipage}\hfill
  \begin{minipage}[t]{0.33\linewidth}
    \centering
    \includegraphics[trim=300 100 200 75, clip,width=\linewidth]{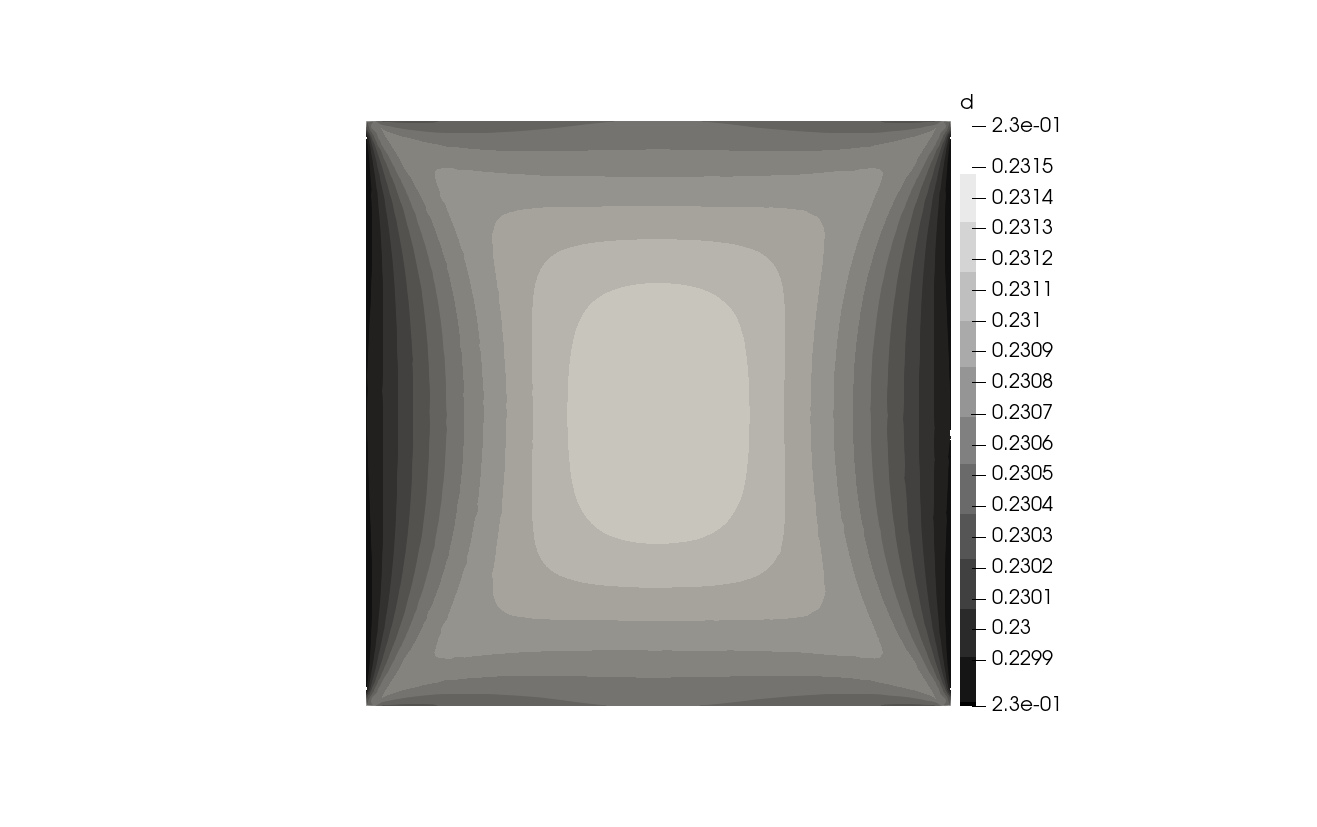}  
  \end{minipage}
  \captionof{figure}{\texttt{TC01S01}. Displacements in $X$- and $Y$-direction, damage after $10$ seconds.}
  \label{figure:results TC01S01}
\end{center}
Pulling on the top and lower ends with a force of same magnitude but opposite direction causes displacements $\displacements_Y$ as shown in the respective image.
The fact that $\displacements_Y(\spatialvariables)=-\displacements_Y(\boldsymbol{y})$ for all $\spatialvariables,\boldsymbol{y}\in\spatialdomain_0$ where $\spatialvariable_2,-y_2=\constant$ for $\constant\in[0,1]$ agrees with our physical expectation.
The same holds true for $\displacements_X$, i.e., $\displacements_X(\spatialvariables) = - \displacements_X(\boldsymbol{y})$ with $\spatialvariable_1,-y_1=\constant$.\\[2ex]
Figure \ref{figure:results TC01S01} show the fully-coupled model response applying the same data as in \texttt{TC01S00}.
The first thing that catches the eye is that damage is non zero.
Looking at the scale of the respective plot, we can see, that damage is practically constant over the spatial domain.
From a physical perspective this makes sense.
The material is in a homogenous state of stress.
According to Kachanov's model, stress is the root cause for damage eveolution in a homogeneous material, so damage should be homogeneous, too.\\[2ex]
We also see that the resposne in view of displacements exhibits the same characteristics as the one without damage evolution in \texttt{TC01S00} but the magnitudein displacements has increased.
Since damage is reducing the load-carrying capacity of the material the same force applied causes higher stresses and, thus, displacements in $X$- and $Y$-directions are of higher magnitude.
\begin{center}
  \begin{minipage}[t]{0.49\linewidth}
    \centering
    \includegraphics{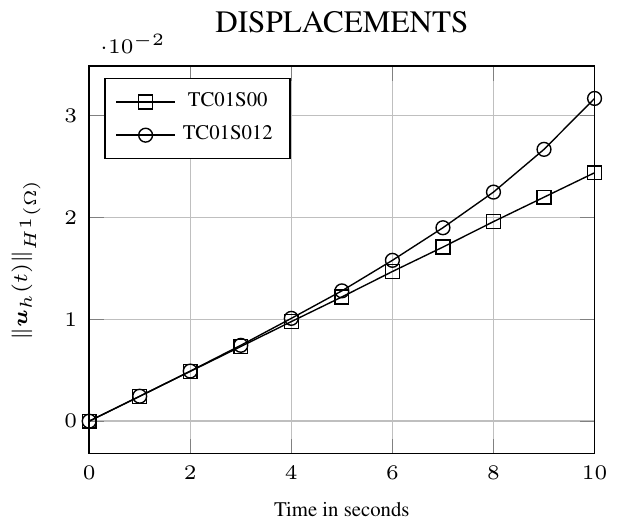}
  \end{minipage}\hfill
  \begin{minipage}[t]{0.49\linewidth}
    \centering
    \includegraphics{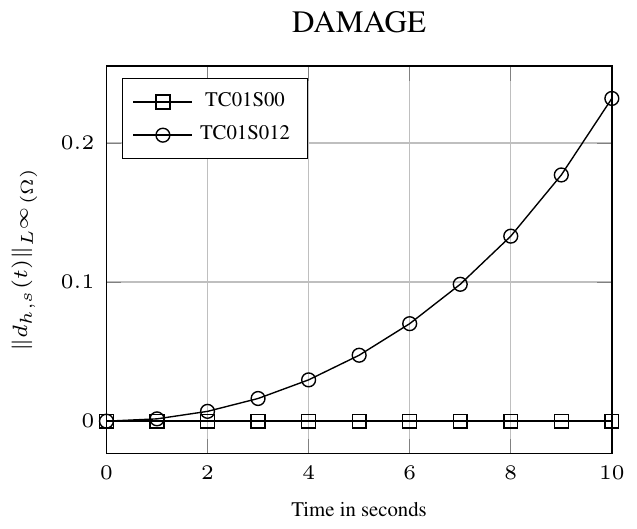}
  \end{minipage}
  \captionof{figure}{\texttt{TC01}. Comparing results.}
  \label{figure:comparing results TC01}
  \end{center}
We take a look in Figure \ref{figure:comparing results TC01} at the evolution of the respective norms over time.
The displacements measured in $\spatialhilbertfunctions^1$-norm of the damage-free model show a linear response to the applied linear force.
This agrees with the model assumptions.\\[2ex]
We tend to the Kachanov-like model, i.e., \texttt{TC01S01}.
Here, a nonlinear behavior in damage and in displacements can be observed.
As the load-carrying capacity is updated every step in time, the response of the equation of motion to the same forces changes and larger displacements and gradients are the result of that.
\subsubsection{Periodic loading in time}\label{sec:tc02}
We discuss the results from \texttt{TC02S00} and \texttt{TC02S01}.
In both these cases, we subjected the model to a load that is periodic in time.
\begin{center}
  \begin{minipage}[t]{0.33\linewidth}
    \centering
    \includegraphics[trim=300 100 200 75, clip,width=\linewidth]{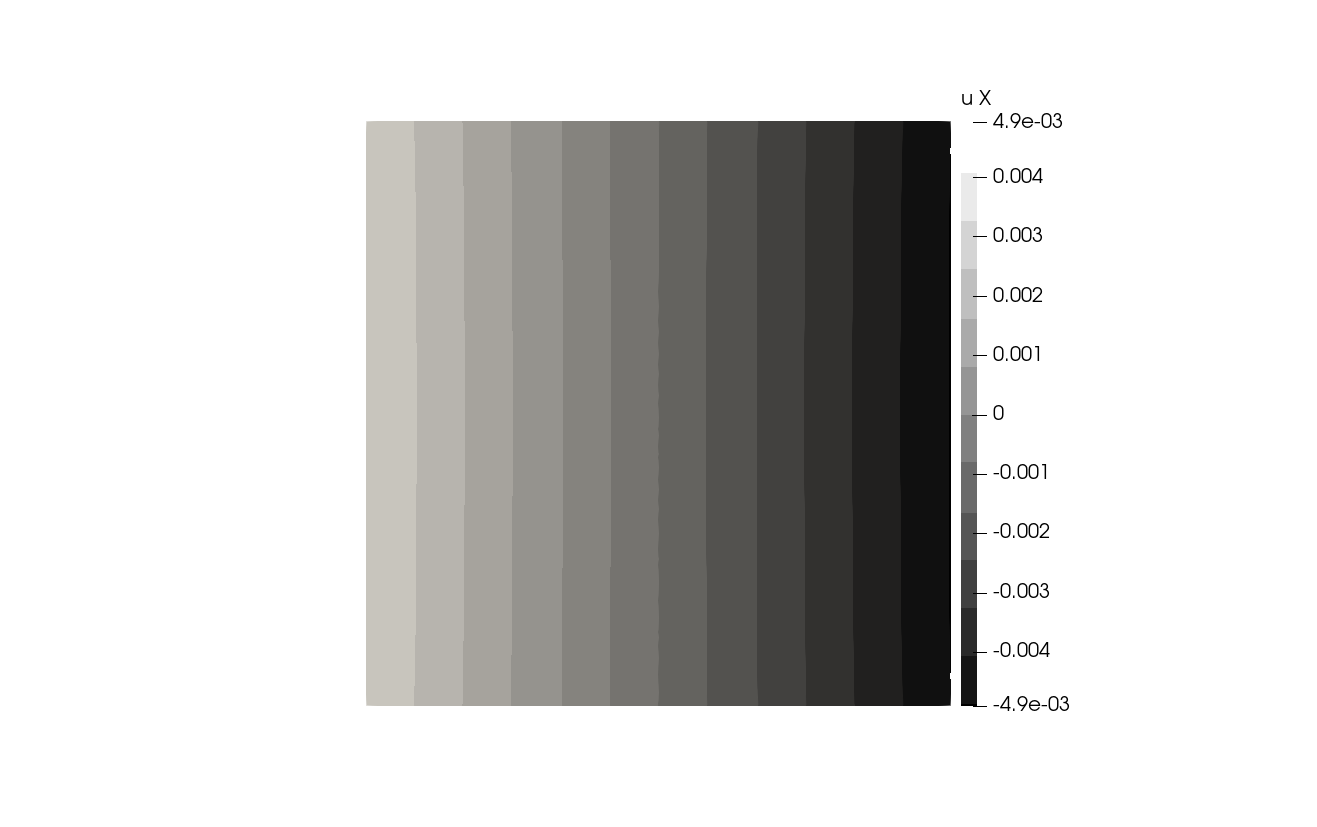}  
  \end{minipage}\hfill
  \begin{minipage}[t]{0.33\linewidth}
    \centering
    \includegraphics[trim=300 100 200 75, clip,width=\linewidth]{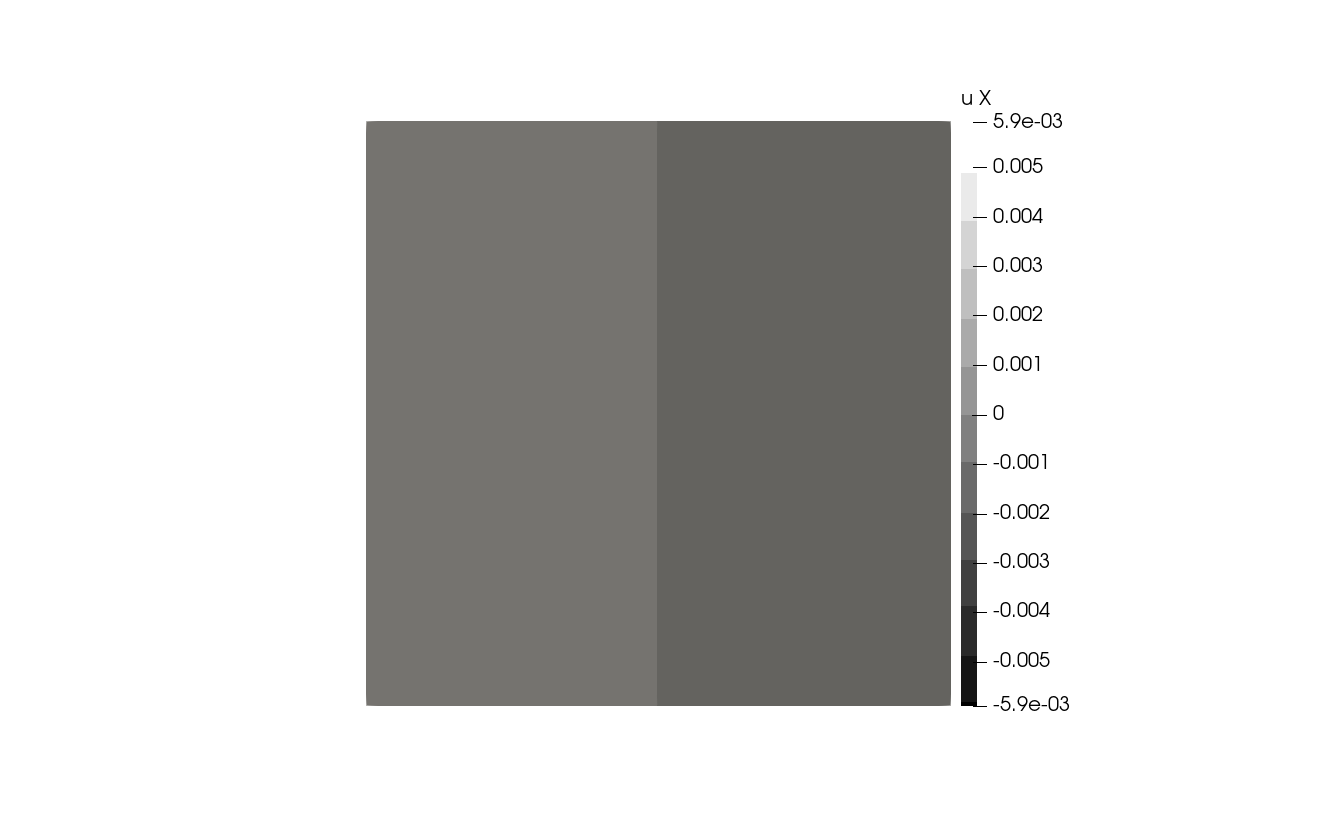}  
  \end{minipage}\hfill
    \begin{minipage}[t]{0.33\linewidth}
    \centering
  \includegraphics[trim=300 100 200 75, clip,width=\linewidth]{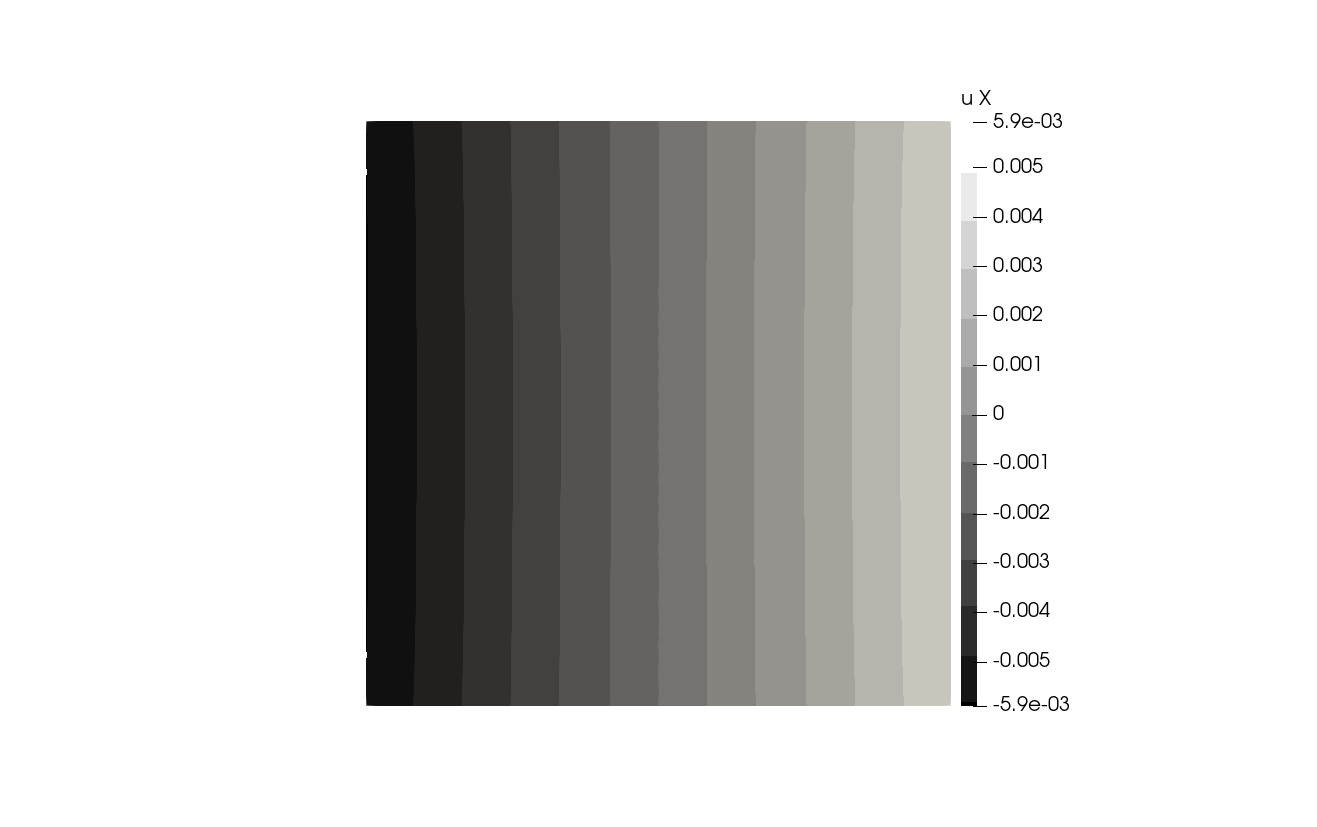}  
  \end{minipage}
  \captionof{figure}{\texttt{TC02S01}. Displacements in $X$- and $Y$-direction, damage after $2.5$, $5.0$, and $7.5$ seconds.}
  \label{figure:results TC02S00 in X}
\end{center}
Both models show qualitatively the same response as in \texttt{TC01}.
Keeping the shape of the dynamic load in mind, the force emposes a tension phase followed by a compression phase.
We can see in Figure \ref{figure:results TC02S00 in X} that displacements $\displacements_X$ show the anticipated behavior.
\begin{center}
  \begin{minipage}[t]{0.33\linewidth}
    \centering
    \includegraphics[trim=300 100 200 75, clip,width=\linewidth]{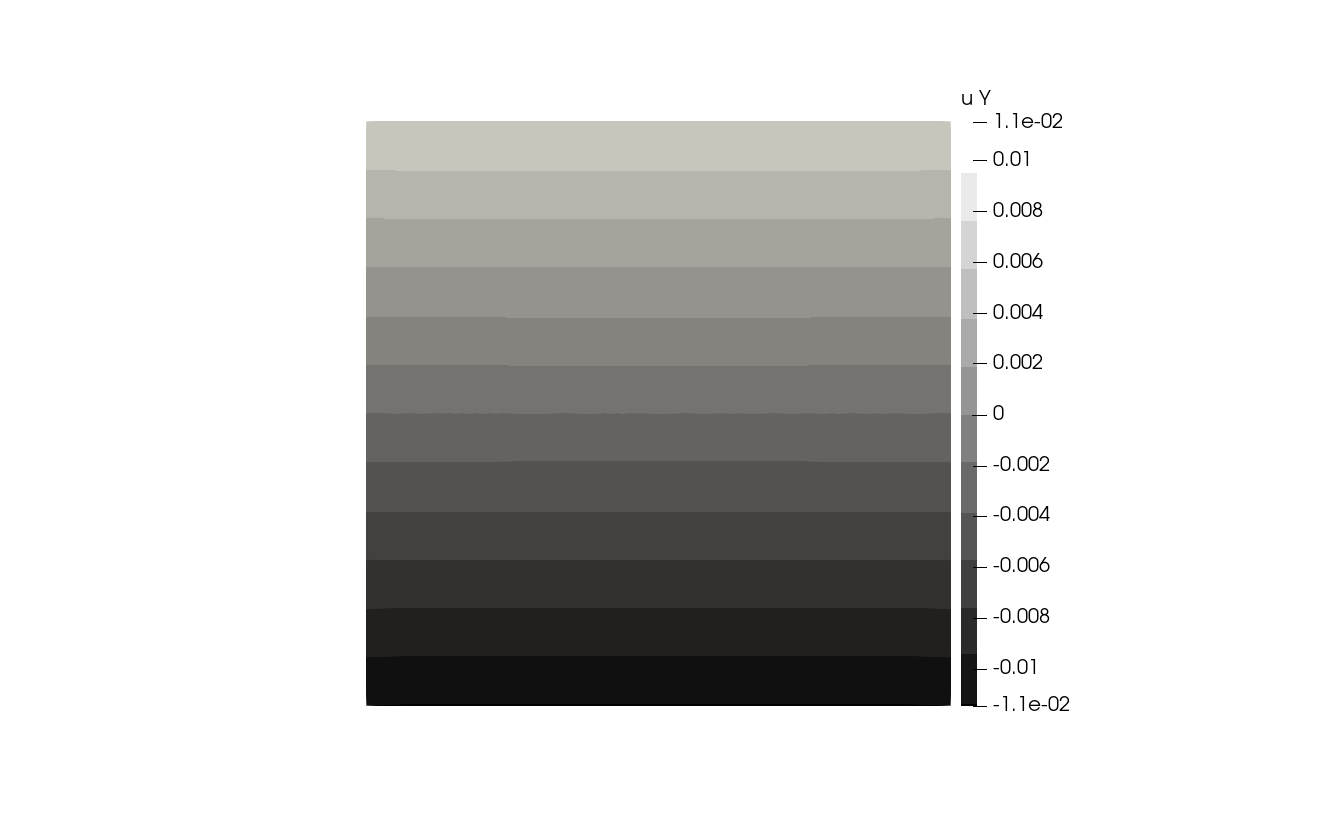}  
  \end{minipage}\hfill
  \begin{minipage}[t]{0.33\linewidth}
  \centering
    \includegraphics[trim=300 100 200 75, clip,width=\linewidth]{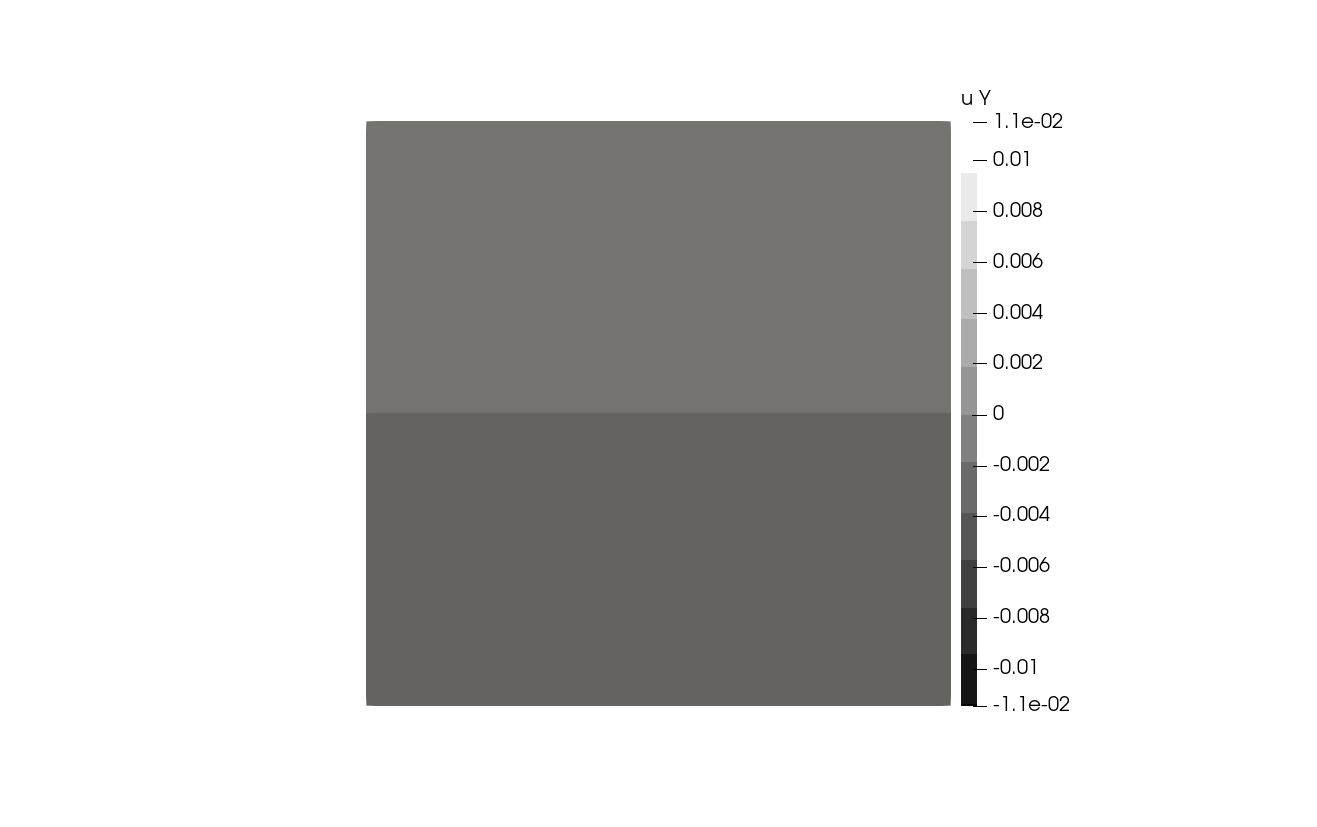}  
    \end{minipage}\hfill
  \begin{minipage}[t]{0.33\linewidth}
    \centering
    \includegraphics[trim=300 100 200 75, clip,width=\linewidth]{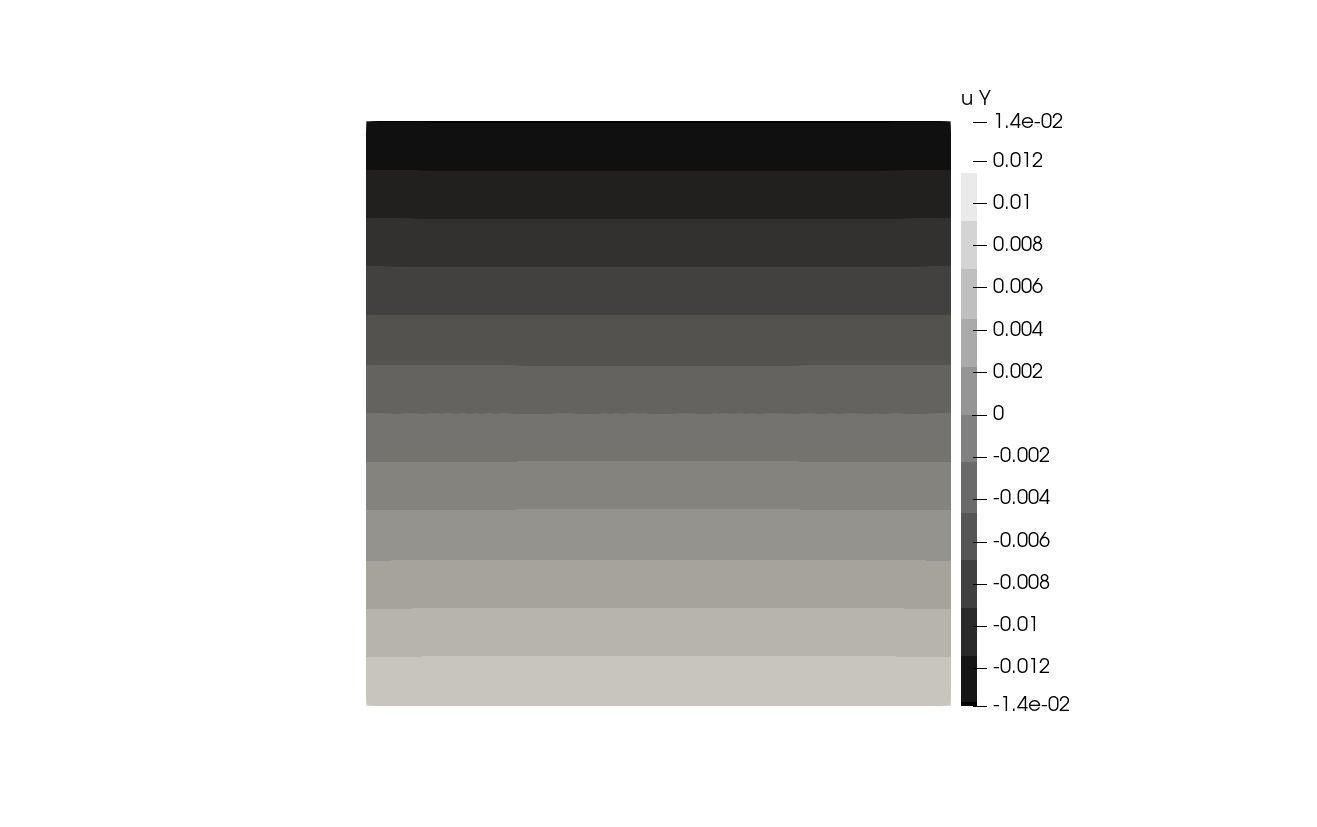}  
  \end{minipage}
  \captionof{figure}{\texttt{TC02S01}. Displacements in $X$- and $Y$-direction, damage after $2.5$, $5.0$, and $7.5$ seconds.}
  \label{figure:results TC02S01 in Y}
\end{center}
During tension the material points move towards the axis in timestep $\timevariable=2.5$, decrease in the tension's realese phase to zero in $\timevariable=5.0$, and move away from the middle axis during compression at $\timevariable=7.5$.
Qualitatively, the same behavior can be observed for displacements $\displacements_Y$.
Keeping in mind that our equation of motions models a linear setting and does not account for plasticity, this is the model response we expect physically.
Note that we did not present the respective results from the damage-free response in \texttt{TC02S00} as this only differs in magnitude from case \texttt{TC02S01}.
\begin{center}
  \begin{minipage}[t]{0.49\linewidth}
   \centering
   \includegraphics{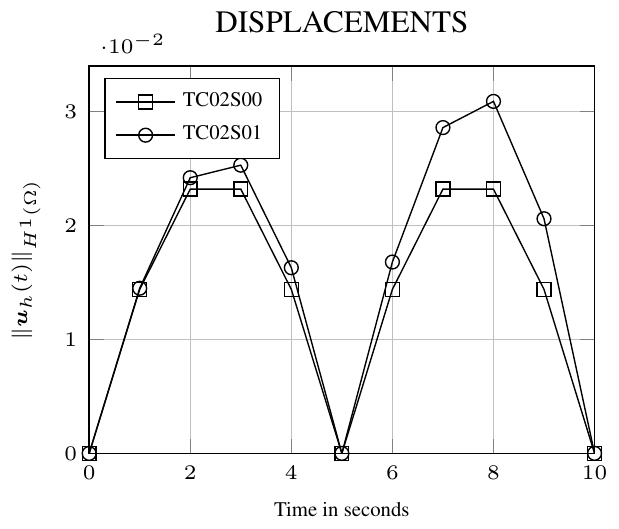}
  \end{minipage}\hfill
  \begin{minipage}[t]{0.49\linewidth}
   \centering
   \includegraphics{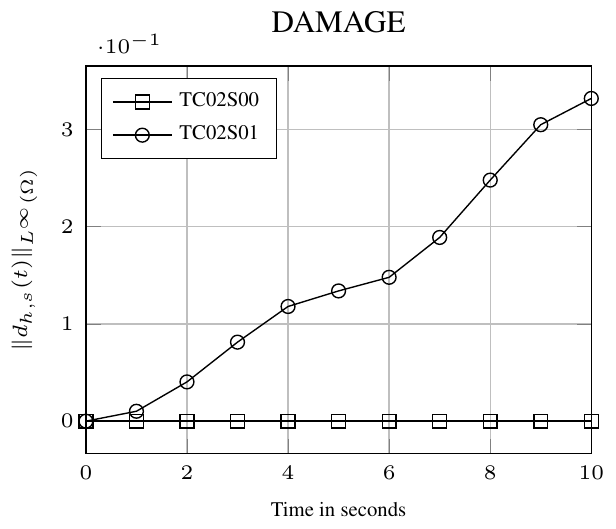}
  \end{minipage}
  \captionof{figure}{Comparing results \texttt{TC02}.}
  \label{figure:comparing results TC02}
\end{center}
Looking at the norm plots in Figure \ref{figure:comparing results TC02} we see that damage has now a different non linear shape.
This is again due to the fact that damage is only driven by stresses.
Applying a periodic force causes periodic stresses.
As these transfer directly to the damage evolution, we see the response shown on the left image in Figure \ref{figure:comparing results TC02}.
Note that damage is modeled as a non-decreasing variable.\\[2ex]
Keeping damage evolution in mind and with it the reduced load-carrying capacity over time, the left plot in Figure \ref{figure:comparing results TC02} immediately agrees with our physical expectation.
The increased damage causes higher displacements at the same level of force.
This is why the $\spatialhilbertfunctions^1$-norm of the displacements for the fully coupled model is higher than in the damage-free setting and also further increases during cyclic loading.
\subsection{Singularities}\label{sec:tc03}
We pick up our discussion on boundary regularity that we started in the first part of this work and to which we contributed with physically motivated examples in Section \ref{Sec:Corners}.
We first discuss the effect this has on the Kachanov-type model.
Afterwards, we shed some light on the proposed options to remedy the situation.
Namely, changing the type of boundary conditions or the computational spatial domain.
Note that we chose the same values for boundary stresses and boundary displacements for all of the following test cases.
To keep the cases comparable, these values had to be high enough to show the desired effect.
Therefore substantial damage was reached in almost all cases before $10$ seconds were reached.
\subsubsection{Dirichlet-type boundary conditions}\label{sec:tc03_1}
Figure \ref{figure:results TC03S01} shows the response of the fully coupled model with homogeneous Dirichlet-type boundary conditions on the bottom and linearly increasing displacements on the top end, i.e., $\displacements=\boundarydisplacements_2$ on $\boundary_2$.
\begin{center}
  \begin{minipage}[t]{0.33\linewidth}
   \centering
   \includegraphics[trim=300 100 200 75, clip,width=\linewidth]{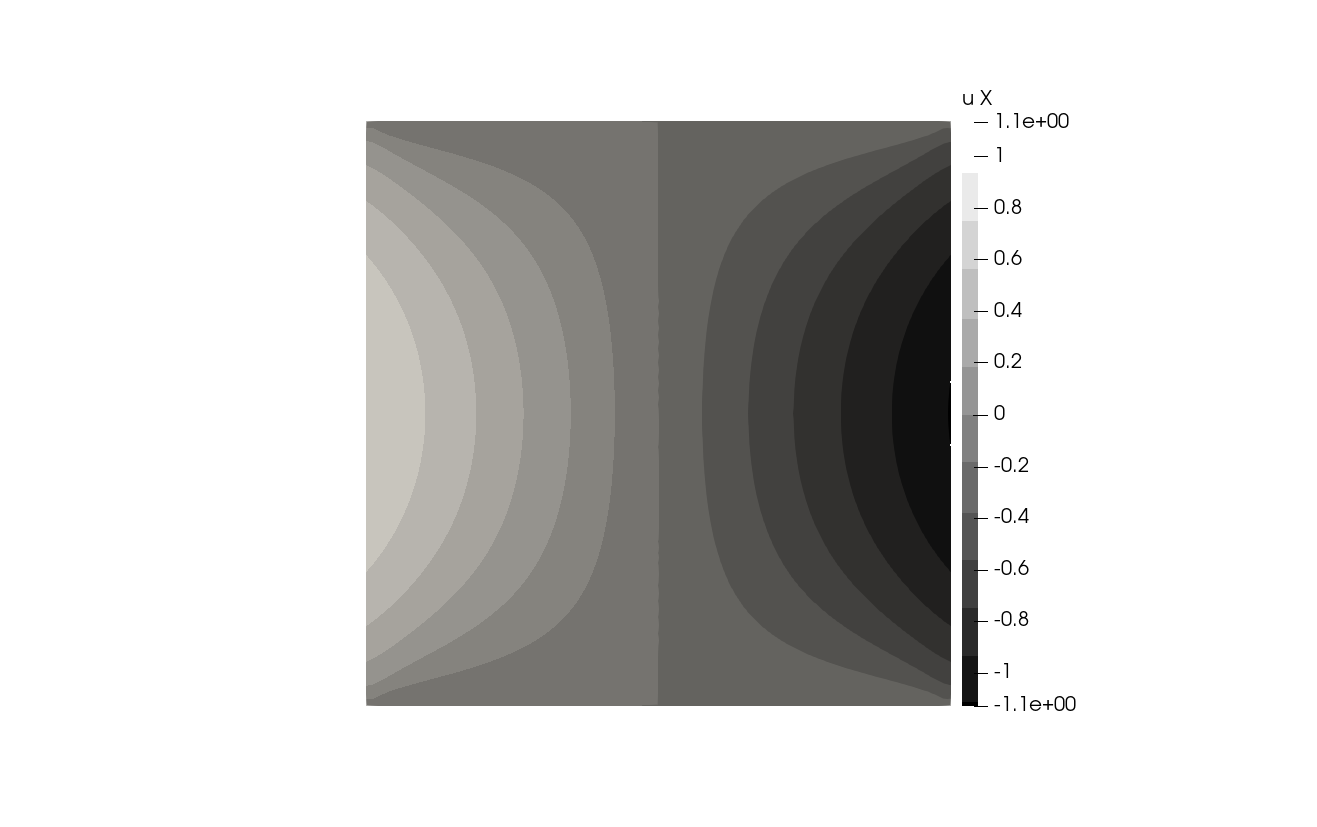}  
  \end{minipage}\hfill
  \begin{minipage}[t]{0.33\linewidth}
   \centering
   \includegraphics[trim=300 100 200 75, clip,width=\linewidth]{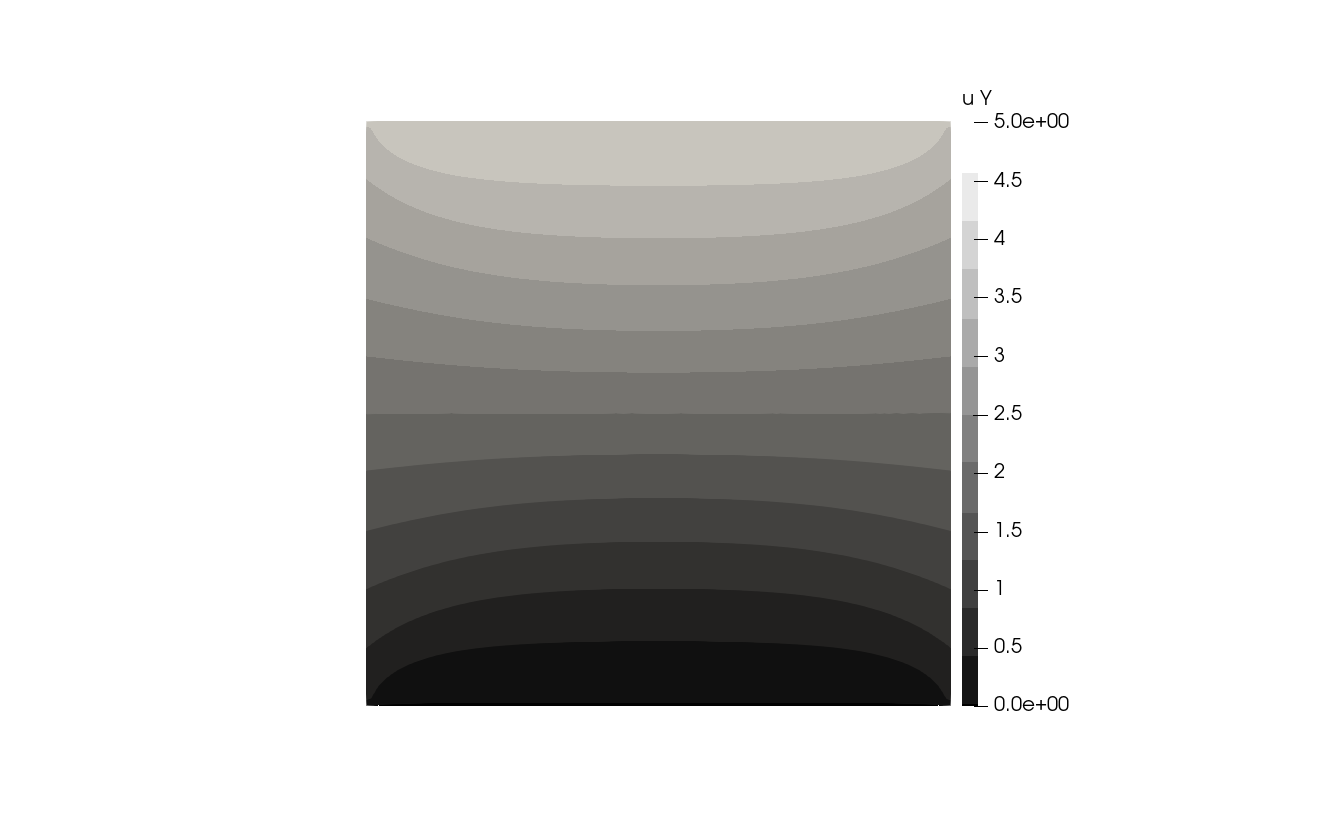}  
  \end{minipage}\hfill
  \begin{minipage}[t]{0.33\linewidth}
   \centering
   \includegraphics[trim=300 100 200 75, clip,width=\linewidth]{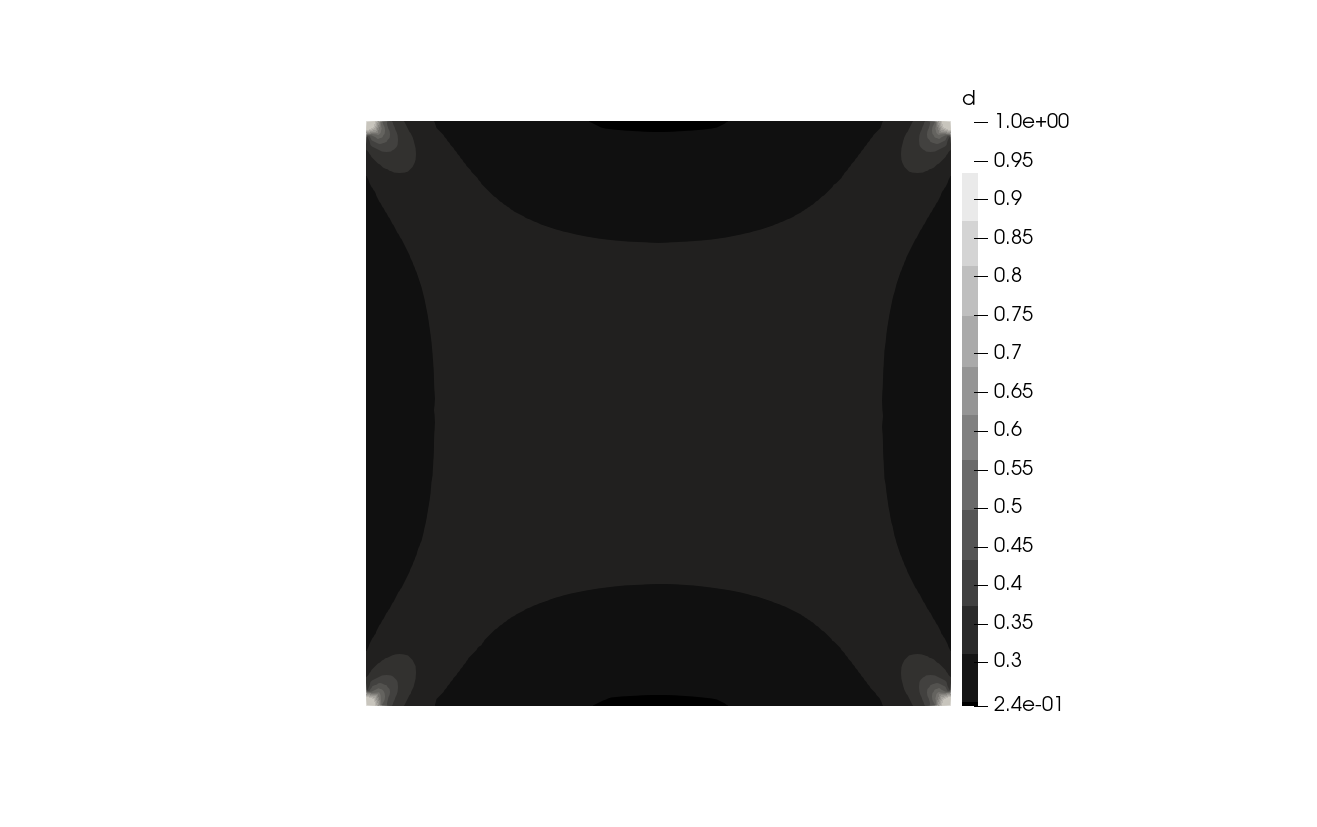}  
  \end{minipage}
  \captionof{figure}{\texttt{TC03S01}. Displacements in $X$- and $Y$- direction, damage after $9.8$ seconds.}
  \label{figure:results TC03S01}
\end{center}
The distribution of displacements vastly differs from the ones we looked at before.
Although the neutral axis is the same as before, the fixed displacements on $\boundary_0$ and $\boundary_2$ cause a form of necking that is different to the one we saw before.
Particularly, the observed symmetry in the previous test cases is no longer present.
Looking at the spatial distribution of damage, we see that the maximum is obtained in the domain's corners (see Section \ref{Sec:Corners}).
The damage is substantially higher and we now longer have a homogeneous spatial distribution.
\begin{center}
  \begin{minipage}[t]{0.49\linewidth}
   \centering
   \includegraphics{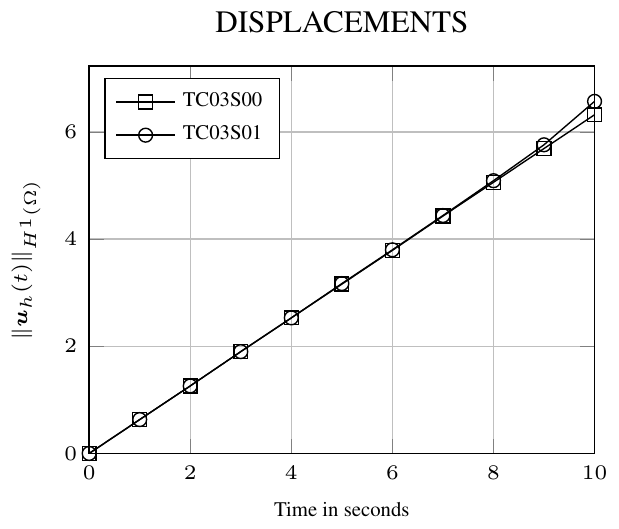}
  \end{minipage}\hfill
  \begin{minipage}[t]{0.49\linewidth}
   \centering
   \includegraphics{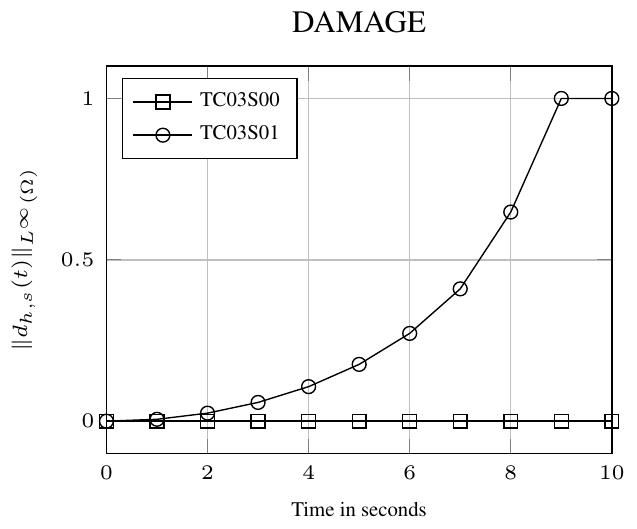}
  \end{minipage}
  \captionof{figure}{Comparing results \texttt{TC03}}
  \label{figure:comparing results TC03}
\end{center}
These high stresses are very localized.
This causes damage to rapidly increase if measured with a $\integrablefunctions^{\infty}$-norm, see Figure \ref{figure:comparing results TC03} on the right.
But as we can see on the left side, this has barely any effect on the displacements.
The clear nonlinear response from the previous test cases is barely visible.
\subsubsection{Robin-type boundary conditions}
In Figure \ref{figure:results TC03S02} we can see the effect of Robin-type boundary conditions.
The results are somewhat inbetween the responses provided by applying Dirichlet- or Neumann-type boundary conditions.
\begin{center}
  \begin{minipage}[t]{0.33\linewidth}
   \centering
   \includegraphics[trim=300 100 200 75, clip,width=\linewidth]{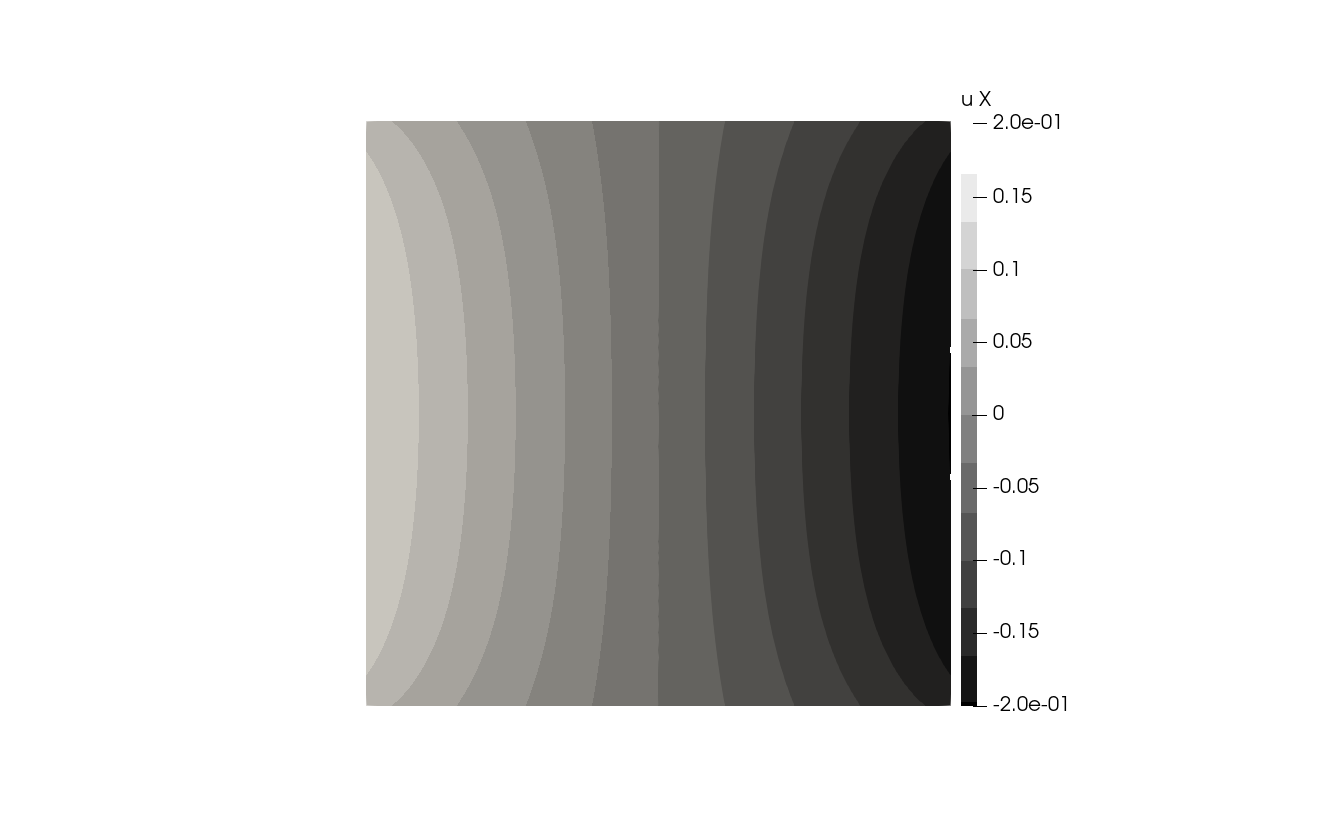}  
  \end{minipage}\hfill
  \begin{minipage}[t]{0.33\linewidth}
   \centering
   \includegraphics[trim=300 100 200 75, clip,width=\linewidth]{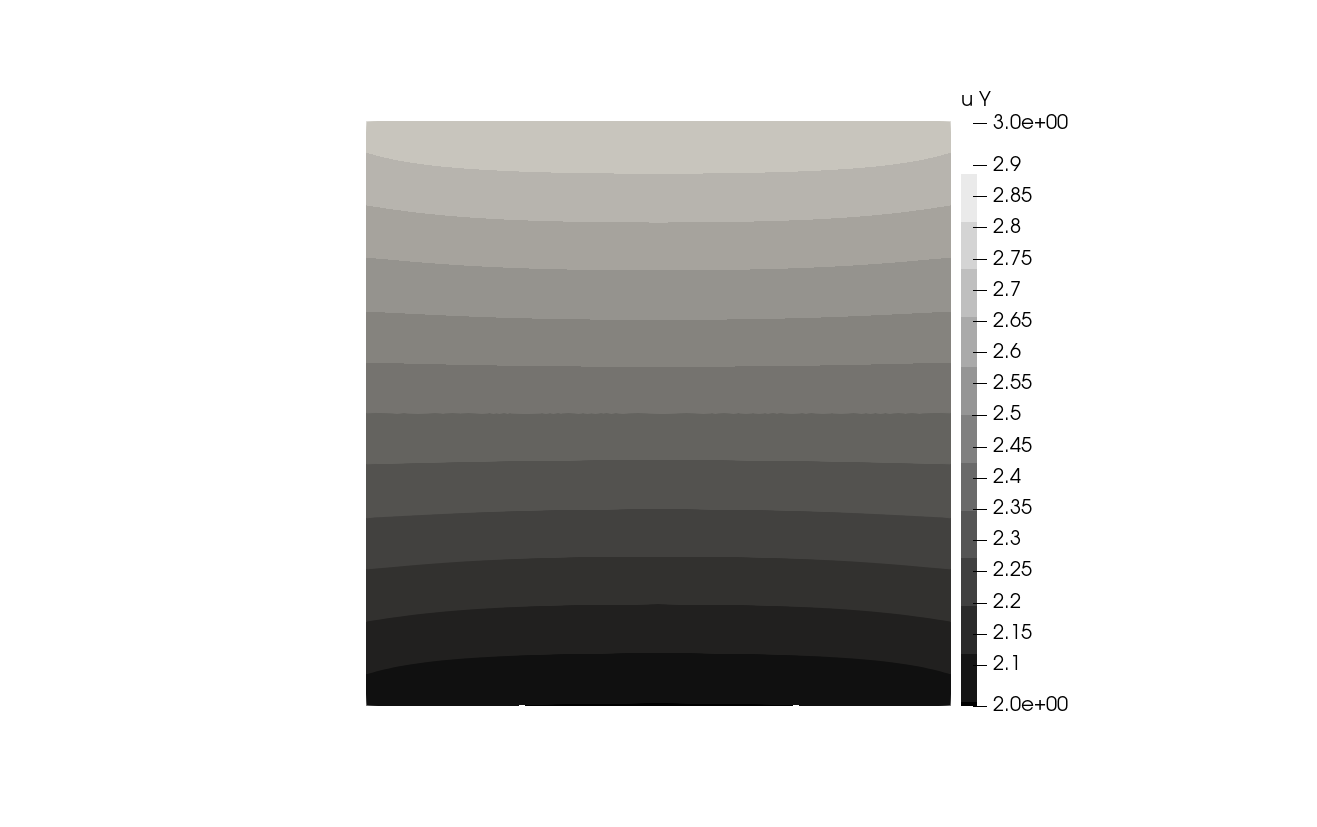}  
  \end{minipage}\hfill
  \begin{minipage}[t]{0.33\linewidth}
   \centering
   \includegraphics[trim=300 100 200 75, clip,width=\linewidth]{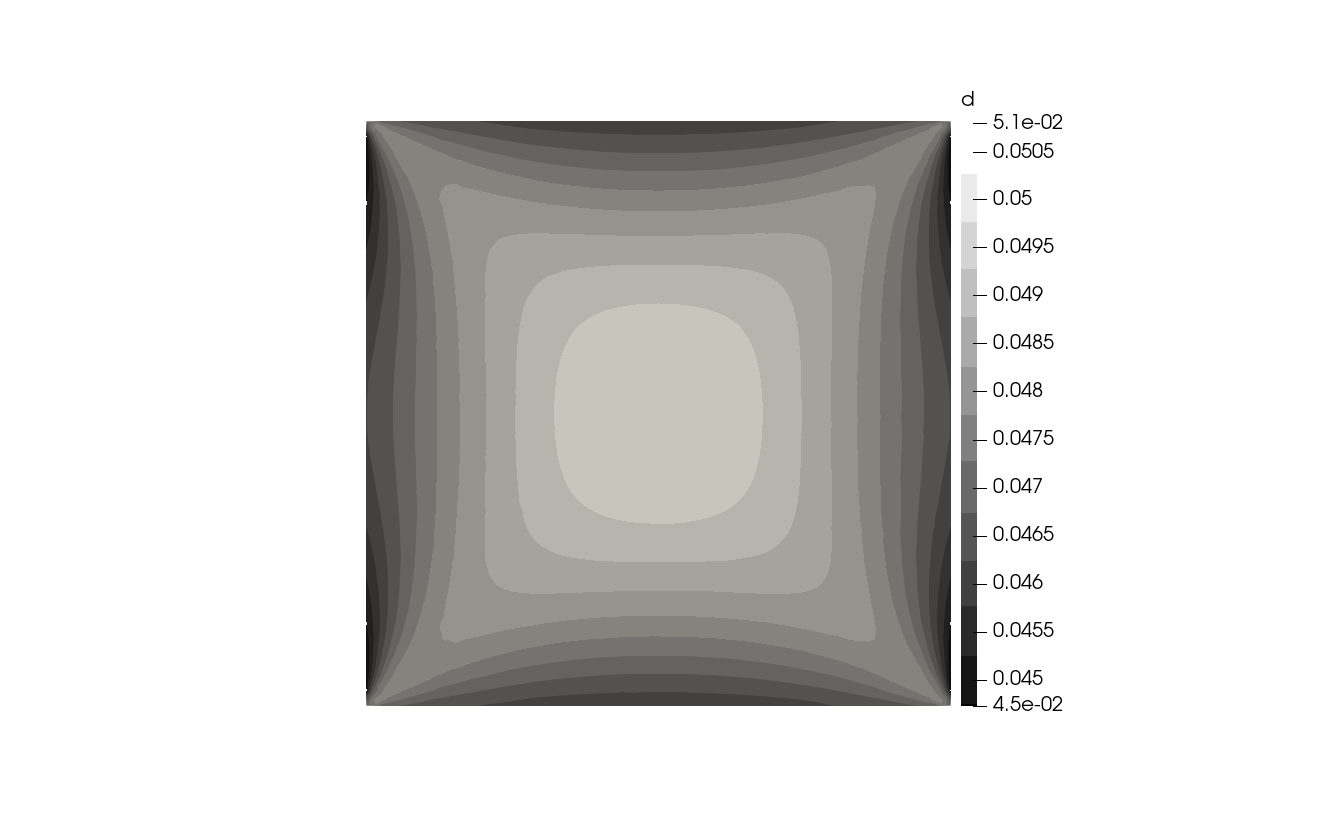}  
  \end{minipage}
  \captionof{figure}{\texttt{TC03S02}. Displacements in $X$- and $Y$- direction, damage after $10.0$ seconds.}
  \label{figure:results TC03S02}
\end{center}
Damage is also spatially distributed but within a vastly smaller range.
The stresses are relaxed a lot since particularly intersection points are not strictly bound to their position anymore.
Looking at the displacements we can see this effect of relaxing the condition on boundary displacements, too.
\subsubsection{Different domains}
We want to investigate two different settings numerically.
On the one hand, want loot at a domain with two connected components and assess the effect this has on the model response.
On the other hand, we look at the effect a smoother boundary has on the results.\\[2ex]
In Figure \ref{figure:results TC03S03} we see the results from \texttt{TC03S03}.
When only focusing on the part between the two inner boundaries, we see qualitatively a very similar picture to the Dirichlet case from before.
Note that the plots are the last time step before the maximum damage was attained.
This causes those small distortions in the displacements' distribution.
\begin{center}
  \begin{minipage}[t]{0.33\linewidth}
   \centering
   \includegraphics[trim=300 100 200 75, clip,width=\linewidth]{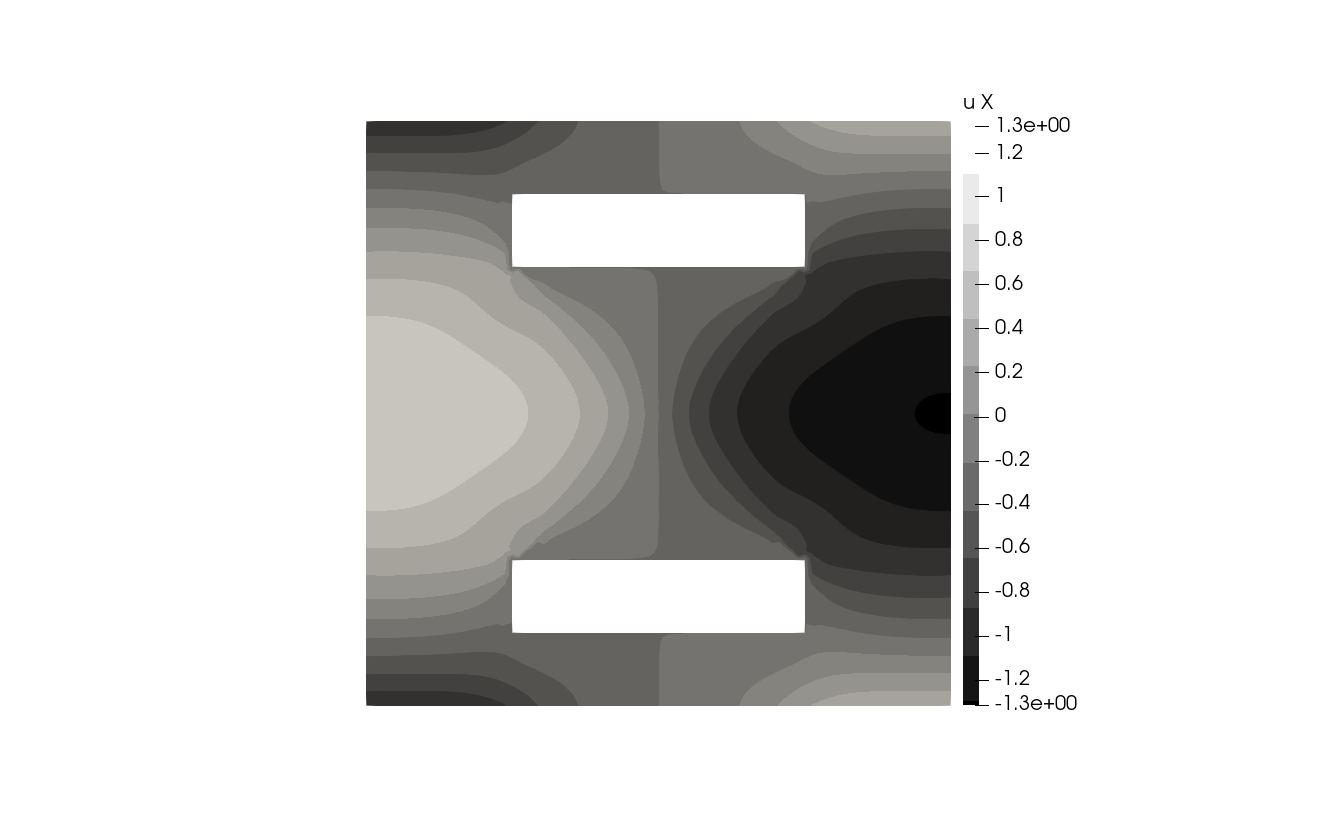}  
  \end{minipage}\hfill
  \begin{minipage}[t]{0.33\linewidth}
   \centering
   \includegraphics[trim=300 100 200 75, clip,width=\linewidth]{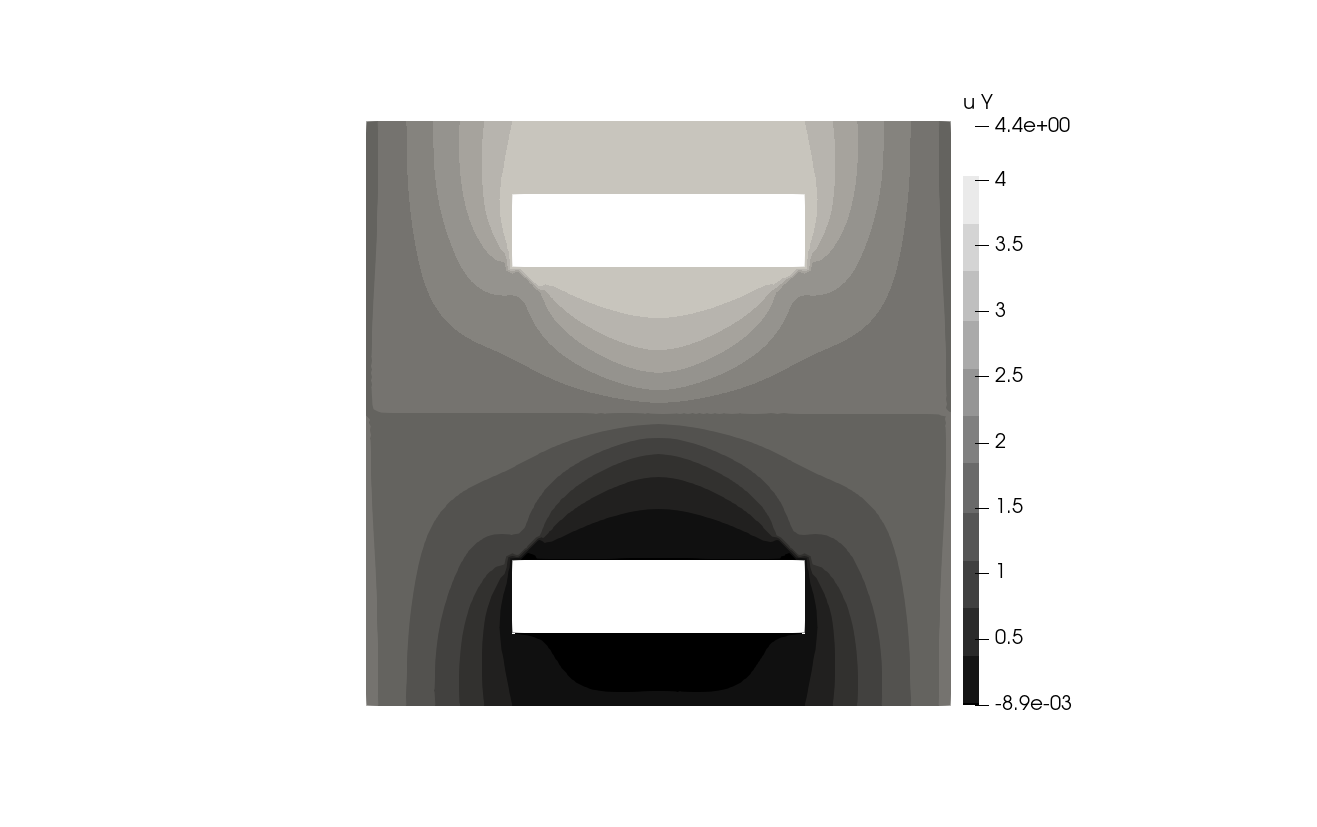}  
  \end{minipage}\hfill
  \begin{minipage}[t]{0.33\linewidth}
   \centering
   \includegraphics[trim=300 100 200 75, clip,width=\linewidth]{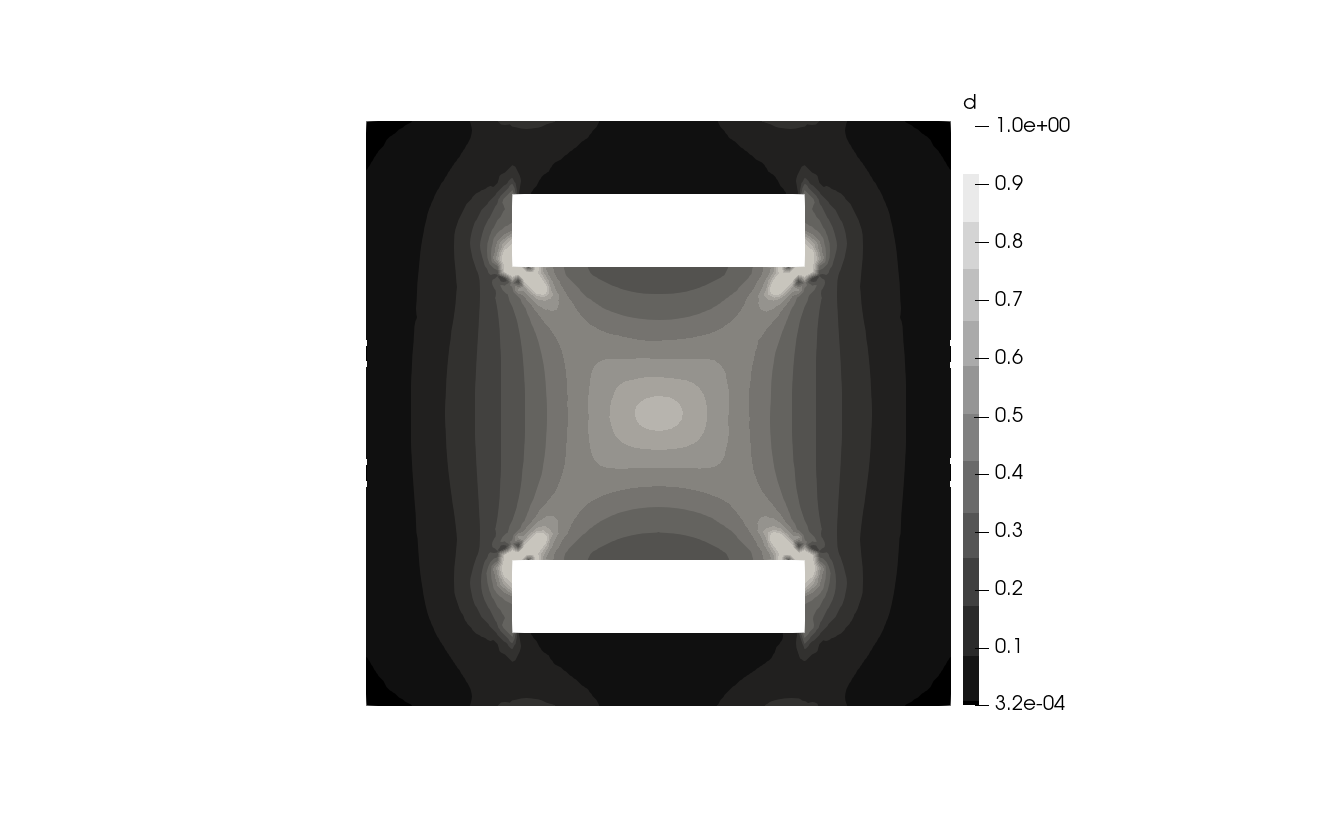}  
  \end{minipage}
  \captionof{figure}{\texttt{TC03S03}. Displacements in $X$- and $Y$- direction, damage after $8.8$ seconds.}
  \label{figure:results TC03S03}
\end{center}
In the area between the inner boundaries, we also have a very similar distribution for damage compared to the one from before.
But a major difference is that damage is not as localized as it was before.
This strongly impacts the overall response which we circle back to in the discussion part at the end.
\begin{center}
 \begin{minipage}[t]{0.33\linewidth}
  \centering
  \includegraphics[trim=300 100 200 75, clip,width=\linewidth]{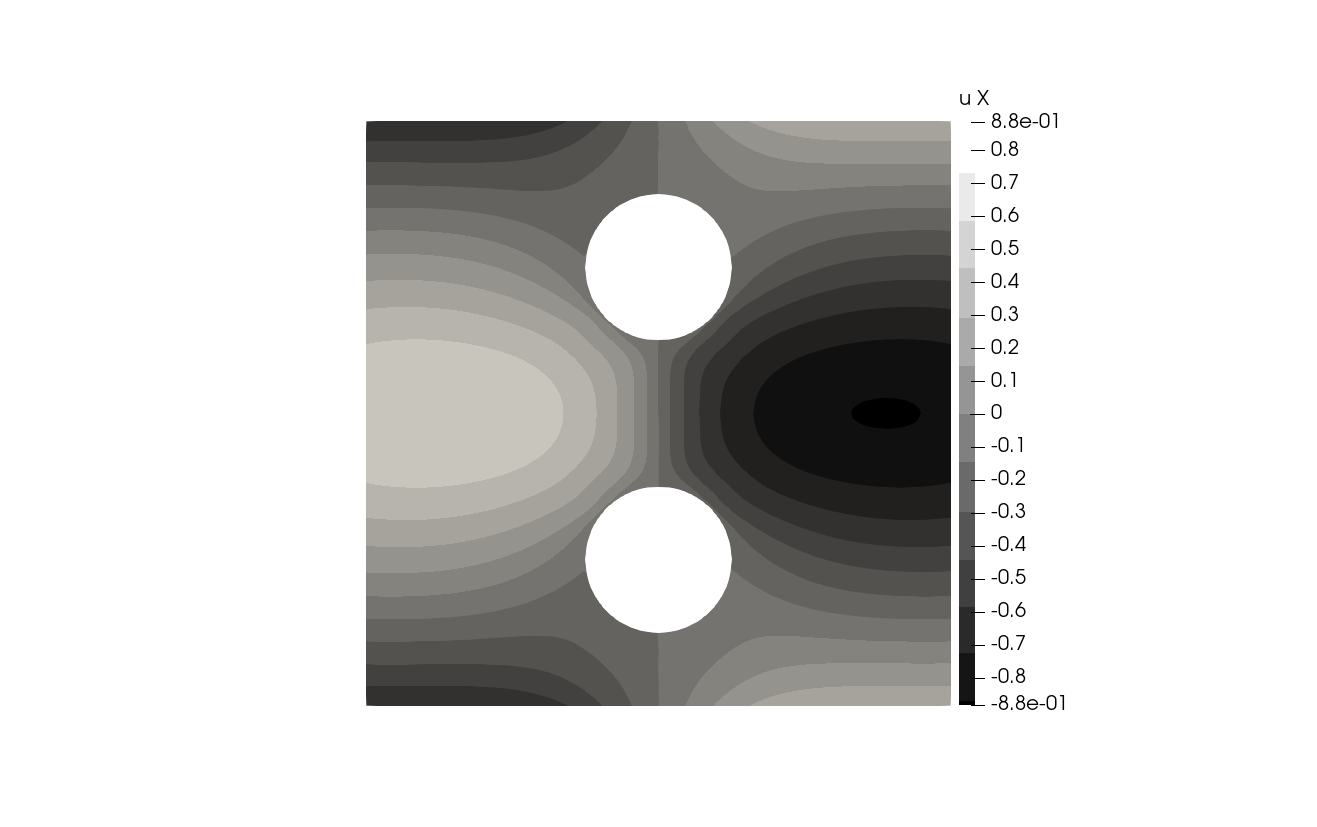}  
 \end{minipage}\hfill
 \begin{minipage}[t]{0.33\linewidth}
  \centering
  \includegraphics[trim=300 100 200 75, clip,width=\linewidth]{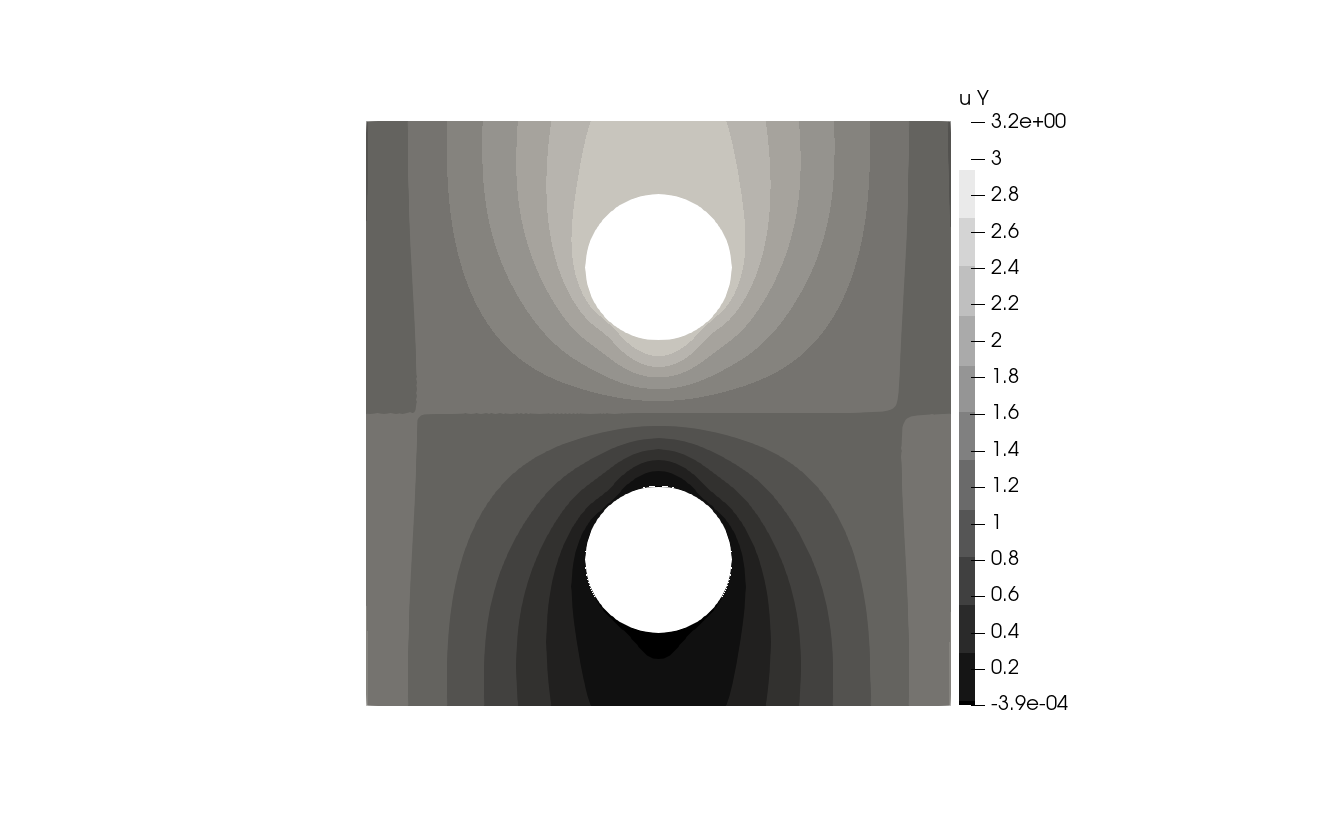}  
 \end{minipage}\hfill
 \begin{minipage}[t]{0.33\linewidth}
  \centering
  \includegraphics[trim=300 100 200 75, clip,width=\linewidth]{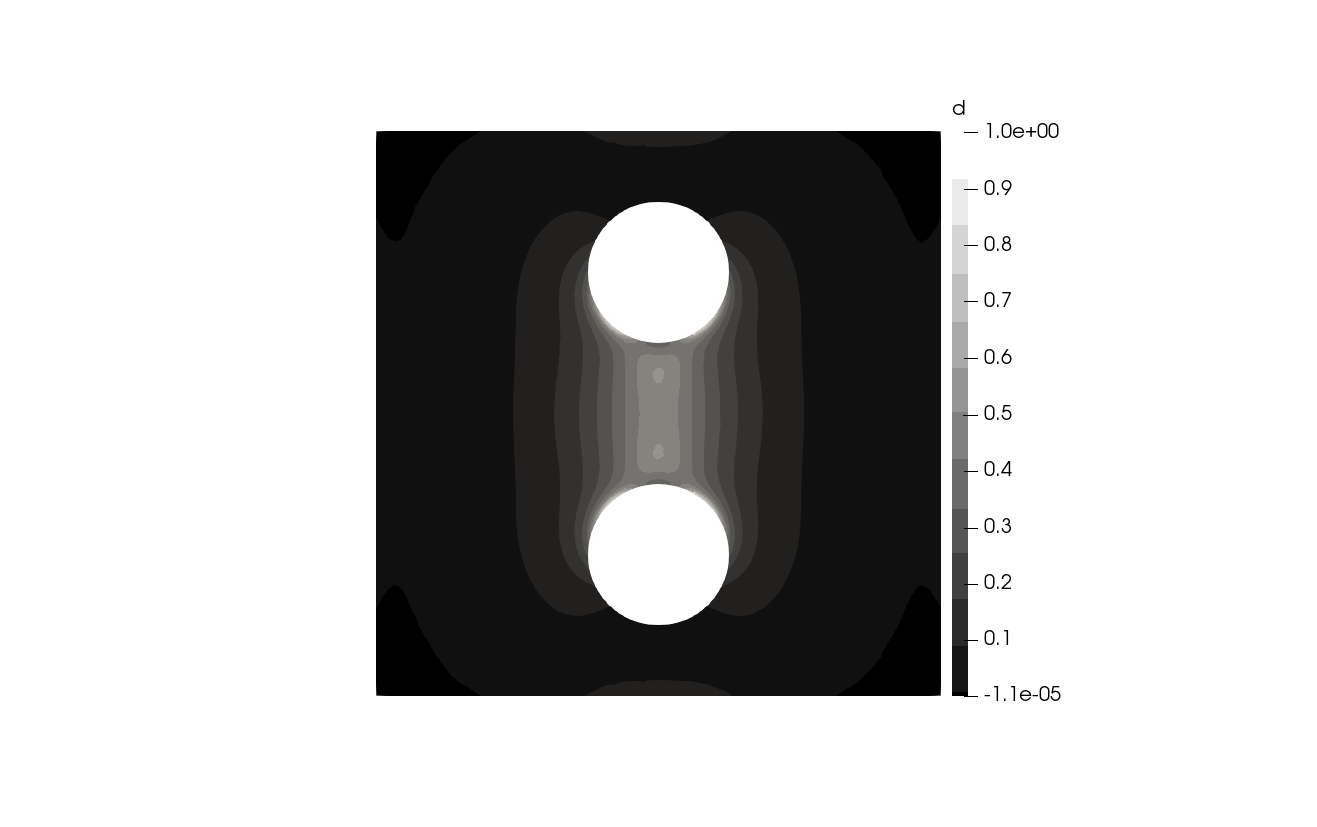}  
 \end{minipage}
 \captionof{figure}{\texttt{TC03S04}. Displacements in $X$- and $Y$- direction, damage after $6.4$ seconds.}
 \label{figure:results TC03S04}
\end{center}
Figure \ref{figure:results TC03S04} displays the same situation but with a much smoother boundary.
The effect on damage and displacements is much stronger than before.
Looking at only a thin stripe along the respective middle axis shows a similar behavior to the previous cases where we applied Dirichlet conditions.\\[2ex]
When focusing on damage, the maximum is still on or close to the boundaries, but not as strongly localized as before.
\subsubsection{Discussion}
Figure \ref{figure:comparing results TC03SUM} shows the results of all test cases from this subsection.
The plot with diamond marks depicts \texttt{TC03S01}, i.e., Dirichlet conditions on the unit square $\spatialdomain_0$.
Although we see substantial damage measured in the respective norm, the norm of the displacements is barely affected.
\begin{center}
  \begin{minipage}[t]{0.49\linewidth}
   \centering
   \includegraphics{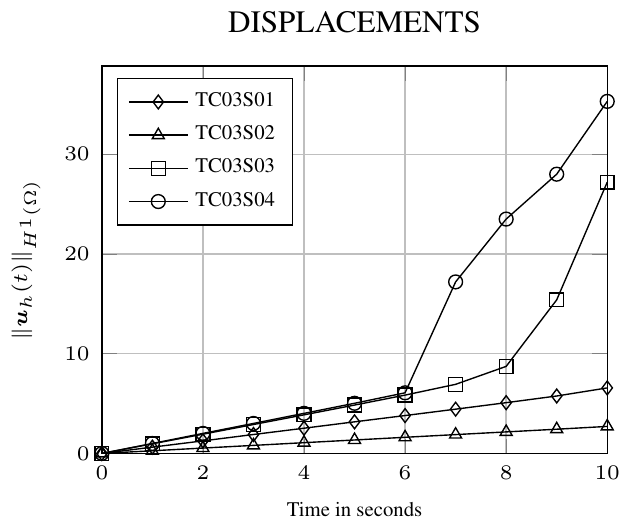}
  \end{minipage}\hfill
  \begin{minipage}[t]{0.49\linewidth}
   \centering
   \includegraphics{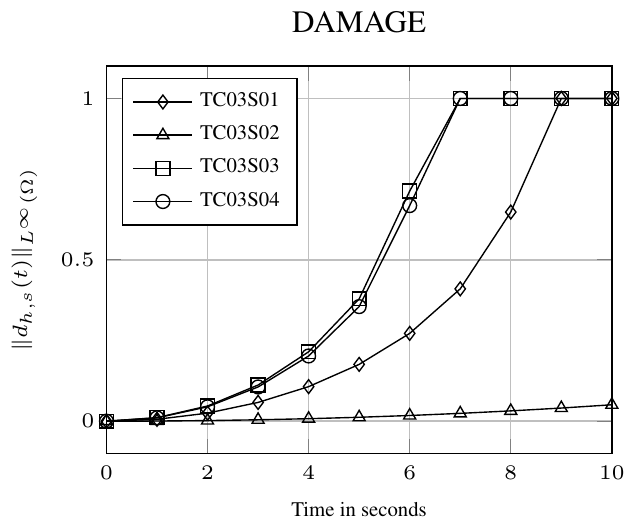}
  \end{minipage}
  \captionof{figure}{Comparing results \texttt{TC03}.}
  \label{figure:comparing results TC03SUM}
\end{center}
Test case \texttt{TC03S02}, where Robin conditions were used to relax the Dirichlet condition, shows very little damage.
This stems from the fact that the stresses were heavily reduced and thus the nonlinear effect is almost not visible.\\[2ex]
The remaining two cases behave very differently compared to the others.
For instance, damage evolution is much higher for both of them than before.
The main reasons for this are the following.
The different shapes of both boundaries for \texttt{TC03S03} and \texttt{TC03S04} and the remaining effective material between both boundaries are very different to \texttt{TC03S01} and damage distribution in this area is a lot less localized than in \texttt{TC03S01}.
Therefore high damage values are not only attained at the corners, but also in the area in between the boundaries and therefore adding to the $\integrablefunctions^{\infty}$-norm.
This also explains the higher effect on the displacements.\\[2ex]
Comparing only \texttt{TC03S03}, \texttt{TC03S04} to each other, we see that damage evolution is very similar, but norms of displacements behave variously.
In view of damage evolution, the small difference stems from the localized stresses in the corners, which are higher in \texttt{TC03S04} than in \texttt{TC03S04}.
The difference in norms of the displacements can be explained by the different shape of the damage variable's spatial distribution between the inner boundaries.
\subsection{Substantial damage}\label{sec:tc04}
In general, substantial damage means that material bonds have been fully desintegrated and a macroscopic crack is visible.
In our case, this translates to damage being $\damage = 1$ at some point in $\spatialdomain$.
\begin{center}
  \begin{minipage}[t]{0.33\linewidth}
   \centering
   \includegraphics[trim=300 100 200 75, clip,width=\linewidth]{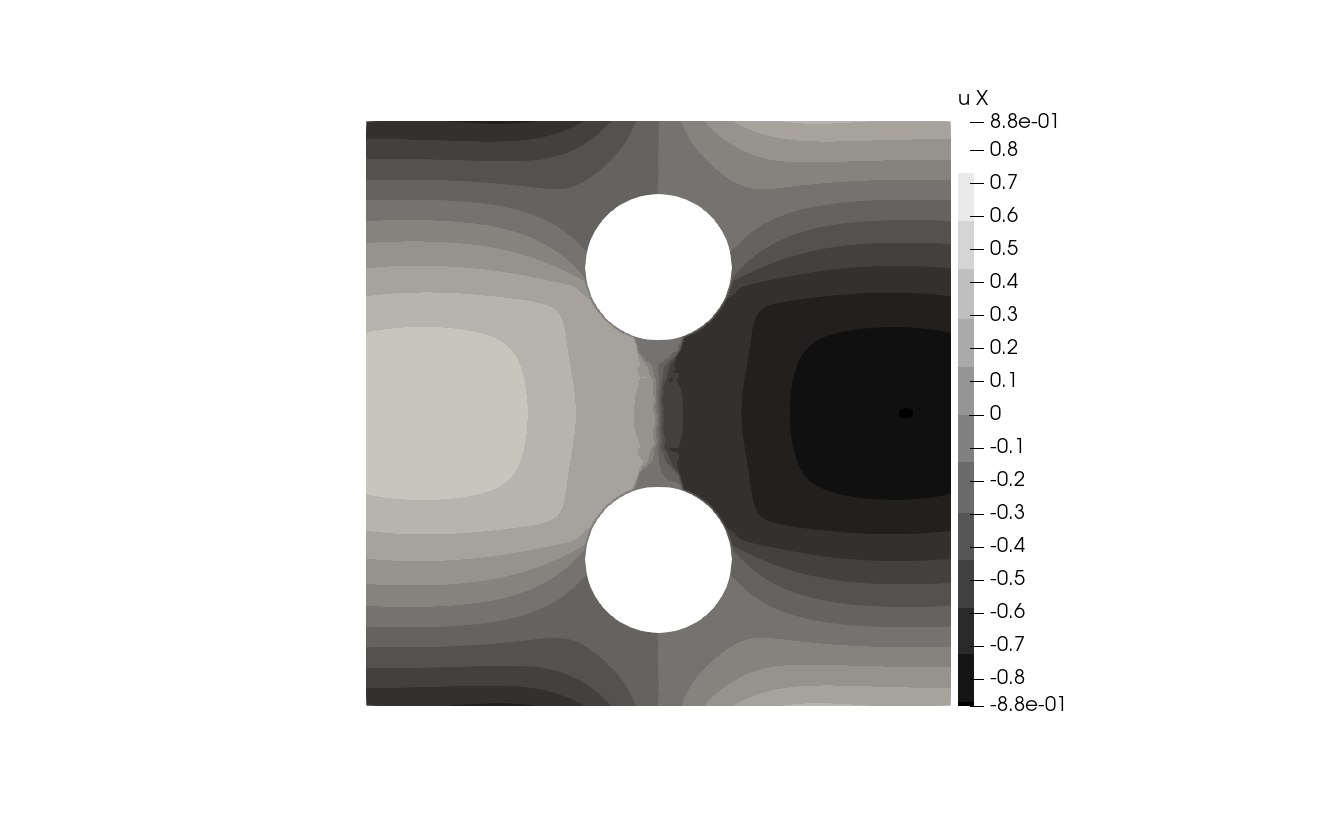}  
  \end{minipage}\hfill
  \begin{minipage}[t]{0.33\linewidth}
   \centering
   \includegraphics[trim=300 100 200 75, clip,width=\linewidth]{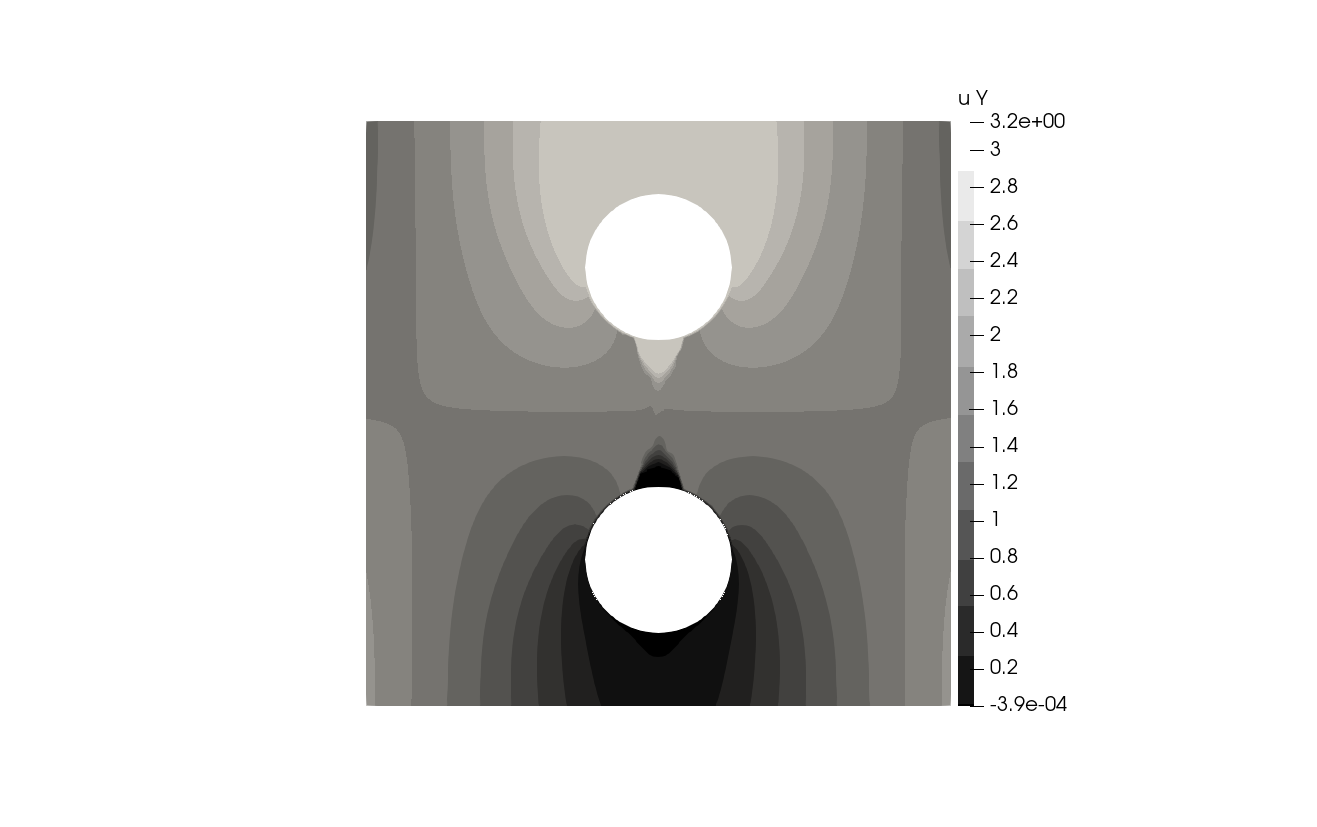}  
  \end{minipage}\hfill
  \begin{minipage}[t]{0.33\linewidth}
   \centering
   \includegraphics[trim=300 100 200 75, clip,width=\linewidth]{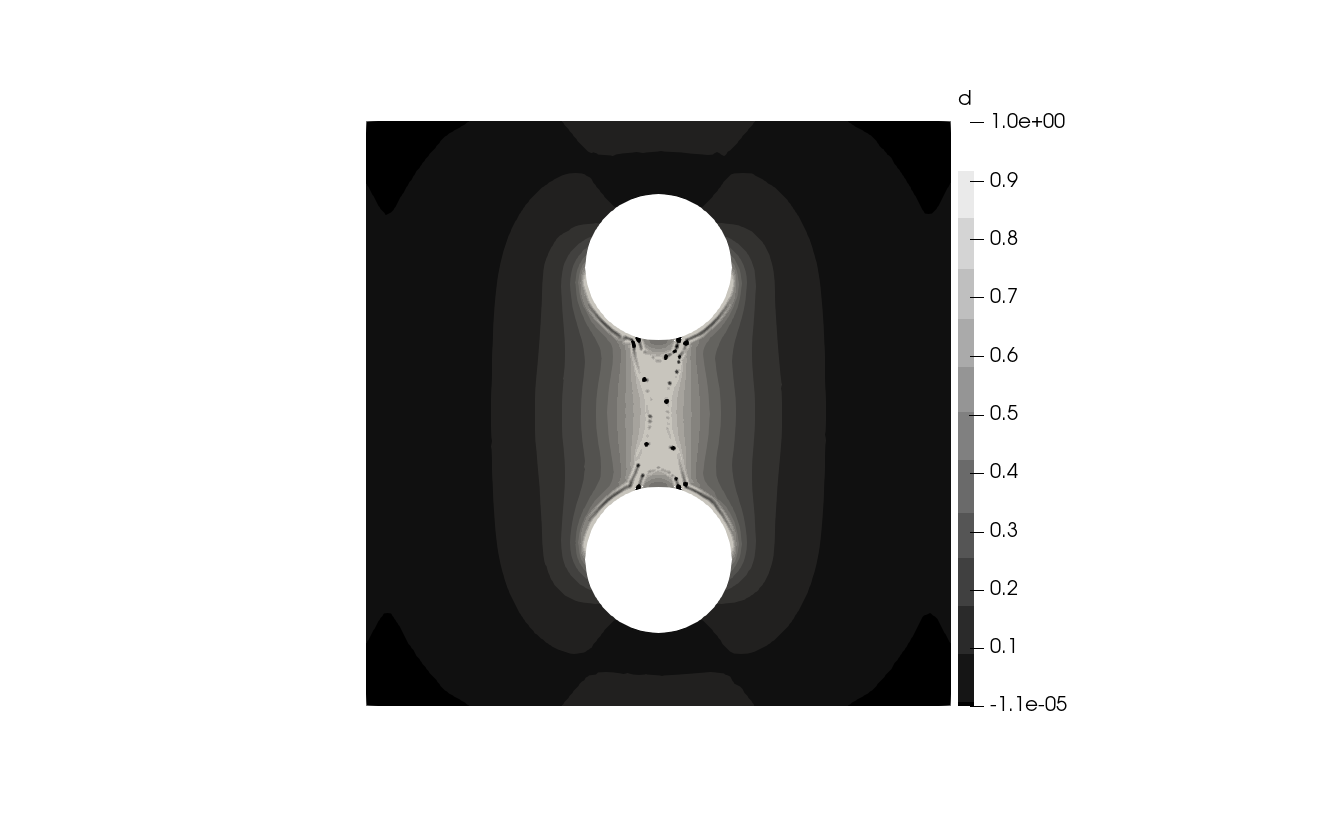}  
  \end{minipage}
  \captionof{figure}{\texttt{TC03S04}. Displacements in $X$- and $Y$- direction, damage after $7.0$ seconds.}
  \label{figure:results TC04S00}
 \end{center}
The presented model is only viable for partial damage.
If substantial damage is reached at some point in space, the coercivity of the left-hand side form no longer holds true, and hence, the equation of motion degenerates.
Looking at the model output exhibited in Figure \ref{figure:results TC04S00}, these effects can be seen in every plot, where those show very different behavior compared to the time steps prior to this (see Figure \ref{figure:results TC03S04}).

\section{Conclusion}
In this framework, we have provided a FEM discretization of the forward problem for the Kachanov damage model and have implemented it in the computational platform FEniCS.
It turns out that our numerical scheme approximates well the target weak solution of the non-linearly coupled elliptic-ode system, whose existence, uniqueness and stability of solutions with respect to parameters has been established in \cite{GM21}.
The numerical simulations recover the expected analytical results and are also within the physical range.\\[2ex]
In view of our project to reconstruct the damage process $\damageprocess\in\damageprocesses$ that governs damage evolution, we stress that Neumann-type boundary conditions should be chosen for the tensile-testing setting since the modelling with Dirichlet-type boundary conditions causes singularities in the displacements' gradient that create heavily localized damage.
In a setting where Dirichlet type boundary conditions cannot be avoided, the possibility to change the computational domain should be investigated.
It would be interesting to see how localized damage will affect the reconstruction process.
We leave this for future work. \\[2ex]
Although boundary conditions of Robin type show great promise, these cannot be applied in our setting from analytical reasons.
To relax this restriction, it might be worth investigating whether a modification of our proof of the nonlinear tangential cone condition is possible.

\section*{Acknowledgements}
The authors are indebted to Michael B\"ohm (Bremen, Germany) for initiating and supporting this research.
\addcontentsline{toc}{section}{References}
\printbibliography

@Book{Br07,
    Title                    = {Finite Elements. Theory, fast Solvers, and Applications in Elasticity Theory. Translated from the German by Larry L. Schumaker},
    Author                   = {Dietrich Braess},
    Publisher                = {Cambridge University Press},
    Year                     = {2007},
    ISBN                     = {978-0-521-70518-9}
}

@Book{PhCi02,
    Author = {Philippe G. {Ciarlet}},
    Title = {{The Finite Element Methods for Elliptic Problems.}},
    FJournal = {{Classics in Applied Mathematics}},
    Journal = {{Classics Appl. Math.}},
    Volume = {40},
    ISBN = {0-89871-514-8},
    Pages = {xxiv + 530},
    Year = {2002},
    Publisher = {Philadelphia, PA: SIAM},
    Language = {English},
    MSC2010 = {65N30 65-02 35J25 35J40 65N15 74K15 74K20 74S05},
}

@Article{DA12,
    Title                    = {A continuum damage mechanics framework for modeling micro-damage healing},
    Author                   = {M. K. Darabi and R. K. {Abu Al-Rub} and D. N. Little},
    Journal                  = {International Journal of Solids and Structures},
    Year                     = {2012},
    Pages ={492--513},
    Volume                   = {49},
}

@Book{DauLi00D,
    Title                    = {Mathematical Analysis and Numerical Methods for Science and Technology. Volume 4: Integral Equations and Numerical Methods},
    Author                   = {Robert {Dautray} and Jacques-Louis {Lions}},
    Publisher                = {Berlin: Springer},
    Year                     = {2000},
    Edition                  = {2nd},
    ISBN                     = {3-540-66100-X},
    Language                 = {English},
    Msc2010                  = {74-01 74S05 74B05 74Kxx 45Exx 76N15 65N30 65R20},
    Pages                    = {x + 494},
    Zbl                      = {0951.74001}
}

@Article{NP98,
    Title                    = {Continuum modelling and numerical simulation of material damage at finite strains},
    Author                   = {E.A. {de Souza Neto} and D. Peric and D.R.J.  Owen},
    Journal                  = {Arch. Computat. Methods Eng.},
    Year                     = {1998},
    Volume                   = {311},
}

@Book{DeBo2002,
    Author = {Peter {Deuflhard} and Folkmar {Bornemann}},
    Title = {{Scientific Computing with Ordinary Differential Equations}},
    FJournal = {{Texts in Applied Mathematics}},
    Journal = {{Texts Appl. Math.}},
    ISSN = {0939-2475},
    Volume = {42},
    ISBN = {0-387-95462-7},
    Pages = {xix + 485},
    Year = {2002},
    Publisher = {New York, NY: Springer},
    Language = {English},
    MSC2010 = {65L05 65L06 65-02 65L50 65L80 34A09 34A34 34E13 65L20},
    Zbl = {1001.65071}
}

@Book{Dz2010,
    Author = {Gerhard {Dziuk}},
    Title = {{Theorie und Numerik partieller Differentialgleichungen}},
    FJournal = {{De Gruyter Studium}},
    Journal = {{De Gruyter Stud.}},
    ISBN = {978-3-11-021481-9},
    Pages = {ix + 319},
    Year = {2010},
    Publisher = {Berlin: Walter de Gruyter},
    Language = {German},
    MSC2010 = {35-01 65-01 65Mxx 65Nxx},
    Zbl = {1203.35001}
}

@Article{GeGr20,
    author    = {Thies Gerken and Simon Grützner},
    title     = {Dynamic inverse wave problems{\textemdash}part I: regularity for the direct problem},
    journal   = {Inverse Problems},
    year      = {2020},
    volume    = {36},
    number    = {2},
    pages     = {024004},
    month     = {1},
    doi       = {10.1088/1361-6420/ab47ec},
    publisher = {{IOP} Publishing},
    url       = {https://doi.org/10.1088%2F1361-6420%2Fab47ec},
}

@inproceedings{GM18,
    author    = "S. Gr{\"u}tzner and  A. Muntean",
    title     = "Brief introduction to damage mechanics and a note on its relation to deformations",
    booktitle = "Mathematical Analysis of Continuum Mechanics and Industrial Applications II",
    editor ="P. {van Meurs} and M. Kimura and H. Notsu",
    series    = "Mathematics for Industry",
    year      = 2018,
    pages     = "115--124",
    publisher = "Springer"
}

@Article{GM21,
    Title                    = {Identifying processes governing damage evolution in quasi-static elasticity. Part 1- Analysis},
    Author                   = {S. Gr{\"u}tzner and  A. Muntean},
    Journal                  = {Advances in Mathematical Sciences and Applications (AMSA)},
    Year                     = {2021},
    Pages                    = {305--334},
    Volume                   = {30},
}

@Unpublished{SiGr15,
    Title                    = {\emph{An approach to parameter identification in damaged continua}},
    Author                   = {Simon Gr\"utzner},
    Note                     = {Diploma Thesis, University of Bremen},
    Year                     = {2015}
}

@Article{JeLe84,
    author    = {Jean {Lemaître}},
    title     = {How to use damage Mechanics},
    journal   = {Nuclear Engineering and Design},
    year      = {1984},
    volume    = {80},
    pages     = {233-245},
}

@Book{Ka86,
    Title                    = {{Introduction to Continuum Damage Mechanics}},
    Author                   = {Lazaŕ Markovich {Kachanov}},
    Publisher                = {Dordrecht: Springer},
    Year                     = {1986},
    Edition                  = {1st},
    ISBN                     = {978-90-481-8296-1},
    Language                 = {English},
    Msc2010                  = {74R99 74-01 74C99 74D99},
    Zbl                      = {0596.73091}
}

@Article{LP04,
    Title                    = {Modeling damage growth and failure in elastic materials with random defect distributions},
    Author                   = {A.B. Lennon and P.J. Prendergast},
    Journal                  = {Mathematical Proceedings of the Royal Irish Academy},
    Year                     = {2004},
    Pages ={155--171},
    Volume                   = {104A},
}

@Article{LW21,
    Title                    = {Continuum damage mechanics approach for modeling cumulative-damage model},
    Author                   = {H. Li and J. Wang and J. Wang and M. Hu and Y. Peng},
    Journal                  = {Mathematical Problems in Engineering},
    Year                     = {2021},
    Pages                    ={Article ID 7136846},
    Volume                   = {2021},
}

@Article{LiLi05,
    Title                    = {A Review on Damage Mechanisms, Models and Calibration Methods under Various Deformation Conditions},
    Author                   = {J. Lin and Y. Liu and T.~A. Dean},
    Journal                  = {Int. J. Damage Mech.},
    Year                     = {2005},
    Pages                    = {299--319},
    Volume                   = {14},
    Fjournal                 = {International Journal of Damage Mechanics},
    Language                 = {English},
    Publisher                = {Sage Publishing}
}

@Book{LoMa2012a,
    Title                    = {Automated Solution of Differential Equations by the Finite Element Method},
    Author                   = {Anders Logg and Kent-Andre Mardal and Garth N. Wells and others},
    Publisher                = {Springer},
    Year                     = {2012},
    Doi                      = {10.1007/978-3-642-23099-8},
    ISBN                     = {978-3-642-23098-1}
}

@Article{OKM19,
    Title                    = {On the implementation of finite deformation gradient-enhanced damage models},
    Author                   = {R. Ostwald and E. Kuhl and A. Menzel},
    Journal                  = {Computational Mechanics},
    Year                     = {2019},
    Pages                    = {847--877},
    Volume                   = {64},
}

@Article{HL16,
  Title                    = {Adaptive coupling between damage mechanics and peridynamics: {A} route for objective simulation of material degradation up to complete failure},
  Author                   = { F. Han and G. Lubineau and Y. Azdoud},
  Journal                  = {Journal of the Mechanics and Physics of Solids},
  Year                     = {94},
  Pages ={437--472},
  Volume                   = {2016},
  }
\label{page:e}
\end{document}